\documentclass[final,leqno]{siamltex}

\usepackage{setspace,url}
\usepackage[left=1in,top=1in,right=1in,bottom=1in,dvips,letterpaper]{geometry}
\usepackage[algo2e,linesnumbered,vlined,ruled]{algorithm2e}
\usepackage{graphics,graphicx,epsf,subfigure,epstopdf}
\usepackage{comment}
\usepackage{times,hyperref}
\usepackage{amsmath,amsxtra,amsfonts,amscd,amssymb,bm}
\usepackage{supertabular,multirow}
\usepackage{booktabs}
\usepackage[usenames]{color}
\usepackage[normalem]{ulem}
\usepackage{exscale}
\hypersetup{colorlinks=true,linkcolor=blue,citecolor=blue}

\renewcommand{\theenumi}{\roman{enumi}}

\newtheorem{remark}[theorem]{Remark}
\newtheorem{assumption}[theorem]{Assumption}

\newcommand{\R}{\mathbb{R}}
\newcommand{\Rn}{\mathbb{R}^{n}}  
\newcommand{\Rnp}{\mathbb{R}^{n\times p}} 
 
\newcommand{\Rnn}{\mathbb{R}^{n\times n}}

\newcommand{\Sn}{\mathbb{S}^{n}}
\newcommand{\Sp}{\mathbb{S}^{p}}

\newcommand{\cL}{\mathcal{L}}
\newcommand{\cLb}{\mathcal{L}_{\beta}}

\newcommand{\TX}{\mathcal{T}(X)}

\newcommand{\us}{\underline{\sigma}}
\newcommand{\cC}{\mathcal{C}}

\newcommand{\stiefel}{{\cal S}_{n,p}}

\newcommand{\tr}{\mathrm{tr}}
\newcommand{\Diag}{\mathrm{Diag}}

\newcommand{\zz }{^{\top}}
\newcommand{\half }{\frac{1}{2}}
\newcommand{\inv}{^{-1}}
\newcommand{\st}{\mathrm{s.\,t.}\,\,} 
\newcommand{\ff}{_{\mathrm{F}}}
\newcommand{\fs}{^2_{\mathrm{F}}}

\newcommand{\dkh}[1]{\left(#1\right)}

\newcommand{\norm}[1]{\left\|#1\right\|}

\newcommand{\rand}[2]{\mathbf{rand}(#1,#2)}

\newcommand{\abs}[1]{\left|#1\right|}
\newcommand{\jkh}[1]{\left\langle#1\right\rangle}
\newcommand{\hess}[1]{\nabla^2 f(#1)}

\newcommand{\comm}[1]{{\color{black}#1}}
\newcommand{\revise}[1]{{\color{black}#1}}

\allowdisplaybreaks
\graphicspath{ {./results/} }
\definecolor{Gray}{rgb}{0.5,0.5,0.5}
\DeclareMathOperator*{\argmin}{arg\,min}

\begin{document}

\title{Parallelizable Algorithms for Optimization Problems with Orthogonality Constraints}

\author{Bin Gao\thanks{State Key Laboratory of Scientific and Engineering 
		Computing, Academy of Mathematics and Systems Science, Chinese  Academy of Sciences, and University of Chinese Academy of Sciences, China (gaobin@lsec.cc.ac.cn)}
		\and Xin Liu\thanks{State Key Laboratory of Scientific and Engineering 
		Computing, Academy of Mathematics and Systems Science, Chinese Academy of Sciences, 
		and University of Chinese Academy of Sciences, China (liuxin@lsec.cc.ac.cn). Research supported in part by NSFC grants 11622112, 11471325, 91530204 and 11688101, the National Center for Mathematics and Interdisciplinary Sciences, CAS, and Key Research Program of Frontier Sciences QYZDJ-SSW-SYS010, CAS.}
		 \and Ya-xiang Yuan\thanks{State Key Laboratory of Scientific and 
		Engineering Computing, Academy of Mathematics and Systems Science, Chinese Academy of Sciences, China (yyx@lsec.cc.ac.cn). Research supported in part by NSFC grant 11331012 and 11461161005.}
}

\maketitle

\begin{abstract}
	To construct a parallel approach for solving optimization problems with orthogonality constraints is 
	usually regarded as an extremely difficult mission, due to the low scalability of the 
	orthonormalization procedure. 
	However, such demand is particularly huge in some application areas such as materials computation.
	In this paper, we propose a proximal linearized augmented Lagrangian algorithm (PLAM)
	for solving optimization problems with orthogonality constraints. 
	Unlike the classical augmented Lagrangian methods, 
	in our algorithm, the prime variables are updated by minimizing a proximal 
	linearized approximation of the augmented Lagrangian function, 
	meanwhile the 
	dual variables are updated by a closed-form expression which holds at any first-order 
	stationary point.
	The orthonormalization procedure is only invoked once 
	at the last step of the above mentioned algorithm if high-precision feasibility is needed.
	Consequently, the main parts of the proposed algorithm can be parallelized naturally. 
	We establish global subsequence convergence, worst-case complexity 
	and local convergence rate for PLAM under some mild assumptions. 
	To reduce the sensitivity of the penalty parameter, we put forward a modification
	of PLAM, which is called parallelizable column-wise block minimization of PLAM (PCAL).
	Numerical experiments in serial
	illustrate that 
	the novel updating rule for the Lagrangian multipliers significantly accelerates 
	the convergence of PLAM and makes it comparable with the existent feasible solvers for 
	optimization problems with orthogonality constraints, and the performance of PCAL
	does not highly rely on the choice of the penalty parameter.
	Numerical experiments under parallel environment demonstrate that
	PCAL attains good performance and high scalability in solving discretized Kohn-Sham total energy 
	minimization problems.
\end{abstract}



\section{Introduction}

In this paper, we consider the following matrix variable optimization problem
with orthogonality constraints.
\begin{eqnarray}\label{prob}
\left.
\begin{array}{rl}
\min\limits_{X\in\Rnp} &	f(X)\\
\st & X\zz X=I_p,
\end{array}
\right.
\end{eqnarray}
where $I_p$ is the $p$-by-$p$ identity matrix with $2p\leq n$, and $f : \Rnp \longmapsto \R $
is a continuously differentiable function. 
The feasible set of the orthogonality constraints is also known as Stiefel manifold,
\revise{$\stiefel=\{X\in \Rnp|~X\zz X = I_p\}$}. 

Throughout this paper, we assume
\begin{assumption}\label{a1}
	{\bf (Blanket Assumption)}
$f$ is continuously differentiable. 
\end{assumption}

The twice differentiability of $f$ will be particularly mentioned once it is required in some 
theoretical analyses.

\subsection{Literature Survey}

Kohn-Sham density functional theory (KSDFT) is known to 
be an important topic in materials science \cite{kohn1965self}. The last step
of KSDFT is to minimize a discretized 
Kohn-Sham total energy function 
\begin{eqnarray}\label{eq:KS-energy} 
E(X):=\frac{1}{4}\tr(X\zz L X) + \half\tr(X\zz V_{ion}  X)
+ \frac{1}{4} \rho^\top
L^\dagger \rho + \half \rho\zz \epsilon_{xc}(\rho), 
\end{eqnarray}
subject to orthogonality constraints.
Here $\rho(X):=\diag(XX\zz)$ denotes the charge density, and $L\in\R^{n\times n}$ is a 
finite-dimensional representation of the Laplace operator in the planewave basis. 
The discretized local ionic potential can be represented by a diagonal matrix $V_{\mathrm{ion}}$. 
And the matrix $L^\dag$ which is the discrete form of the Hartree potential corresponds 
to the pseudo-inverse of $L$. The exchange correlation function $\epsilon_{\mathrm{xc}}$ 
is used to model the non-classical and quantum interaction between electrons. 
The discretized energy minimization is exactly a special case of \eqref{prob}.
The variable scale of such problems is often very large,
and hence the demand on efficient solvers for optimization with orthogonality constraints
is high.

In recent decades, 
lots of researchers 
have proposed quite a few 
efficient optimization approaches for 
discretized Kohn-Sham total energy minimization
 \cite{Yang06,Yang2007,Yang09,YangGao2009,WenUlbrich13,Wen_Zhou14,Ulbrich_Wen2015,Dai_Zhou17,LiWen17}. 
For the general purpose of solving optimization problems with orthogonality constrains, 
there are abundant algorithms. Retraction based approaches
\cite{Edelman98,Nishimori05,Abisil2008,WenYin2013,JiangDai15},
splitting algorithm \cite{LaiOsher14}, multipliers correction framework \cite{Gao2016},
just to mention a few. Interested readers are referred to the references 
in \cite{Gao2016}. There are a few successful solvers.
The most famous one is the toolbox for optimization 
on manifolds, which is called Manopt\footnote{Available from 
	\href{http://www.manopt.org}{http://www.manopt.org}}, in which 
lots of retraction based algorithms for Problem \eqref{prob}, such as MOptQR, a QR projection
algorithm, are included.
Another quasi-geodesic based approach called OptM\footnote{Available from \href{http://optman.blogs.rice.edu}{http://optman.blogs.rice.edu}} is widely used in the area of discretized Kohn-Sham energy minimization.

However, the lack of concurrency becomes a major 
bottleneck of solving optimization problems with orthogonality constraints, 
particularly, when the number of columns of the variable matrix is large.
Unfortunately, parallel computation does not attract much attention from optimization area 
until very recently.
Refer to \cite{Hogwild11,Boyd11,PengYin13,Wright15,PengYin16}, 
there is \comm{an} urgent demand of parallelization in optimization area.
Although high scalability algorithms are desired by KSDFT area for decades, 
there is no successful attempt in this regard so far \cite{Dai_Zhou17}.

We find that parallelization is particular difficult for 
optimization problems with orthogonality constraints.
The main reason is that the scalability of  \comm{orthonormalization calculations} is low no matter which 
particular way how you do it.

\subsection{Contribution}

In this paper, we propose an infeasible algorithm for optimization problems with 
orthogonality constraints. It is based on \revise{the} augmented Lagrangian method but \revise{employing} totally different 
updating scheme for both prime and dual variables. 
The main motivation of the so-called proximal linearized augmented Lagrangian method
(PLAM)
is an observation that the dual variables enjoy a closed-form formula at each 
first-order stationary point. Therefore, we consider to use the symmetrization of this formula
as the updating rule for the dual variables, 
to take the place of dual ascent step in the classical augmented Lagrangian method.
For the prime variables, instead of solving the augmented Lagrangian subproblem to some preset precision,
we minimize a proximal linearized approximation
of the augmented Lagrangian function, which is equivalent to take one step gradient descent.

The orthonormalization procedures are waived in all iterations except the last
one to guarantee high-precision feasibility. The cost of waiving orthonormalization
is to do more BLAS3 calculations (matrix-matrix multiplication) which are known to have
high scalability. 

We show the global convergence, worst-case complexity
and local Q-linear convergence rate for PLAM under some mild assumptions.
The global convergence of PLAM requires sufficiently large penalty parameter
and correspondingly small stepsize. Numerical tests also verify the sensitivity 
of the penalty parameter. Consequently, we put forward a novel modification strategy
that is to add 
redundant unit norm constraints to the proximal linearized augmented Lagrangian 
subproblem for updating the prime variables. By using this strategy, we can 
restrict the iterates
in a compact set such that the penalty parameter is no longer required to be large.
On the other hand, such modification does not destroy the structure
that the subproblem has a closed-form solution which can be calculated in
parallel. We call the consequent algorithm PCAL, namely, parallelizable column-wise block minimization
for PLAM. The boundedness of PCAL iterates can be guaranteed automatically, and hence 
the penalty parameter is no longer required to be sufficiently large. 

The numerical experiments under serial computing
demonstrate the way how to choose default settings
for our algorithms, and show that the infeasible algorithms
are at least as efficient as the existent feasible algorithms
in solving a bunch of test problems. The numerical 
experiments under parallel computing illustrate the computational complexity
of PCAL and expose its high scalability.

\subsection{Organization and Notations}

The motivation of new approaches will be introduced in the next section.
In Section 3, we will present the algorithm frameworks.
We investigated the theoretical behaviors of the new proposed algorithms in Section 4.
Numerical experiments will be demonstrated in Section 5.
In the last section, we will draw a brief conclusion and discuss possible future works.

Notations.
$\Sp:=\{X\in\R^{p\times p}\mid X\zz = X\}$
refers to the $p$-by-$p$ real symmetric matrices set.
$\lambda_{\max}(A)$ and $\lambda_{\min}(A)$
stand for the largest and smallest eigenvalues of given symmetric real matrix $A$,
respectively. $\sigma_{\max}(A)$ and $\sigma_{\min}(A)$
denote the largest and smallest singular values of given real matrix $A$,
respectively. $X^\dagger :=(X\zz X)\inv X\zz$
refers to the pseudo inverse of $X$. 
$\Diag(v)\in\Sn$ denotes a diagonal matrix with all entries of $v\in\Rn$
in its diagonal, and $\diag(A)\in\Rn$ extracts the diagonal entries
of matrix $A\in\Rnn$. For convenience,
$\Phi(M):=\Diag(\diag(M))$ represents the diagonal matrix with the diagonal entries
of square matrix $M$ in its diagonal. $\varPsi(A):=\frac{1}{2}(A+A\zz)$ stands for the 
average of a square matrix and its transpose.

\section{Motivation}

As mentioned in the previous section, almost all 
the \revise{existing} practically useful methods 
require feasible iterates all the time.
To realize feasibility, 
either explicit or implicit orthonormalization requires to be 
invoked. Such kind of calculation lacks of scalability and hence becomes the bottleneck 
computation in the corresponding algorithms.
For example, we consider the discretized 
Kohn-Sham total energy minimization \eqref{eq:KS-energy}. 
In each iteration, the function value and first-order derivative evaluation cost
$O(n\log n+np)$ or $O(np)$ flops per iteration, depending on whether plane wave or finite difference,
respectively, is used in the discretization scheme.
For the main iteration of any algorithm for solving \eqref{eq:KS-energy} developed in recent decade,
the computational cost per iteration is $O(np^2)$ for BLAS3 calculation, plus
$O(p^3)$ for orthonormalization which can hardly be parallelized. 

To break through this bottleneck, we suggest to
use infeasible methods to take the place of feasible methods.

There is no existent infeasible approach for general purpose reported to be efficient
for optimization problems with orthogonality constraints.
\comm{Previous} infeasible approaches designed for \eqref{prob} 
either work specially for Rayleigh-Ritz trace minimization \cite{SLRP15,EigPen16}, 
or adopt ADMM framework after introducing auxiliary variables to split the objective 
and orthogonality  constraints \cite{LaiOsher14}. The previous ones can hardly be extended
to general objective, while the latter ones does not have good performance in general.

In the following subsections, 
we introduce how we come up with a new idea on constructing an efficient infeasible algorithm
for problem \eqref{prob}.

\subsection{The Optimality Condition}

We start from the optimality condition of the optimization problem
with orthogonality constraints \eqref{prob}. 
The first-order optimality condition of problem \eqref{prob} can be written as
\begin{eqnarray}\label{eq:kkt}
\left\{
\begin{array}{c}
\nabla f(X) = X\Lambda;\\
X\zz X = I_p,
\end{array}
\right.
\end{eqnarray}
where $\Lambda\in \mathbb{S}^{p}$ consists of the Lagrangian multipliers of 
the orthogonality constraints. 
Condition \eqref{eq:kkt} has the following
equivalent form where $\Lambda$ is eliminated.
\begin{eqnarray}\label{eq:kkt2}
\left\{
\begin{array}{c}
\nabla f(X) - X\nabla f(X)\zz X =0;\\
X\zz X = I_p.
\end{array}
\right.
\end{eqnarray}

\begin{definition}
We call $X$ a first-order stationary point, 
if condition \eqref{eq:kkt2} holds.
We call $X$ a second-order stationary
point, if it is a first-order stationary point and satisfies
\begin{eqnarray}\label{eq:son}
\tr(Y\zz \nabla^2 f(X) [Y] - \Lambda Y\zz Y)\geq 0,\quad \forall\, Y\in\TX,
\end{eqnarray}
where $\TX:=\{Y\mid Y\zz X+X\zz Y=0 \}$ is the tangent space of the orthogonality constraints at $X$.
\end{definition}

The following proposition can be easily verified and hence its proof is omitted here.
\begin{proposition}
If $X$ is a local minimizer of \eqref{prob}, it has to be a second-order stationary point.
If $X$ is a strict local minimizer\footnote{$X$ is called a strict local minimizer,
if \comm{$X\in\stiefel$} and there exists $\delta>0$ such that $f(X)<f(Y)$ holds for any \comm{$Y\in U_\delta(X) :=\{Y\in\stiefel\mid ||X-Y||\in(0,\delta) \}$}.},
if and only if $X$ is a first-order stationary point and satisfies
\begin{eqnarray}\label{eq:local}
\tr(Y\zz \nabla^2 f(X) [Y] - \Lambda Y\zz Y)> 0,\quad \forall\, 0\neq Y\in\TX.
\end{eqnarray}
\end{proposition}

\subsection{Augmented Lagrangian Method}

A straightforward idea to solve \eqref{prob} without requiring feasibility in 
each iteration is to employ the \revise{Augmented Lagrangian Method (ALM)
\cite{powell1969method,Jorge06,Bertsekas14}}, which is described
in Algorithm \ref{alg:ALM}.
\begin{algorithm2e}[ht]
	\caption{Augmented Lagrangian Method (ALM)}
	\label{alg:ALM}
	\SetKwInOut{Input}{input}\SetKwInOut{Output}{output}
	\SetKwComment{Comment}{}{}
	\BlankLine 
	\textbf{Input:} choose initial guess $\Lambda^0$ for the dual variables, and set $k:=0$\;
	\While{certain stopping criterion is not reached}
	{
		Minimize the augmented Lagrangian function with respect to the prime variables $X$:
		\begin{eqnarray*}
			X^{k+1}:=\min\limits_{X\in\Rnp}\quad {\cL}_\beta(X,\Lambda^k),
		\end{eqnarray*}
		where the augmented Lagrangian function fo problem \eqref{prob} is defined as
		\begin{eqnarray}\label{eq:Lag}
		\cLb(X,\Lambda) &=& f(X) - \frac{1}{2}\langle\Lambda, 
		X\zz X-I_p\rangle + \frac{\beta}{4} ||X\zz X-I_p||\fs\nonumber\\
		&=& f(X) + \frac{\beta}{4} \left\|X\zz X- \left( I_p + \frac{1}{\beta}\Lambda\right)\right\|\fs
		- \frac{1}{4\beta} ||\Lambda||\fs.
		\end{eqnarray}
		
		Update the Lagrangian multipliers
		\begin{eqnarray}\label{eq:Lambda-alm}
		\Lambda^{k+1} := \Lambda^k - \beta ({X^{k+1}}\zz X^{k+1} -I_p).
		\end{eqnarray}
		
		Update the penalty parameter $\beta$ if necessary.
		Set $k:=k+1$.
	}
	\textbf{Output:} $X^k$.
\end{algorithm2e}

It is well-known that the augmented Lagrangian function \eqref{eq:Lag} 
is an exact penalty if the parameter $\beta$ is sufficiently large.
Algorithm \ref{alg:ALM} works very well for problem with linear constraints.
For optimization problems with nonlinear constraints, it is not clear
how to choose the parameter $\beta$ in practice, which is very sensitive to 
the numerical performance.

The purpose of this work is to find an infeasible algorithm for solving \eqref{prob}
at similar cost of the existent feasible methods. Otherwise, we can hardly gain
much from the parallelization. To this end, we carefully test Algorithm \ref{alg:ALM}
and try our best to tune the parameter $\beta$. Unfortunately, 
for solving optimization problems with orthogonality constraints \eqref{prob},
the efficiency of classical ALM is far from being satisfactory.  

Therefore, we need to employ a new idea to remould the classical ALM.
According to the conditions \eqref{eq:kkt} and \eqref{eq:kkt2}, 
it is not difficult to verify that the Lagrangian multipliers $\Lambda$
have the following closed-form expression at any first-order stationary point,
\begin{eqnarray}\label{eq:closed-multi}
\Lambda = \nabla f(X)\zz X.
\end{eqnarray}
A straightforward idea is to use the following symmetrized
form of \eqref{eq:closed-multi}
\begin{eqnarray}\label{eq:multi}
\Lambda =\varPsi(\nabla f(X)\zz X)
\end{eqnarray}
as a new multipliers updating rule.
The symmetrization is necessary because the symmetry of the expression $\nabla f(X)\zz X$
can not be guaranteed in each iteration.

As we will demonstrate in the following lemma and the theoretical analyses in Section 4,
an explicit lower bound of the penalty parameter $\beta$ can be estimated
if updating rule \eqref{eq:multi} is applied. Hence, the update of the penalty 
parameter $\beta$ can be waived. Moreover, the numerical experiments verify the validation
of this new updating rule.
	
\begin{lemma}\label{lm:1}
	Let $X^*$ be a second-order stationary point of 
	\begin{eqnarray}\label{eq:Lmin}
	\min\limits_{X\in\Rnp} \cLb(X,\Lambda^*)
	\end{eqnarray}
	with $\Lambda^* = \varPsi(\nabla f(X^*)\zz X^*)$.
	Suppose $\beta>\lambda_{\max}(\hess{X^*})$.
	Then $X^*$ is a second-order stationary point of problem \eqref{prob}. 
	Namely, optimality conditions \eqref{eq:kkt} and \eqref{eq:kkt2} hold at $X^*$.
\end{lemma}
\begin{proof}
	For convenience, we abuse the notation slightly by deleting the superscript $*$ from $X^*$.
	First, we have
	\begin{eqnarray}\label{eq:S1}
	\nabla_X \cLb(X,\Lambda) &=& \nabla f(X) +\beta X\left( X\zz X -\left(I_p+\frac{1}{\beta}\Lambda\right)\right);\\
	\label{eq:S2}
	\nabla_{XX}^2 \cLb(X,\Lambda)[S] &=& \nabla^2 f(X)[S] +\beta S\left( X\zz X -\left(I_p+\frac{1}{\beta}\Lambda\right)\right) +\beta X(S\zz X+X\zz S).
	\end{eqnarray}
	
	Since $X$ is the second-order stationary point of \eqref{eq:Lmin} 
	with $\Lambda = \nabla f(X)\zz X$, we have
	\begin{eqnarray}\label{eq:O1}
	\nabla \cLb(X,\Lambda) &=& 0;\\
	\label{eq:O2}
	\langle S, \nabla_{XX}^2 \cLb(X,\Lambda)[S] \rangle &\geq& 0,\quad \forall S\neq 0.
	\end{eqnarray}
	Substituting \eqref{eq:S1} into \eqref{eq:O1}, we obtain
	\begin{eqnarray}\label{eq:temp1}
	\nabla f(X) - X \Lambda - \beta X(I_p- X\zz X) =0.
	\end{eqnarray}
	Left multiplying $X\zz$ into both sides of \eqref{eq:temp1}, we have
	\begin{eqnarray}\label{eq:temp2}
	X\zz \nabla f(X) = X\zz X \Lambda + \beta X\zz X(I_p- X\zz X).
	\end{eqnarray}
	
	Suppose $X=U\Sigma V\zz$ is the singular value decomposition of $X$ in economy-size,
	which implies $X\zz X =V\Sigma^2 V\zz$. Then, we further have
	\begin{eqnarray*}
		X\zz \nabla f(X) -\beta V\Sigma^2 V\zz = V\Sigma^2 V\zz \Lambda - \beta V\Sigma^4 V\zz.
	\end{eqnarray*}
	Left multiplying $V\zz$ and right multiplying $V$ to both sides of the above equality, we arrive at
	\begin{eqnarray*}
		V\zz X\zz \nabla f(X) V - \beta \Sigma^2 = \Sigma^2 (V\zz \Lambda V -\beta \Sigma^2).
	\end{eqnarray*}
	Taking the $\Phi$ operator and using the fact that 
	\begin{eqnarray}\label{eq:S4}
		\diag(V\zz X\zz \nabla f(X) V)
		= \diag(V\zz \nabla f(X)\zz X V)
		= \diag(V\zz \Lambda V),
	\end{eqnarray}
	we have
	\begin{eqnarray}\label{eq:cl}
		(I_p-\Sigma^2)(\Phi(V\zz \Lambda V) -\beta\Sigma^2)=0,
	\end{eqnarray}
	which implies that 
	\begin{eqnarray}\label{eq:zero}
	D(\Phi(V\zz \Lambda V) -\beta\Sigma^2)=0,
	\end{eqnarray}
	where $p$-by-$p$ diagonal matrix $D$ satisfies
	$$D_{ii}=\left\{
	\begin{array}{cc}
	0, &\mbox{if\,}(I_p-\Sigma^2)_{ii}=0;\\
	1, &\mbox{otherwise,}
	\end{array}
	\right.\quad \forall i=1,...,p.$$
	
	On the other hand, since $n\geq 2p$, there exists $\tilde{U}\in \stiefel$ satisfying $\tilde{U}\zz U=0$. let $S=\tilde{U}DV\zz$. If $S\neq 0$, we substitute \comm{$S$} into \eqref{eq:S2} and obtain
	\begin{eqnarray*}
		\langle S, \nabla_{XX}^2 \cLb(X,\Lambda)[S] \rangle &=& \tr(S\zz \nabla^2 f(X) [S]) -\beta \tr(S\zz S) -\tr(S\zz S (\Lambda -\beta X\zz X))\\
		&=& \tr\left(S\zz (\nabla^2 f(X)-\beta I)[S]\right) -\tr(V\zz S\zz SVV\zz (\Lambda -\beta 
		V\Sigma^2 V\zz)V)\\
		&=& \tr\left(S\zz (\nabla^2 f(X)-\beta I)[S]\right) -\tr(D^2 (V\zz \Lambda V -\beta \Sigma^2))\\
		&=& \tr\left(S\zz (\nabla^2 f(X)-\beta I)[S]\right) -\tr(D^2 (\Phi(V\zz \Lambda V) -\beta \Sigma^2)).
	\end{eqnarray*}
	Here $I$ stands for the identity mapping from $\Rnp$ to $\Rnp$.
	Combining with the second-order optimality condition \eqref{eq:O2}, relationship \eqref{eq:zero}
	and the assumption on $\beta$, we have
	\begin{eqnarray}\label{eq:main}
	0\leq \langle S, \nabla_{XX}^2 \cLb(X,\Lambda)[S] \rangle =\tr\left(S\zz (\nabla^2 f(X)-\beta I)[S]\right) 
	<0,
	\end{eqnarray}
	which leads to contradiction. Hence, $S=0$, which immediately implies that $\Sigma =I_p$.
	Therefore, we have $X\in\stiefel$. Together with \eqref{eq:O1} and \eqref{eq:O2}, 
	we can easily show that the optimality condition
	\eqref{eq:kkt} and \eqref{eq:kkt2} hold. This completes the proof.
\end{proof}

Lemma \ref{lm:1} guarantees that \revise{the} augmented Lagrangian function is still an exact penalty 
function with the Lagrangian multipliers updated by explicit formula \eqref{eq:multi}.
However, to achieve 
the convergence results for first-order methods, we 
need a first-order version of Lemma \ref{lm:1}. Moreover, to
obtain the global convergence rate, 
the feasibility should be controlled by the first-order
optimality violation.
 
\begin{lemma}\label{lm:2}
	For any $X^*$ satisfying
	$\sigma_{\min}(X^*)>0$\revise{,} suppose $\beta>\left(||\nabla f(X^*)||_2\cdot ||X^*||_2+\delta\right)/\sigma^2_{\min}(X^*)$ with $\delta>0$.
	Then it holds
	\begin{eqnarray*}
			||{X^*}\zz X^* -I_p||\ff \leq \frac{||X^*||_2}{\delta}\cdot||\nabla_X \cLb(X^*,\Lambda^*)||\ff,
	\end{eqnarray*}
	with $\Lambda^* = \varPsi(\nabla f(X^*)\zz X^*)$.
	In particular, if it happens that $X^*$ is a first-order stationary point of 
	\begin{eqnarray*} 
		\min\limits_{X\in\Rnp} \cLb(X,\Lambda^*)
	\end{eqnarray*}	
	with $\Lambda^* = \varPsi(\nabla f(X^*)\zz X^*)$,
	then $X^*$ is \revise{also} a first-order stationary point of problem \eqref{prob}. 
\end{lemma}
\begin{proof}
	For brevity, we denote $G=\nabla_X \cLb(X,\Lambda)$. Left multiplying $X\zz$
	into both sides of \eqref{eq:S1} and using 
	the singular value decomposition $X=U\Sigma V\zz$, we have
	\begin{eqnarray*}
		X\zz G = X\zz \nabla f(X) - \beta V\Sigma^2 V\zz -
		V\Sigma^2 V\zz \Lambda + \beta V \Sigma^4 V\zz.
	\end{eqnarray*}
	Left multiplying $V\zz$ and right multiplying $V$ to both sides of the above equality,
	we obtain
	\begin{eqnarray*}
		V\zz X\zz G V = V\zz X\zz \nabla f(X)V
		- \beta \Sigma^2 - \Sigma^2 \left( V\zz \Lambda -\beta \Sigma^2\right).
	\end{eqnarray*}
	Taking the $\Phi$ operator and using the fact \eqref{eq:S4}, we arrive at
	\begin{eqnarray}\label{eq:S3}
		\Phi(V\zz X\zz G V) = (I_p-\Sigma^2)(\Phi(V\zz \Lambda V) -\beta\Sigma^2).
	\end{eqnarray}
	\revise{Since} $\beta>\left(||\nabla f(X^*)||\ff\cdot ||X^*||_2+\delta\right)/\sigma^2_{\min}(X^*)$, 
	we have
	\begin{eqnarray*}
		\beta \sigma^2_{\min}(X^*) \geq ||\nabla f(X^*)||_2 \cdot ||X^*||_2+\delta,
	\end{eqnarray*}
	which implies
	\begin{eqnarray*}
		\sigma_{\min}(\beta \Sigma^2)  \geq ||V\zz \Lambda V ||_2 + \delta
		\geq ||\Phi (V\zz \Lambda V) ||_2 + \delta.
	\end{eqnarray*}
    Hence, it holds that 
    \begin{eqnarray}\label{eq:S5}
    	\sigma_{\min}\left(\beta \Sigma^2 -\Phi (V\zz \Lambda V) \right)  \geq \delta.
    \end{eqnarray}
	Submitting \eqref{eq:S5} into \eqref{eq:S3}, we arrive at 
	\begin{eqnarray*}
	||X||_2 ||G||\ff &\geq& ||\Phi(V\zz X\zz G V)||\ff =  ||(I_p-\Sigma^2)(\Phi(V\zz \Lambda V) -\beta\Sigma^2)||\ff\\
		&\geq &  ||I_p-\Sigma^2||\ff \cdot \sigma_{\min}\left(\beta\Sigma^2 - \Phi(V\zz \Lambda V)\right)
		\geq  ||I_p-X\zz X||\ff \cdot \delta
	\end{eqnarray*}
    and complete the proof.
\end{proof}

\section{Parallelizable Algorithms}

In this section, we introduce a parallelizable approach and one of its variant for 
optimization problem with orthogonality constraints \eqref{prob}. 
Both of these two approaches are based on the augmented Lagrangian function
\eqref{eq:Lag} and employ the new idea of updating the 
multipliers by explicit expression instead of dual ascent step in Algorithm \ref{alg:ALM}.

Another distinction between our algorithms
and the classical ALM
is that the minimization subproblem for the prime variables 
is replaced by 
a proximal linearized approximation.

\subsection{The Proximal Linearized Augmented Lagrangian Algorithm}
We describe our main algorithm framework in Algorithm \ref{alg:PLAM}.
\begin{algorithm2e}[ht]
	\caption{Proximal Linearized Augmented Lagrangian Algorithm (PLAM)}
	\label{alg:PLAM}
	\SetKwInOut{Input}{input}\SetKwInOut{Output}{output}
	\SetKwComment{Comment}{}{}
	\BlankLine 
	\textbf{Input:} choose initial guess $X^0$, and set $k:=0$\;
	\While{certain stopping criterion is not reached}
	{
		Compute the Lagrangian multipliers 
		\begin{eqnarray}\label{eq:Lambda1}
		\Lambda^k := \varPsi(\nabla f(X^k)\zz X^k).
		\end{eqnarray}
		
		Minimize the following proximal linearized Lagrangian function 
		\begin{eqnarray}\label{eq:PLLag}
		X^{k+1}:=\argmin\limits_{X\in\Rnp} \, \tilde{\cL}_\beta(X)= \tr(\nabla_X \cLb(X^k,\Lambda^k)\zz (X-X^k)) + \frac{\eta^k}{2}
		||X-X^k||\fs.
		\end{eqnarray}
		
		Set $k:=k+1$.
	}
	\textbf{Output:} $X^k$.
\end{algorithm2e}

The main calculation costs of Algorithm \ref{alg:PLAM} concentrate
at Step 3 and 4. Step 3 only involves BLAS3 calculation. The minimization
subproblem \eqref{eq:PLLag} in Step 4 is nothing but a gradient step
\begin{eqnarray}\label{eq:PLAM-main}
		X^{k+1} &=& X^k - \frac{1}{\eta^k}\nabla_X \cLb(X^k,\Lambda^k)\nonumber\\
		&=& X^k - \frac{1}{\eta^k}\dkh{\nabla f(X^k) + \beta X^k\dkh{{X^k}\zz X^k -I_p-\frac{1}{\beta}\Lambda^k}}\nonumber\\
		&=& X^k - \frac{1}{\eta^k}\left(
		\nabla f(X^k) -  X^k \varPsi(\nabla f(X^k)\zz X^k)
		+ \beta X^k({X^k}\zz X^k - I_p)
		\right),
\end{eqnarray}
where the last step is due to the updating formula \eqref{eq:Lambda1}.
Apparently, the arithmetic operations involved in
\eqref{eq:PLAM-main} belong to BLAS3 as well. 

We notice that ${1}/{\eta^k}$ is nothing but the stepsize of
gradient step. Hence, the proximal parameter $\eta^k$
can be chosen in the same manner as how we choose stepsize for gradient
methods. This issue will be described in details in Section 5.

\subsection{Parallelizable Column-wise Block Minimization}
An obvious demerit of PLAM is the boundedness of the iterate sequence
can hardly be expected without any restriction on
the penalty parameter $\beta$ and the proximal parameter $\eta^k$. 
Theoretically, to guarantee the global convergence, 
$\beta$ should be sufficiently large. Accordingly, $\eta^k$ should 
be large as well which means sufficiently small stepsize is required
and slow convergence can be expected.
In fact, according to the 
empirical observations, the performance of PLAM is very sensitive to 
parameters
$\beta$ and $\eta^k$. In other word, it is not easy to 
tune these two parameters to guarantee good performance of Algorithm \ref{alg:PLAM} in general.

Therefore, we put forward an upgraded version of PLAM.
It is based on PLAM, but redundant column-wise unit sphere constraints
are imposed to Step 4. Therefore, the proximal gradient takes the place of the gradient step in
the Step 4 of Algorithm \ref{alg:PLAM}.
With redundant constraints, the resulting iterate sequence will 
then be restricted to a compact set and hence bounded.
We describe the framework of this upgraded PLAM in Algorithm \ref{alg:PCAL}.

\begin{algorithm2e}[ht]
	\caption{Parallelizable Column-wise Block Minimization for PLAM (PCAL)}
	\label{alg:PCAL}
	\SetKwInOut{Input}{input}\SetKwInOut{Output}{output}
	\SetKwComment{Comment}{}{}
	\BlankLine 
	\textbf{Input:} choose initial guess $X^0$, and set $k:=0$\;
	\While{certain stopping criterion is not reached}
	{
		Compute the Lagrangian multipliers by \eqref{eq:Lambda1} or
		\begin{eqnarray}\label{eq:lambda2}
			\Lambda^k :=  \varPsi(\nabla f(X^k)\zz X^k) + \Phi\left({X^k}\zz 
			\nabla_X L_\beta (X^k,\varPsi(\nabla f(X^k)\zz X^k))
			\right).
		\end{eqnarray}
	
		\For{$i = 1,...,p$}{ 
			Minimize the following proximal linearized Lagrangian function
			\begin{eqnarray}\label{eq:PLLagSph}
				\begin{array}{rcl}
				X_i^{k+1}:=\argmin\limits_{x\in\Rn} && \tilde{\cL}^{(i)}_\beta(x)= \nabla_{X} \cLb(X^k,\Lambda^k)_i\zz (x-X_i^k) + \frac{\eta^k}{2}
				||x-X_i^k||_2^2,\\
				\st && ||x||_2 = 1.
				\end{array}
			\end{eqnarray}
			
			Update $X^{k+1}=[X_1^{k+1},\dots,X_p^{k+1}]$, and set $k:=k+1$.
		}
	}
	\textbf{Output:} $X^k$.
\end{algorithm2e}

Subproblem \eqref{eq:PLLagSph} in Algorithm \ref{alg:PCAL} can be solved in a 
column-wisely parallel fashion. In fact, it is of closed-form solution
\begin{eqnarray*}
		X_i^{k+1} &=& \frac{X_i^k -\frac{1}{\eta^k}\nabla_{X_i} \cLb(X^k,\Lambda^k) }{\norm{X_i^k -\frac{1}{\eta^k}\nabla_{X_i} \cLb(X^k,\Lambda^k)}_2}.	
\end{eqnarray*}	

For PCAL, we can update the Lagrangian multipliers in the same manner as PLAM,
i.e. by formula \eqref{eq:Lambda1}. To obtain a better performance, we can also use 
the heuristic formula \eqref{eq:lambda2}. The motivation of updating formula \eqref{eq:lambda2}
comes from the following observation.
In the KKT condition \eqref{eq:kkt}, we impose an additional term for the redundant 
sphere constraints. Namely,
\begin{eqnarray}\label{eq:kkt3}
\left\{
\begin{array}{c}
\nabla f(X) = X\Lambda + XD;\\
X\zz X = I_p,
\end{array}
\right.
\end{eqnarray}
where $D$ is a diagonal matrix. Furthermore, $D$ is determined by the 
Lagrangian multiplier of $X_i$ in the subproblem \eqref{eq:PLLagSph}.

\subsection{Computational Cost}

In this subsection, we compare the computational cost per iteration among
MOptQR, PLAM and PCAL. The computational cost of the basic linear algebra operations
and the overall costs of the aforementioned algorithms are listed in {Table} \ref{tab:cost}.

\begin{table}[htbp]
	\scriptsize
	\centering
	\begin{tabular}{c|c|c|c|c}
		\hline
		\toprule[.3mm]
		\multicolumn{5}{c}{\bf evaluate function} \\\hline
		\multirow{4}{*}{$f(X):=\frac{1}{2}\tr(X\zz A X) +\tr(G\zz X)$} & \multirow{2}{*}{$AX$} & {$A$: dense} & {$A$: sparse} & {$A$: sparse}\\\cline{3-5}
		& & $2n^2p$ & $O(np)$ & \multirow{3}{*}{$O(np)$}\\\cline{2-4}
		& $\nabla f(X)=AX+G$ & \multicolumn{2}{c|}{$np$} &\\\cline{2-4}
		& $\frac{1}{2}\tr(X\zz A X) +\tr(G\zz X)$ & \multicolumn{2}{c|}{$4np$} &\\\hline
		
		\multicolumn{5}{c}{\bf KKT: $\nabla f(X)-X{\nabla f(X)}\zz X$} \\\hline
		${\nabla f(X)}\zz X$ & \multicolumn{3}{c|}{$2np^2$} & \multirow{2}{*}{$4np^2+np$}\\\cline{1-4}
		$X({\nabla f(X)}\zz X)$ & \multicolumn{3}{c|}{$2np^2$} &\\\hline
		
		\multicolumn{5}{c}{\bf feasibility: $X\zz X -I$} \\\hline
		${X}\zz X$ & \multicolumn{3}{c|}{$np^2$} & \multirow{1}{*}{$np^2+np$}\\\hline
		
		\multicolumn{5}{c}{\bf solvers} \\\hline
		\multirow{2}{*}{PLAM} & $X(X\zz X-I)$ & \multicolumn{2}{c|}{$2np^2$} & \multirow{2}{*}{$4np^2+O(np)$}\\\cline{2-4}
		& $X\varPsi(\nabla f(X)\zz X)$ & \multicolumn{2}{c|}{$2np^2$} & \\\cline{1-5}	
		
		\multirow{4}{*}{PCAL} & $X(X\zz X-I)$ & \multicolumn{2}{c|}{$2np^2$} & \multirow{4}{*}{$4np^2+O(np)$}\\\cline{2-4}
		& $X\varPsi(\nabla f(X)\zz X)$ & \multicolumn{2}{c|}{$2np^2$} & \\\cline{2-4}
		& $\Phi\left({X^k}\zz 
		\nabla_X L_\beta (X^k,\varPsi(\nabla f(X^k)\zz X^k))
		\right)$ & \multicolumn{2}{c|}{$O(np)$} & \\\cline{2-4}
		& $X\Lambda=X\varPsi(\cdot)+X\Phi(\cdot)$ & \multicolumn{2}{c|}{$O(np)$} & \\\cline{1-5}
		
		\multirow{4}{*}{MOptQR (cholesky $LL\zz$)} & $V:=X-\tau({\nabla f(X)}-X{\nabla f(X)}\zz X)$ & \multicolumn{2}{c|}{$2np$} & \multirow{4}{*}{$3np^2+{\color{red}O({p^3})}+O(np)$}\\\cline{2-4}
		& $V\zz  V$ & \multicolumn{2}{c|}{$np^2$} & \\\cline{2-4}
		& $\text{chol}(V\zz V)=LL\zz $ & \multicolumn{2}{c|}{{\color{red}$p^3/3$}} & \\\cline{2-4}
		& $VL^{-\top}$ & \multicolumn{2}{c|}{$2np^2+{\color{red}O(p^3)}$} & \\\cline{1-5}
		MOptQR (Gram-Schmidt) & \multicolumn{3}{c|}{$2np^2$} & \multirow{1}{*}{\color{red}$2np^2+O(np)$}\\\hline
		
		\multicolumn{5}{c}{\bf in total} \\\hline
		PLAM & \multicolumn{4}{c}{$7np^2+O(np)$} \\\hline
		PCAL & \multicolumn{4}{c}{$7np^2+O(np)$} \\\hline
		MOptQR & \multicolumn{4}{c}{$7np^2+{\color{red}O(p^3)}+O(np)$ for cholesky, $4np^2+{\color{red}2np^2+O(np)}$ for Gram-Schmidt} \\\hline
		
	\end{tabular}
	\caption{The comparison of computational cost\label{tab:cost}} 
\end{table}

Here, those terms in red represent the corresponding operations that cannot be parallelized.

In practice, we calculate $X\varPsi(\nabla f(X)\zz X)$ instead of $X({\nabla f(X)}\zz X)$ for KKT evaluation since they are very close to each other around any first-order stationary point. Consequently, it 
saves $2np^2$ flops computational cost.

\section{Convergence of PLAM}

In this section, we focus on the theoretical analyses of our proposed PLAM.
The global convergence, worst-case complexity and Q-linear local convergence rate 
will be established under different mild assumptions.

\subsection{Global Convergence of PLAM}
Besides blanket Assumption \ref{a1}, 
to prove the convergence of Algorithm \ref{alg:PLAM},
we need to impose a mild condition on the initial
guess, and restrictive conditions on $\beta$ and $\eta^k$.  
To facilitate the narrative, we first state all these conditions here.
\begin{assumption}\label{a2}
	For a given $X^0$, we say it is a qualified initial guess, if 
	there exists $\us\in(0,1)$ so that 
	\begin{eqnarray*}
		\sigma_{\min}(X^0)\geq \us,\qquad 0<||{X^0}\zz X^0 -I_p||\ff \leq 1-\us^2.
	\end{eqnarray*}
\end{assumption}

Assumption \ref{a2} is not restrictive at all. Therefore two types
of points satisfying this assumption and can be obtained easily:
\begin{itemize}
	\item[(i)] $X^0=Q\Sigma$, where $Q\in\stiefel$, 
	$\Sigma=\Diag(\underbrace{1,...,1}_{p-1},\sqrt{1-\us^2})$;
	\item[(ii)] $X^0\notin\stiefel$ satisfying $\sigma^2_{\min}(X^0)>1-\frac{1}{\sqrt{p}}$ and 
	$\sigma^2_{\max}(X^0)<1+\frac{1}{\sqrt{p}}$.
\end{itemize}

Now, we list all the special notations to be used in this section.
\begin{eqnarray}\label{eq:notation}
\begin{array}{ccl}
		&& R = ||{X^0}\zz X^0-I_p||\ff; \quad 
	\cC =\{X\mid  ||X\zz X-I_p||\ff \leq R\};\quad  \underline{f}=  \min\limits_{X\in\cC} f(X);\\
	&& M = \max\limits_{X\in\cC} ||X||_2; \quad N = \max\limits_{X\in\cC} ||\nabla f(X)||\ff;\quad 
	L = \max\limits_{X\in\cC} ||\nabla^2 f(X)||_2.
\end{array}
\end{eqnarray}
We introduce the following merit function 
\begin{eqnarray}\label{eq:merit}
h(X) = f(X) - \frac{1}{2}\left\langle \varPsi(\nabla f(X)\zz X), 
X\zz X-I_p\right\rangle + \frac{\beta}{4} ||X\zz X-I_p||\fs.
\end{eqnarray}
According to the twice continuous differentiability of $f(X)$, $\nabla f(X)$ is Lipschitz continuous
on the compact set $\cC$. Namely, there exists constant $L_h>0$, related to $\beta$, so that
\begin{eqnarray}\label{eq:Lip}
||\nabla h(X)-\nabla h(y)||\ff\leq L_h||X-Y||\ff, \quad\forall\,X,Y\in\cC.
\end{eqnarray}

The algorithm parameters $\beta$ and $\eta^k$, and the constants used in the proof can be selected
by the following rules.
\begin{assumption}\label{a3}
\begin{eqnarray}\label{eq:c1beta}
	&& c_1\in\left(0,\frac{1}{2}\right);\quad \beta > \max \left\{
	\frac{MN}{\us^2}+\sqrt{\frac{M^2N^2}{\us^4}+\frac{(N+LM)^2}{4\us^2(1-2c_1)}}, \frac{MN}{\us}, \frac{4MN}{\us^2}
	\right\};\\
	\label{eq:c2eta}
	&& c_2\in\left(0,
	\frac{R^2(\beta\us^2 - 4MN)}{2N_L^2}\right];\quad  \eta^k \in  \left[\underline{\eta},  \bar{\eta}\right],\\ 
	&&\mbox{where}~\,  
	\underline{\eta} = \max \left\{
	\frac{L_h}{2c_1},\frac{2N_L M + N_L\sqrt{4M^2+2R}}{R},\frac{R+2M^2}{c_2}
	\right\},\nonumber\\
	&& N_L = (1+M^2)N+\beta RM,\quad 
	\bar{\eta}\geq \underline{\eta}.
	\nonumber
\end{eqnarray}
\end{assumption}

Now we give \revise{a} sketch of our proof. Suppose $\{X^k\}$ is the iterate sequence 
generated by Algorithm \ref{alg:PLAM}. The main steps of the proof include:
\begin{itemize}
	\item[(1)] Any iterate $X^k$ is in $\cC$,
	and $\us$ is a unified lower bound of the smallest singular values of the iterates $X^k$;
	\item[(2)] The merit function $h(X)$ is bounded below;
	\item[(3)] $\{h(X^k)\}$ monotonically decreases, and hence $\{X^k\}$ has at least one convergent subsequence;
	\item[(4)] Any cluster point of $\{X^k\}$, say $X^*$, is a first-order stationary point of
	 the augmented Lagrangian function \eqref{eq:Lmin} with $\Lambda^*= 
	\varPsi(\nabla f(X^*)\zz X^*)$;
	\item[(5)] Any cluster point of $\{X^k\}$, say $X^*$, is a first-order stationary point of
	the original optimization problem with orthogonality constraints  \eqref{prob}.
\end{itemize}

Next we provide five concrete lemmas or corollaries following the above-mentioned sketch.

\begin{lemma}\label{lm:Xbound}
	Suppose $\{X^k\}$ is the iterate sequence 
	generated by Algorithm \ref{alg:PLAM} initiated from $X^0$ satisfying Assumption \ref{a2},
	and the problem parameters satisfy Assumption \ref{a3}.
	Then it holds that
	\begin{eqnarray}\label{eq:bound}
	\sigma_{\min}(X^k) \geq \us,\quad X^k\in\cC.
	\end{eqnarray}
\end{lemma} 
\begin{proof}
	We use mathematical induction. The argument \eqref{eq:bound} directly holds for $X^0$ resulting 
	from Assumption \ref{a2}. Next we investigate whether \eqref{eq:bound}
	holds at $X^{k+1}$ provided that it holds for $X^k$.
	
	Case I, $||{X^k}\zz X^k - I_p||\ff \leq \frac{R}{2}$.
	We have
	\begin{eqnarray*}
	&&||{X^{k+1}}\zz X^{k+1} - I_p||\ff\\
	&=& \left\|\left(X^k -\frac{1}{\eta^k}\nabla_X \cLb(X^k,\Lambda^k)\right)\zz
	\left(X^k -\frac{1}{\eta^k}\nabla_X \cLb(X^k,\Lambda^k)\right) - I_p \right\|\ff\\
	&\leq& ||{X^{k}}\zz X^{k} - I_p||\ff + \frac{2}{\eta^k}
	||X^k||_2 ||\nabla_X \cLb(X^k,\Lambda^k)||\ff   + \frac{1}{(\eta^k)^2}
	||\nabla_X \cLb(X^k,\Lambda^k)||\fs.
	\end{eqnarray*}
	It is not difficult to verify that 
	\begin{eqnarray*}
		||\nabla_X \cLb(X^k,\Lambda^k)||\ff = 
		\norm{\nabla f(X^k) -  X^k \varPsi(\nabla f(X^k)\zz X^k)
			+ \beta X^k({X^k}\zz X^k - I_p)}\ff  \leq 
		(1+M^2)N+\beta RM = N_L
	\end{eqnarray*}
	holds for any $X^k\in\cC$.
	By using the facts $X^k\in\cC$, \eqref{eq:notation} and \eqref{eq:c2eta}, we have
	\begin{eqnarray*}
		\frac{2}{\eta^k}
		||X^k||_2 ||\nabla_X \cLb(X^k,\Lambda^k)||\ff   + \frac{1}{(\eta^k)^2}
		||\nabla_X \cLb(X^k,\Lambda^k)||\fs \leq \frac{R}{2},
	\end{eqnarray*}
	which implies $||{X^{k+1}}\zz X^{k+1} - I_p||\ff\leq R$. \comm{This shows \eqref{eq:bound} is true for $k+1$.}
	
	Case II, $||{X^k}\zz X^k - I_p||\ff > \frac{R}{2}$. For convenience, we denote
	$c(X) = \frac{1}{2}||X\zz X-I_p||\fs$, 
	\begin{eqnarray}\label{eq:notation2}
		d = \nabla f(X^k) - X^k \Lambda^k,\quad C = {X^k}\zz X^k - I_p,\quad \delta
		= X^kC.
	\end{eqnarray}
	According to the facts $\sigma_{\min}(X^k) \geq \us$ and $X^k\in\cC$, we have
	\begin{eqnarray}\label{eq:notation3}
	\norm{\delta}\ff > \frac{R\us}{2}.
	\end{eqnarray}
	By using the fact that $\tr(AB)=\tr(AB\zz)$ if $A$ is symmetric, we have
	$$\tr(C{X^k}\zz\nabla f(X^k)) = \tr(C\nabla f(X^k)\zz X^k)=\tr(C\Lambda^k).$$
	Hence, we have 
	\begin{eqnarray}\label{eq:innerprod}
		\langle d, \delta\rangle  &=& \tr(C{X^k}\zz \nabla f(X^k) -C{X^k}\zz X^k \Lambda^k )\nonumber\\
		&=& \tr(C{X^k}\zz\nabla f(X^k) -C(C+I_p)\Lambda^k) = -\tr(C^2\Lambda^k).
	\end{eqnarray}
	
	 Notice that $L_c = 2R+4M^2$ is the Lipschitz constant 
	 of $\nabla c(X)$ over $\cC$. Due to the fact \eqref{eq:c2eta}, \eqref{eq:notation3} and \eqref{eq:innerprod}, we have
	 \begin{eqnarray*}
	 &&	\tr\dkh{\nabla_X \cLb(X^k,\Lambda^k)\zz \nabla c(X^k)} - c_2 ||\nabla_X \cLb(X^k,\Lambda^k)||\fs\\
	 &\geq&  2\langle d+\beta\delta, \delta\rangle - c_2 N_L^2 = 2\beta\norm{\delta}\fs + 2\langle d, \delta\rangle - c_2 N_L^2\\
	 &>&  \frac{\beta R^2\us^2}{2}-2\norm{C}\fs\cdot\tr(\Lambda^k) - c_2 N_L^2\\
	 &\geq& \frac{\beta R^2\us^2}{2}-2R^2 MN - c_2 N_L^2 \geq 0.
	 \end{eqnarray*}
	
	 According to the Taylor expansion, we have
	 \begin{eqnarray*}
	 	&& c(X^{k+1}) = c\left(X^k-\frac{1}{\eta^k}\nabla_X \cLb(X^k, \Lambda^k)\right)\\
	 	&\leq& c(X^k) -\frac{1}{\eta^k}\left\langle \nabla_X \cLb(X^k, \Lambda^k), \nabla c(X^k)\right\rangle
	 	+ \frac{L_c}{2(\eta^k)^2} ||\nabla_X \cLb(X^k, \Lambda^k)||\fs\\
	 	& <&  c(X^k) - \left(
	 	\frac{c_2}{\bar{\eta}} -\frac{L_c}{2\bar{\eta}^2}
	 	\right)\cdot||\nabla_X \cLb(X^k, \Lambda^k)||\fs
	 	\leq c(X^k).
	 \end{eqnarray*}
     
     According to assumption, $R\leq 1- \us^2$, we can easily obtain that $\sigma_{\min}(X^{k+1})\geq \us$.
     This completes the proof.
\end{proof}

\begin{lemma}\label{lm:hbound}
	$h(X)$ defined by \eqref{eq:merit} is bounded below at $\cC$.
\end{lemma}

This lemma immediately holds by the continuous
differentiability of $h(X)$ and the compactness of $\cC$, and hence, the proof is omitted.

\begin{lemma}\label{lm:hdecrease}
	Suppose $\{X^k\}$ is the iterate sequence 
	generated by Algorithm \ref{alg:PLAM} initiated from $X^0$ satisfying Assumption \ref{a2},
	the problem parameters satisfy Assumption \ref{a3},
	and $h(X)$ is defined by \eqref{eq:merit}.
	Then it holds that
	\begin{eqnarray}\label{eq:suffred}
		h(X^k) - h(X^{k+1}) \geq  c_3 ||\nabla_X \cLb(X^k,\Lambda^k)||\fs,
	\end{eqnarray}
	where $c_3=
	\frac{c_1}{\bar{\eta}} - \frac{L_h}{2\bar{\eta}^2}>0
	$.
\end{lemma}
\begin{proof}
	Firstly,  we notice that
	\begin{eqnarray*}
		\nabla h(X) = \nabla_X \cLb(X,\varPsi(\nabla f(X)\zz X))
		- \frac{1}{2} (\nabla^2 f(X)[X] +\nabla f(X)) (X\zz X-I_p).
	\end{eqnarray*}
	We keep using the notations \eqref{eq:notation2} and investigate 
	\begin{eqnarray*}
		&&||\nabla_X \cLb(X^k,\Lambda^k)||\fs - 
		\frac{1}{1-2c_1} ||\nabla h(X^k) - \nabla_X \cLb(X^k,\Lambda^k)||\fs \\
		&\geq & 
		||d+\beta \delta||\fs - \frac{(N+LM)^2}{4(1-2c_1)}||C||\fs
		\,\geq\,  2\beta \langle d, \delta\rangle + \beta^2||\delta||\fs -
		\frac{(N+LM)^2}{4(1-2c_1)}||C||\fs\\
		&\geq & - \beta\norm{C}\fs\cdot\tr(\Lambda^k) + \left(
		 \beta^2\us^2 - \frac{(N+LM)^2}{4(1-2c_1)}
		\right)\cdot||C||\fs\\
		&\geq & 
		- 2\beta MN ||C||\fs + \left(
		\beta^2\us^2 - \frac{(N+LM)^2}{4(1-2c_1)}
		\right)\cdot||C||\fs \geq 0,
	\end{eqnarray*}
    where the second last inequality is implied by \revise{relation} \eqref{eq:innerprod}.
    Hence, we arrive at
    \begin{eqnarray}\label{eq:angle}
    	\langle  \nabla_X \cLb(X^k,\Lambda^k),\nabla h(X^k)\rangle
    	\geq c_1 ||\nabla_X \cLb(X^k,\Lambda^k)||\fs.
    \end{eqnarray}
	Substituting \eqref{eq:c2eta} and \eqref{eq:angle} into the Taylor expansion, we have
	\begin{eqnarray*}
	    &&h(X^{k+1}) = h\left(X^k -\frac{1}{\eta^k}\nabla_X \cLb(X^k,\Lambda^k)\right)\\
	   &\leq& h(X^k) -\frac{1}{\eta^k}\left\langle
	   \nabla_X \nabla h(X^k), \cLb(X^k,\Lambda^k)\right\rangle
	   +\frac{L_h}{2(\eta^k)^2}||\nabla_X \cLb(X^k,\Lambda^k)||\fs\\
	   &\leq & h(X^k) -\left(\frac{c_1}{\bar{\eta}} - \frac{L_h}{2\bar{\eta}^2}\right)\cdot
	   ||\nabla_X \cLb(X^k,\Lambda^k)||\fs.
	\end{eqnarray*}
 	We complete the proof.
\end{proof}

With the boundedness of $h(X)$ at $\cC$, Lemma \ref{lm:hdecrease} immediately implies
the convergent of $\{h(X^k)\}$. More precisely, we have the following corollary.

\begin{corollary}\label{cl:glc}
	Suppose $\{X^k\}$ is the iterate sequence 
	generated by Algorithm \ref{alg:PLAM} initiated from $X^0$ satisfying Assumption \ref{a2},
	and the problem parameters satisfy Assumption \ref{a3}.
	Then the algorithm is finitely terminated at $k$-th iteration with 
	$\nabla_X \cLb(X^k,\Lambda^k)=0$, 
	or
	\begin{eqnarray*}
		\lim\limits_{k\rightarrow +\infty}\nabla_X \cLb(X^k,\Lambda^k)=0.
	\end{eqnarray*} 
	Moreover, $\{X^k\}$ has at least one convergent subsequence.
	Any cluster point of $\{X^k\}$, $X^*$, is a first-order stationary point of 
    minimizing the augmented Lagrangian function \eqref{eq:Lmin} with $\Lambda^*= \varPsi(\nabla f(X^*)\zz X^*)$.
\end{corollary}
\begin{proof}
	This is a direct corollary of Lemmas \ref{lm:Xbound} and \ref{lm:hdecrease}.
\end{proof}

Finally, we give the global convergence rate of PLAM, namely, the worst case complexity.
\begin{theorem}\label{thm:PLAM}
	Suppose $\{X^k\}$ is the iterate sequence 
	generated by Algorithm \ref{alg:PLAM} initiated from $X^0$ satisfying Assumption \ref{a2},
	and the problem parameters satisfy Assumption \ref{a3}.
	Then the sequence $\{X^k\}$ has at least one cluster point, and any which is a first-order stationary point of problem \eqref{prob}. More precisely, for any $K>1$, it holds
	\begin{eqnarray}\label{eq:rate}
		\min\limits_{k=0,...,K-1} ||\nabla_X \cLb(X^k,\Lambda^k)||\ff
		< \sqrt{\frac{f(X^0)-\underline{f}+MNR+\beta R^2/4}{c_3K}}.
	\end{eqnarray}
\end{theorem}
\begin{proof}
	The first part directly holds from Corollaries \ref{cl:glc} and Lemma \ref{lm:2}.
	Recalling Lemma \ref{lm:hdecrease}, we have
	\begin{eqnarray}\label{eq:tmp1}
		h(X^0) - \min\limits_{X\in\cC} h(X)
		&\geq& h(X^0) - h(X^K)\geq \sum\limits_{k=0}^{K-1} c_3 ||\nabla_X \cLb(X^k,\Lambda^k)||\fs\\
		&\geq& c_3 K\cdot  \min\limits_{k=0,...,K-1} ||\nabla_X \cLb(X^k,\Lambda^k)||\fs.
	\end{eqnarray}
	Moreover, we have
	\begin{eqnarray}\label{eq:tmp2}
		h(X^0)\leq f(X^0) +\frac{1}{2}MNR+\frac{\beta}{4}R^2,\quad \min\limits_{X\in\cC} h(X)
		\geq \underline{f} - \frac{1}{2}MNR.
	\end{eqnarray}
	Combining \eqref{eq:tmp1}-\eqref{eq:tmp2}, we arrive at the argument \eqref{eq:rate}.
\end{proof}

\begin{corollary}\label{cl:2}
Suppose all the assumptions of Theorem \ref{thm:PLAM} hold.
Besides, for a given positive parameter $\delta$, it holds that
$\beta > (MN+\delta)/\underline{\sigma}$, then it holds that
\begin{eqnarray*} 
\min\limits_{k=0,...,K-1} \max\left\{||I_p-{X^k}\zz X^k||\ff, ||\nabla_X \cLb(X^k,\Lambda^k)||\ff\right\}
< \max\left\{\frac{M}{\delta},1\right\}\sqrt{\frac{f(X^0)-\underline{f}+MNR+\beta R^2/4}{c_3K}}.
\end{eqnarray*}
\end{corollary}
\begin{proof}
 This is a direct corollary of Lemma \ref{lm:2} and Theorem \ref{thm:PLAM}.
\end{proof}

\begin{remark}
The sublinear convergence rate of Corollary \ref{cl:2} actually tells us
that Algorithm \ref{alg:PLAM} terminates after $O(1/\epsilon^2)$ iterations, if the 
stopping criterion is set as $ \max\left\{||I_p-{X^k}\zz X^k||\ff, ||\nabla_X \cLb(X^k,\Lambda^k)||\ff\right\}<\epsilon$.
\end{remark}

\subsection{Local Convergence Rate of PLAM and PCAL}
	
In this subsection, we consider the local convergence of PLAM
once the
optimization problem with orthogonality constraints \eqref{prob} has an isolated local minimizer.

\begin{theorem}\label{thm:PLAM-Local}
	Suppose $X^*$ is an isolated minimizer of \eqref{prob}, 
	and we denote
	\begin{eqnarray*}
		\tau:=\inf\limits_{0\neq Y\in\TX} \frac{\tr(Y\zz \nabla^2 f(X)[Y] - \Lambda Y\zz Y)}{||Y||\fs}.
	\end{eqnarray*}
	The algorithm parameters satisfy
	$\beta \geq \frac{L+MN+\tau}{2}$ and 
	$\eta^k \in[\underline{\eta},\bar{\eta}]$, where $\bar{\eta}\geq \underline{\eta}\geq L+MN+2\beta$. 
	Then, there exists $\varepsilon>0$ such that
	starting from any $X^0$ satisfying $||X^0-X^*||\ff < \varepsilon$, 
	and the iterate sequence $\{X^k\}$ generated by
	Algorithm \ref{alg:PLAM} converges to $X^*$ Q-linearly.
\end{theorem}
\begin{proof}
 	We study the iterate formula \eqref{eq:PLLag}.
 	\begin{eqnarray*}
 		X^{k+1} &=& X^k - \frac{1}{\eta^k}\nabla_X \cLb(X^k,\varPsi(\nabla f(X^k)\zz X^k));\\
 	    X^* &=& X^* - \frac{1}{\eta^k}\nabla_X \cLb(X^*,\varPsi(\nabla f(X^*)\zz X^*)).
 	\end{eqnarray*}
 	Subtracting the second one from the first one and using the Taylor expansion, we have
 	\begin{eqnarray}\label{eq:iter}
 		\delta^{k+1} = \delta^k - \frac{1}{\eta^k}\nabla_{XX}^2 \cLb(X^*,\varPsi(\nabla f(X^*)\zz X^*))[\delta^k]
 		+ o(||\delta^k||),
 	\end{eqnarray}
 	where $\delta^k=X^k-X^*$. Recall the expression of Hessian \eqref{eq:O2},
 	the fact that $\nabla f(X^*)\zz X^* = \varPsi(\nabla f(X^*)\zz X^*)$ and the assumption
 	on $\eta$, we have
 	\begin{eqnarray}\label{eq:rate1}
 		\left\|\frac{1}{\eta^k}\nabla_{XX}^2 \cLb(X^*,\nabla f(X^*)\zz X^*)[\delta^k]\right\|\ff
 		\leq ||\delta^k||\ff.
 	\end{eqnarray}
 	On the other hand, $\delta^k$ can be decomposed as the summation of three \revise{terms}:
 	\begin{eqnarray}\label{eq:decomp}
 		\delta^k = X^* S + X^* W + K,
 	\end{eqnarray}
 	where $S\in\R^{p\times p}$ is symmetric, $W\in\R^{p\times p}$ is skew-symmetric, $K\in\R^{n\times p}$
 	is perpendicular to $X^*$.
 	Since $X^*$ is a strict local minimizer, and $\TX$ is closed, we have $\tau >0$.
 	Hence, it holds that
 	\begin{eqnarray}\label{eq:rate2}
 		\tr\left((X^* W + K)\zz \nabla_{XX}^2 \cLb(X^*,\nabla f(X^*)\zz X^*)[X^* W + K]\right)
 		\geq \tau||X^* W + K||\fs,
 	\end{eqnarray}
 	as $X^* W + K\in\TX$.
 	Moreover, it follows from the assumption on $\beta$ that
 	\begin{eqnarray}\label{eq:rate3}
 		&& \tr\left((X^*S)\zz\nabla_{XX}^2 \cLb(X^*,\nabla f(X^*)\zz X^*)[X^* S]\right)\nonumber\\
 		&=& \tr(SX^*\nabla^2 f(X^*)X^* S - S^2 \nabla f(X^*)\zz X^* + 2\beta S^2)
 		\geq \tau||X^* S||\fs.
 	\end{eqnarray} 
 	Combining \eqref{eq:rate2}, \eqref{eq:rate3}, the symmetry of $S$, 
 	the skew symmetry of $W$, $K\zz X^*=0$ together with the assumption on $\eta$, we arrive at
 	\begin{eqnarray}\label{eq:rate4}
	&&	\tr\left({\delta^k}\zz\nabla_{XX}^2 \cLb(X^*,\nabla f(X^*)\zz X^*)[\delta^k]\right)\nonumber\\
	&=& \tr\left((X^* W + K)\zz \nabla_{XX}^2 \cLb(X^*,\nabla f(X^*)\zz X^*)[X^* W + K]\right)\nonumber\\
	&& + \tr\left((X^* W + K)\zz \nabla_{XX}^2 \cLb(X^*,\nabla f(X^*)\zz X^*)[X^*S]\right)\nonumber\\
    && + \tr\left((X^*S)\zz \nabla_{XX}^2 \cLb(X^*,\nabla f(X^*)\zz X^*)[X^* W + K]\right)\nonumber\\
    && + \tr\left((X^*S)\zz\nabla_{XX}^2 \cLb(X^*,\nabla f(X^*)\zz X^*)[X^* S]\right)\nonumber\\
    &\geq& \tau||X^* W + K||\fs + \tau||X^* S||\fs = \tau ||\delta^k||\fs.
 	\end{eqnarray}
 	
 	Notice that \eqref{eq:rate1} implies the positive semi-definiteness of the linear operator
 	$$I-\frac{1}{\eta^k}\nabla_{XX}^2 \cLb(X^*,\nabla f(X^*)\zz X^*).$$ Together with \eqref{eq:rate4},
 	we can conclude that
 	\begin{eqnarray*}
 		||\delta^{k+1}||\ff \leq (1-\tau)||\delta^k||\ff + o(||\delta^k||),
 	\end{eqnarray*}
 	which completes the proof.
\end{proof}

\begin{remark}
	The global and local convergence of PCAL can be established in the same manner as PLAM, 
	if the multipliers are updated by the same formula, \eqref{eq:Lambda1}, as PLAM.
\end{remark}

\section{Numerical Experiments}
In this section, we evaluate the numerical performance of our proposed algorithms PLAM and PCAL.
We first introduce the implementation details and the testing problems
in Subsection \ref{subsec:detail} and \ref{subsec:problem}, respectively.
Then, we report
the numerical experiments which are mainly of three folds.

In the first part, we mainly determine the default settings of our proposed algorithms,
which will be discussed in Subsection \ref{subsec:default}.
Then, in Subsection \ref{subsection:KSDFT}, we compare our PLAM and PCAL  with a few existing solvers by testing a bunch of instances, 
which are chosen from a MATLAB toolbox KSSOLV \cite{Yang09}.
All the algorithms tested in the first two parts are run in serial.
The corresponding experiments are performed on a 
workstation with one Intel(R) Xeon(R) Processor E5-2697 v2 
(at 2.70GHz$\times 12$, 30M Cache) and 128GB of RAM running in MATLAB R2016b under Ubuntu 12.04. 

Finally, we investigate the parallel efficiency of PCAL by comparing with a parallelized version
of MOptQR in Subsection \ref{section: parallel}. All the 
experiments in this subsection are performed on a single node of LSSC-IV\footnote{More information at \href{http://lsec.cc.ac.cn/chinese/lsec/LSSC-IVintroduction.pdf}{ http://lsec.cc.ac.cn/chinese/lsec/LSSC-IVintroduction.pdf}}, which is a high-performance computing cluster (HPCC) maintained at the State Key Laboratory of Scientific and Engineering Computing (LSEC), Chinese Academy of Sciences. The operating system of LSSC-IV is Red Hat Enterprise Linux Server 7.3. This node called ``{\ttfamily b01}" consists of two Intel(R) Xeon(R) Processor E7-8890 v4 (at 2.20GHz$\times 24$, 60M Cache) with 4TB shared memory. The total number of processor cores in this node is 96.

\subsection{Implementation Details}\label{subsec:detail}
There are two parameters in our algorithms PLAM and PCAL.
According to Theorem \ref{thm:PLAM}, 
the penalty parameter $\beta$ for PLAM should be sufficiently large. Although we
can estimate a suitable $\beta$ to satisfy the assumption of the theorem,
it would be too large in practice.
In the numerical experiments, we set $\beta$ as
an upper bound of $s:=||\nabla^2 f(0)||_2$ for PLAM, and $1$ for PCAL.

Another one is the proximal parameter $\eta$, whose reciprocal is the step size of the
gradient step in Algorithm \ref{alg:PLAM} and \ref{alg:PCAL}. 
Similar to $\beta$, we can not use the rigorous restriction in the theoretical analysis.
In practice, we have the following strategies to choose this parameter:
\begin{enumerate}
	\item $\eta_{\mathrm{C}}^k := \gamma$, where $\gamma>0$ is a sufficiently large constant.
	\item Differential approximation: 
	$$\eta_{\mathrm{D}}^k := \frac{||{\nabla_X \cLb(X^k,\Lambda^k) - \nabla_X \cLb(X^{k-1}, \Lambda^{k-1})}||_F}{||X^k-X^{k-1}||\ff}.$$
	\item Barzilai-Borwein(BB) strategy \cite{BB}:
	\begin{eqnarray*}
		{\eta^k_{\mathrm{BB}1}} := \frac{\abs{\jkh{S^{k-1},{Y^{k-1}}}}}{\jkh{S^{k-1},S^{k-1}}}, \quad
		\mbox{or} \quad {\eta^k_{\mathrm{BB}2}} :=\frac{{\jkh{Y^{k-1},{Y^{k-1}}}}}{\abs{\jkh{S^{k-1},Y^{k-1}}}},
	\end{eqnarray*}
	where
	\begin{eqnarray*}
		S^k = X^k - X^{k-1},\quad Y^k = \nabla_{X} \cLb(X^k,\Lambda^k) -  \nabla_{X} \cLb(X^{k-1},\Lambda^{k-1}).
	\end{eqnarray*}
	\item Alternating BB strategy \cite{Dai_Fletcher_2005}:
	\begin{eqnarray*}
		\eta_{\mathrm{ABB}}^k := \left\{
		\begin{array}{cl}
			\eta^k_{\mathrm{BB}1},& \mbox{for odd}~~k,\\
			\eta^k_{\mathrm{BB}2},& \mbox{for even}~~k.
		\end{array}
		\right.
	\end{eqnarray*}
\end{enumerate}

\comm{Unless specifically mentioned}, the stopping criterion used
for both serial and parallel experiments can be described as follows,
$$\frac{\norm{\nabla f({X})-{X}\nabla f({X})\zz {X}}\ff}{\norm{\nabla f({X^0})-{X^0}\nabla f({X^0})\zz {X^0}}\ff} < 10^{-8}.$$
The maximum number of iteration for all those solvers is set to $3000$. 

\subsection{Testing Problems}\label{subsec:problem}
In this subsection,
we introduce six types of problems which will be
used in the numerical experiments. 

{\bf Problem 1}: A simplification of discretized Kohn-Sham total energy minimization.
\begin{eqnarray}\label{eq:nonlinear}
\left.
\begin{array}{rc}
\min\limits_{X\in\Rnp} & \frac{1}{2}\tr(X\zz L X) + \frac{\alpha}{4}\rho(X)^\top
L^\dagger \rho(X)\\
\st & X\zz X=I_p,
\end{array}
\right.
\end{eqnarray}
where the matrix $L\in \Sn$ and $\rho(X):=\mbox{diag}(XX\zz)$. 
In the numerical experiments, we set $\alpha=1$, and $L$ is randomly generated by Gauss distribution,
 i.e.,  L{\ttfamily=randn(n)} in MATLAB language, and set  $L:=\frac{1}{2}(L+L\zz)$.
In this instance, $s=\norm{L}_2$.

{\bf Problem 2}: A class of quadratic minimization with orthogonality constraints.
\begin{eqnarray}\label{eq:OCQP}
\begin{array}{cc}
\min\limits_{X\in\Rnp}&\frac{1}{2}\tr(X\zz A X) +\tr(G\zz X)\\
\mbox{s.t.} & X\zz X=I_p,
\end{array}
\end{eqnarray}
where the matrices $A\in \Sn$ and $G\in \Rnp$.
This problem is adequately discussed in \cite{Gao2016}. In the numerical experiments, 
the matrices $A$ and $G$ are randomly generated in the same manner as in \cite{Gao2016}. 
Namely, 
\begin{eqnarray}
A&:=&P\Lambda P\zz,\\
G&:=&\kappa \cdot QD,
\end{eqnarray}
where the matrices $P${\ttfamily=qr(rand(n,n))}$\in \R^{n\times n}$, $\tilde{Q}${\ttfamily=rand(n,p)}$\in \R^{n\times p}$, $Q\in\Rnp$ and $Q_i=\tilde{Q}_i/||\tilde{Q}_i||_2$ ($i=1,2,...,p$),
and matrices $\Lambda\in \R^{n\times n}$ and $D\in \R^{p\times p}$ are diagonal matrices with
\begin{eqnarray}
\Lambda_{ii} &:=& \left\{
\begin{array}{rc}
\theta^{1-i},& \rand{1}{1} < \xi\\
-\theta^{1-i},& \rand{1}{1} \ge \xi
\end{array}
\right. \quad \forall i=1,2,\dots,n,\\
D_{jj} &:=& \zeta^{j-1},\quad \forall j=1,2,\dots,p.
\end{eqnarray}
Here, parameter $\theta \ge 1$ determines the decay of eigenvalues of $A$; parameter $\zeta \ge 1$ refers to the growth rate of column's norm of $G$.
The parameter $\kappa>0$ represents the scale difference between the quadratic term and the linear term. The default setting of these parameters are $\kappa=1, \theta=1.01, \zeta=1.01, \xi=1$.
In this instance, $s=\norm{A}_2$.

{\bf Problem 3:} Rayleigh-Ritz trace minimization, which is a special case of Problem 2.
\begin{eqnarray}\label{eq:RR}
\begin{array}{cc}
\min\limits_{X\in\Rnp}&\frac{1}{2}\tr(X\zz A X)\\
\mbox{s.t.} & X\zz X=I_p,
\end{array}
\end{eqnarray}
where the matrices $A\in \Sn$. In our experiments, the matrix $A$ 
is generated in the same manner as in Problem 2. In this instance, $s=\norm{A}_2$.

{\bf Problem 4:} Another class of quadratic minimization with orthogonality constraints.
\begin{eqnarray}\label{eq:QAP}
\begin{array}{cc}
\min\limits_{X\in\Rnp}&\frac{1}{2}\tr(A\zz XBX\zz )\\
\mbox{s.t.} & X\zz X=I_p.
\end{array}
\end{eqnarray}
where the matrices $A\in \Sn$ and $B\in \Sp$.
This problem is out of the scope of problems discussed in \cite{Gao2016},
but can be solved by PLAM or PCAL.
The matrices $A$ and $B$ are randomly generated by 
$A=${\ttfamily randn(n)}, $A:=\frac{1}{2}(A+A\zz)$ and $B=${\ttfamily randn(p)},
$B:=\frac{1}{2}(B+B\zz)$, respectively. In this instance, $s=\norm{A}_2\cdot\norm{B}_2$

{\bf Problem 5:} Discretized Kohn-Sham total energy minimization instances from KSSOLV \cite{Yang09}. 
\begin{eqnarray}\label{eq:KSDFT}
	\min\limits_{X\in\Rnp} \, E(X),\qquad \st\, X\zz X =I_p,
\end{eqnarray}
where the discretized Kohn-Sham total energy function $E(X)$ 
is defined by \eqref{eq:KS-energy}. 
All the data comes from MATLAB toolbox KSSLOV.  

{\bf Problem 6:} A synthetic instance of discretized Kohn-Sham total energy minimization.
\begin{eqnarray}\label{eq:nonlinear-exchange}
\left.
\begin{array}{rc}
\min\limits_{X\in\Rnp} & \frac{1}{2}\tr(X\zz L X) + \frac{1}{2}\rho(X)^\top L^\dagger \rho(X) - \frac{3}{4}\gamma\rho(X)\zz \rho(X)^{\frac{1}{3}}
\\
\st & X\zz X=I_p,
\end{array}
\right.
\end{eqnarray}
where the matrix $L\in \R^{n\times n}$ and $\rho(X):=\mbox{diag}(XX\zz)$. The parameter
 $\gamma=2(\frac{3}{\pi})^{1/3}$ and $\rho(X)^{\frac{1}{3}}$ denotes the component-wise cubic root of the
 vector $\rho(X)$. This problem adopts a special exchange functional $- \frac{3}{4}\gamma\rho(X)\zz 
 \rho(X)^{\frac{1}{3}}$ (the correlation term is ignored), which is introduced in \cite{Liu2015}.
 The generation of $L$ is in the same manner as in Problem 1.

\subsection{Default Settings}\label{subsec:default}
In this subsection, we determine 
the default settings for the 
proposed algorithms PLAM and PCAL. 

In the first experiment, we test PLAM and PCAL with these four different choices of $\eta^k$ on Problem
1-4. Here, we only illustrate the results of $\eta_{\mathrm{BB}1}$ for strategy (iii), since its 
performances overwhelms those with $\eta_{\mathrm{BB}2}$. 
The penalty parameter is fixed as $\beta=s+0.1$.
{\sc Fig.} \ref{fig:eta} shows the results of PLAM and PCAL with different $\eta^k$. 
From subfigures (a)-(d), we observe that PLAM with 
$\eta_{\mathrm{ABB}}$ outperforms others. Under the same setting, a comparison among PCAL with different 
$\eta^k$ is reported in subfigures (e)-(h). We notice
that PCAL with $\eta_{\mathrm{ABB}}$ is superior to 
the other $\eta^k$ choices. Then we set $\eta_{\mathrm{ABB}}$ as the default setting for PLAM and PCAL.
\begin{figure}[htbp]
	\centering	
	\subfigure[Problem 1]
	{\includegraphics[width=0.245\textwidth,height=0.2\textwidth]
		{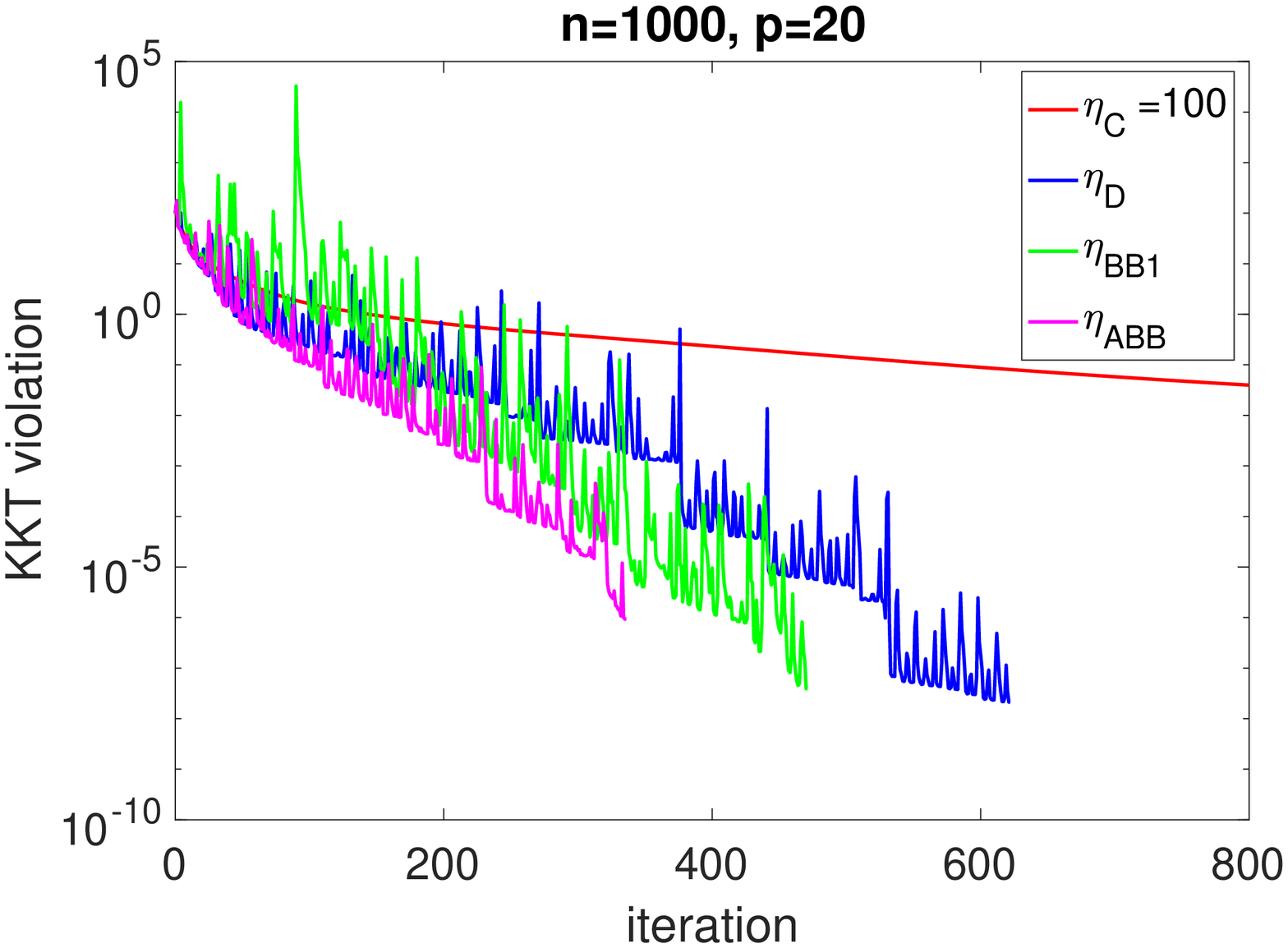}}
	\hfill
	\subfigure[Problem 2]
	{\includegraphics[width=0.245\textwidth,height=0.2\textwidth]
		{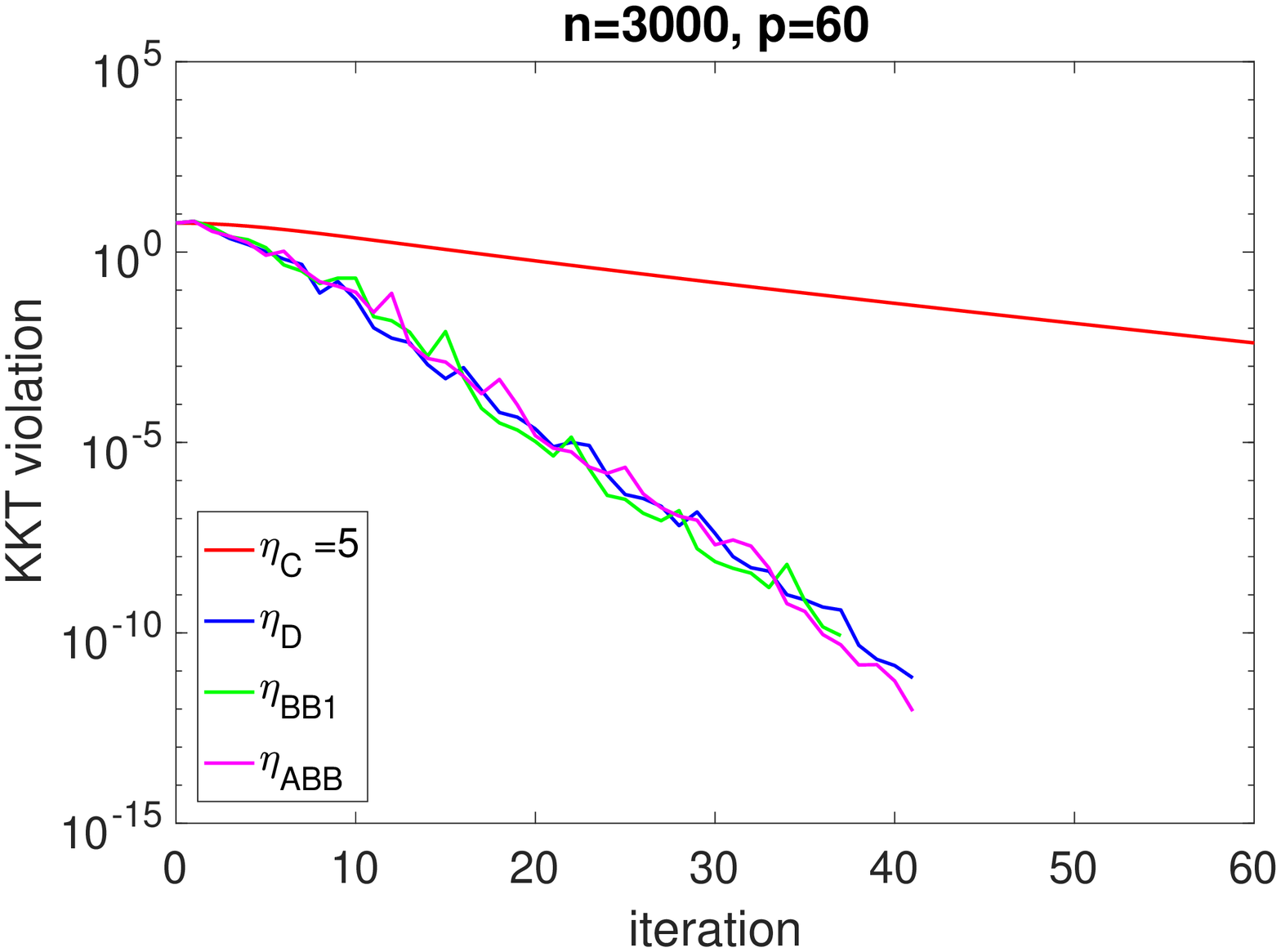}}
	\hfill
	\subfigure[Problem 3]
	{\includegraphics[width=0.245\textwidth,height=0.2\textwidth]
		{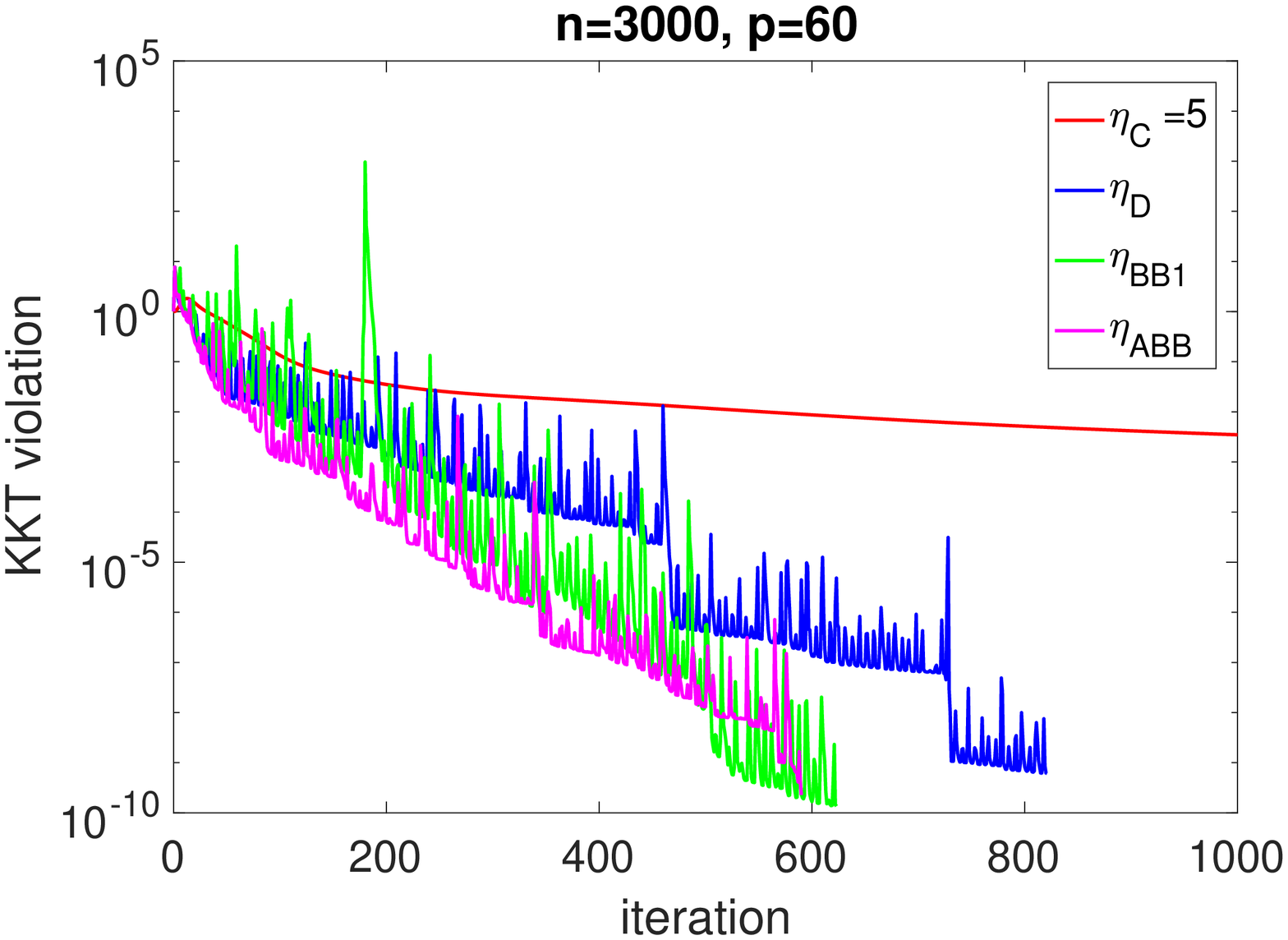}}
	\hfill
	\subfigure[Problem 4]
	{\includegraphics[width=0.245\textwidth,height=0.2\textwidth]
		{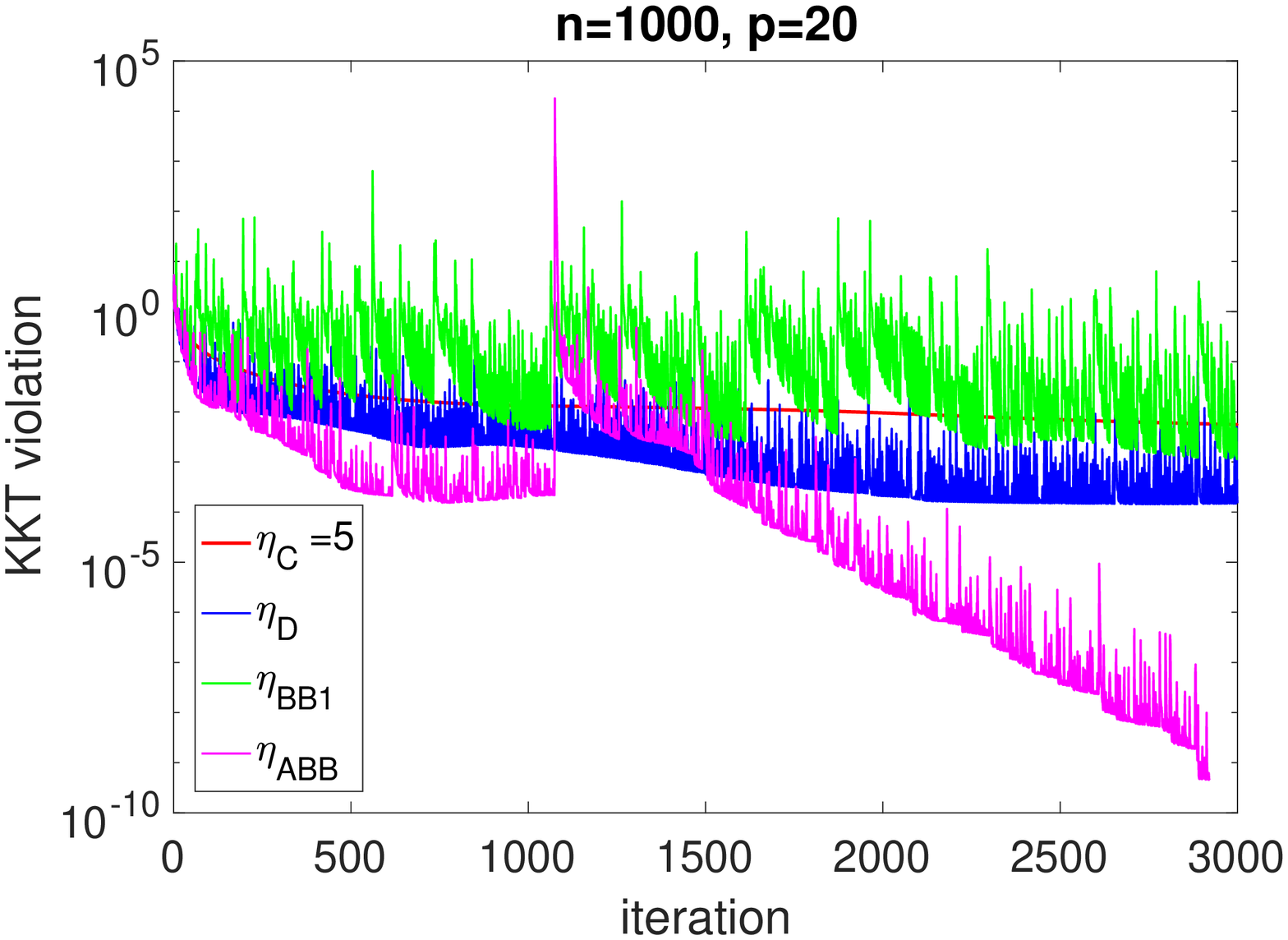}}
	
	\subfigure[Problem 1]
	{\includegraphics[width=0.245\textwidth,height=0.2\textwidth]
		{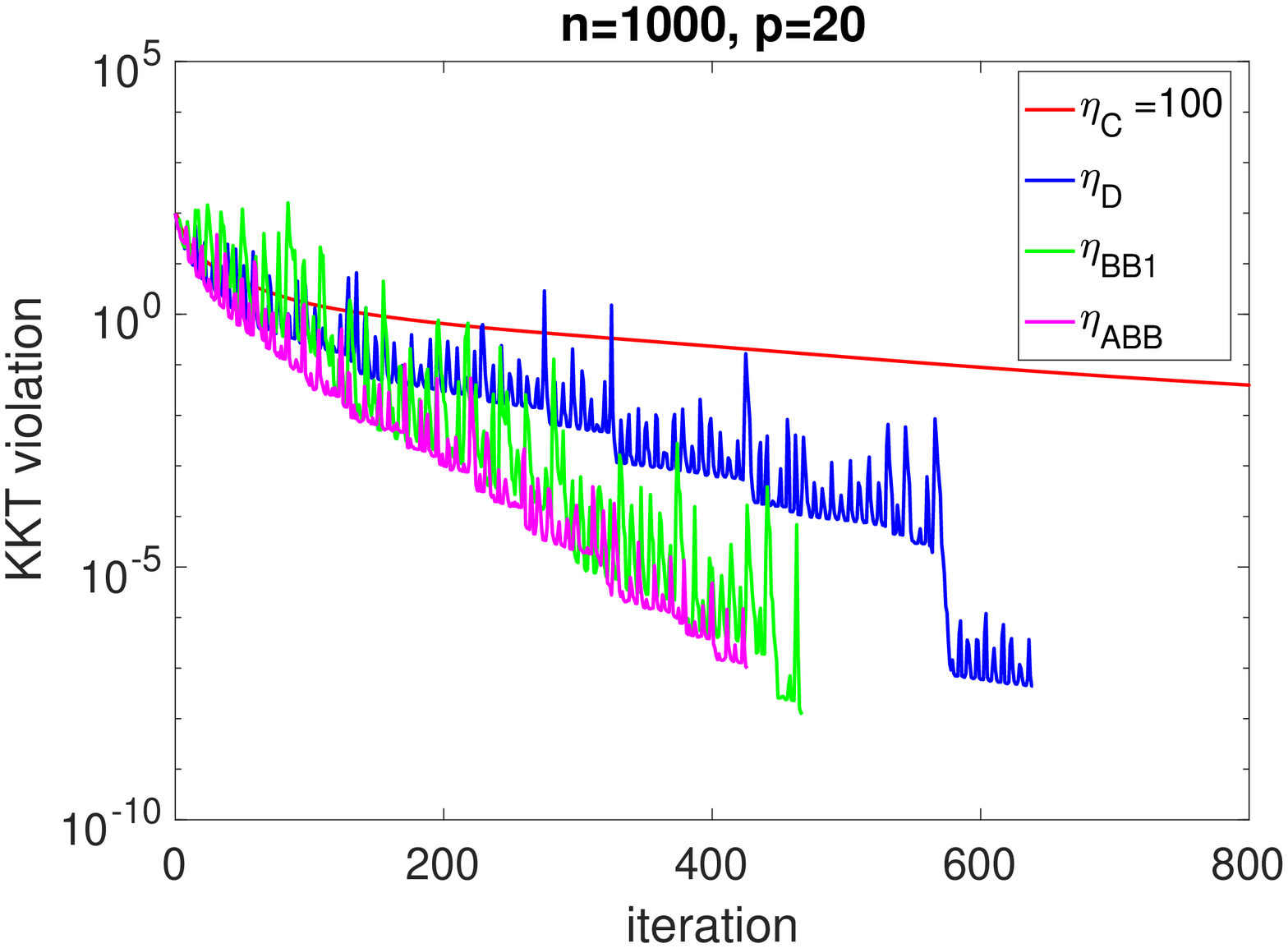}}
	\hfill
	\subfigure[Problem 2]
	{\includegraphics[width=0.245\textwidth,height=0.2\textwidth]
		{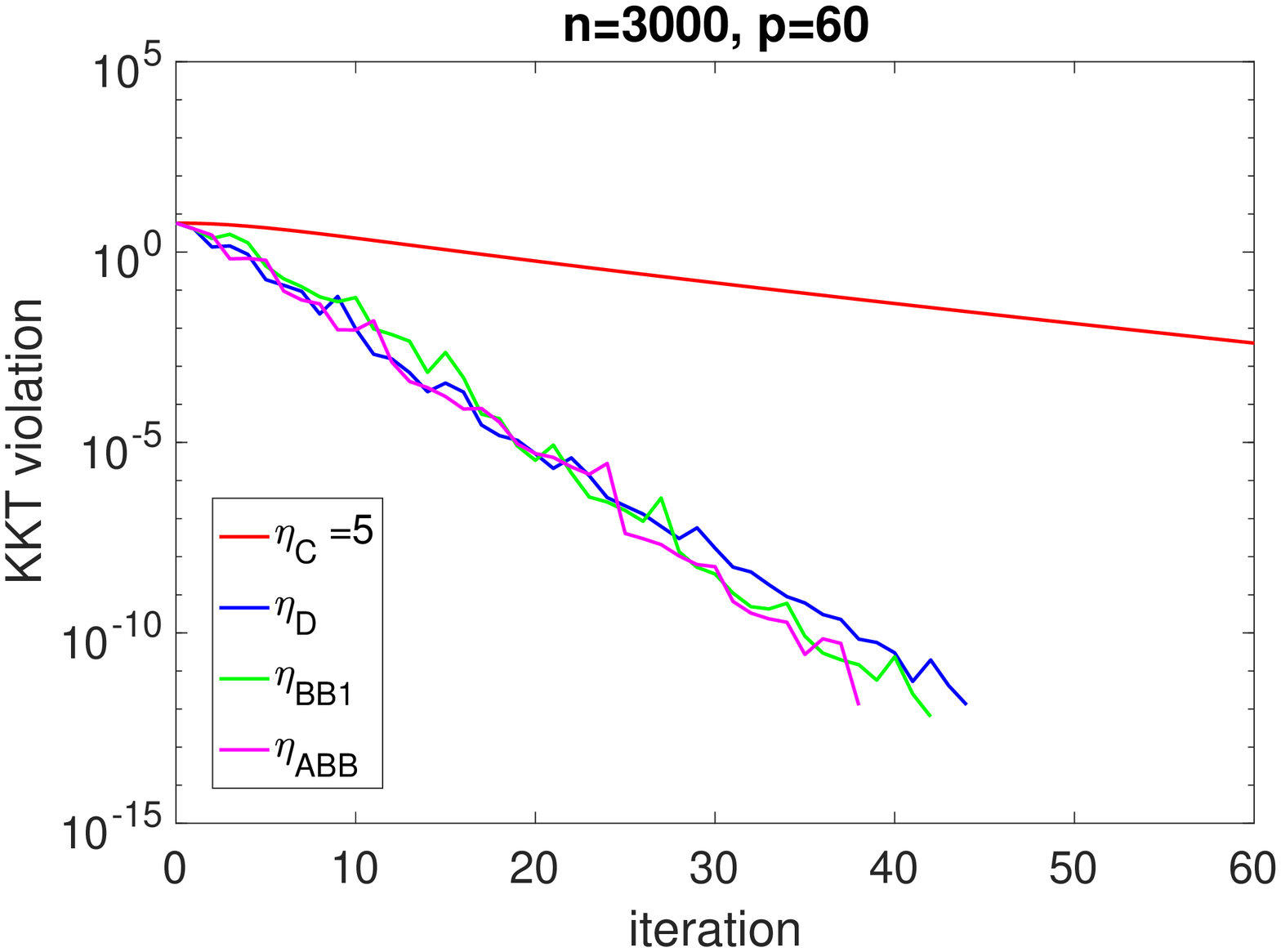}}
	\hfill
	\subfigure[Problem 3]
	{\includegraphics[width=0.245\textwidth,height=0.2\textwidth]
		{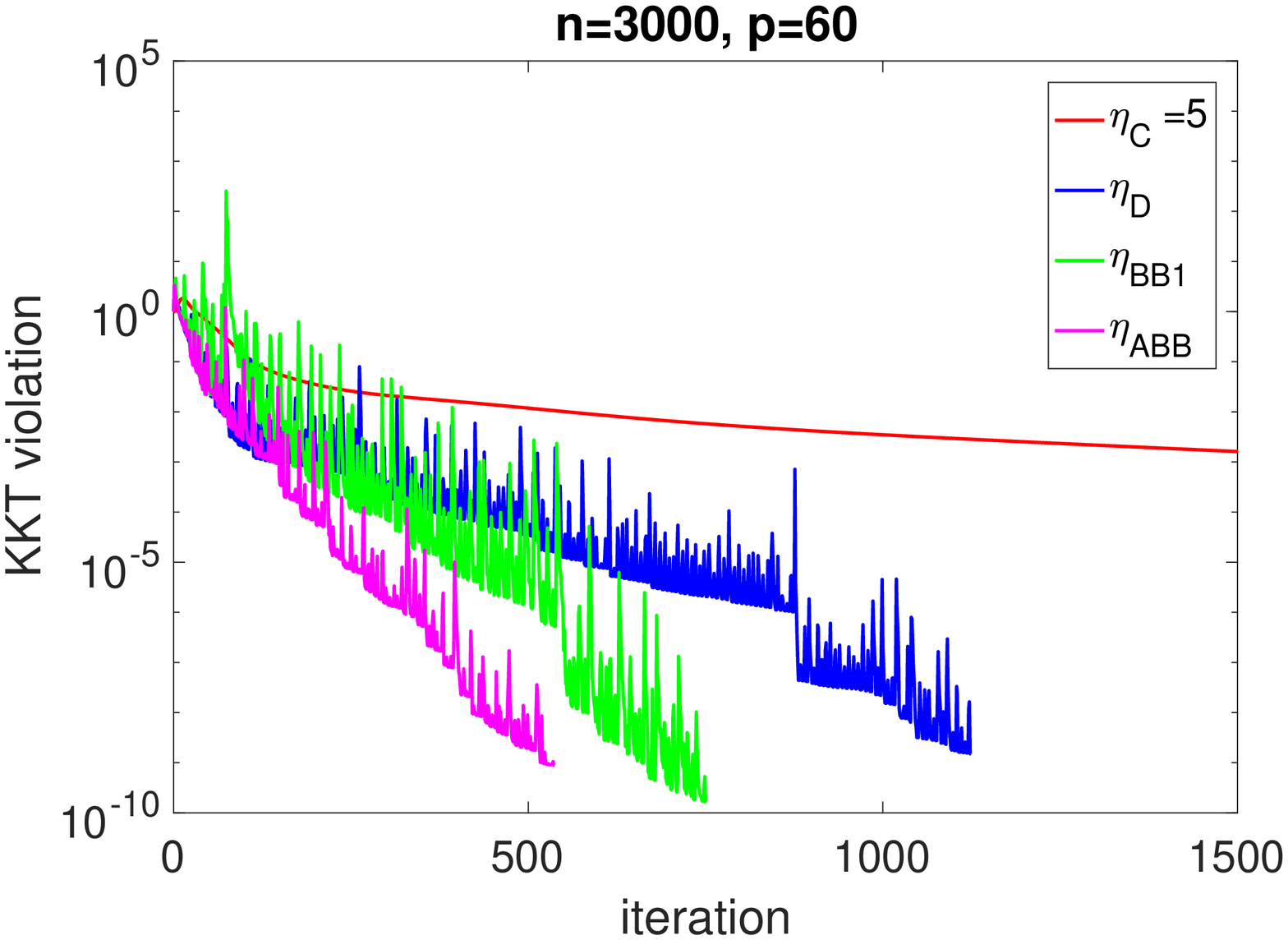}}
	\hfill
	\subfigure[Problem 4]
	{\includegraphics[width=0.245\textwidth,height=0.2\textwidth]
		{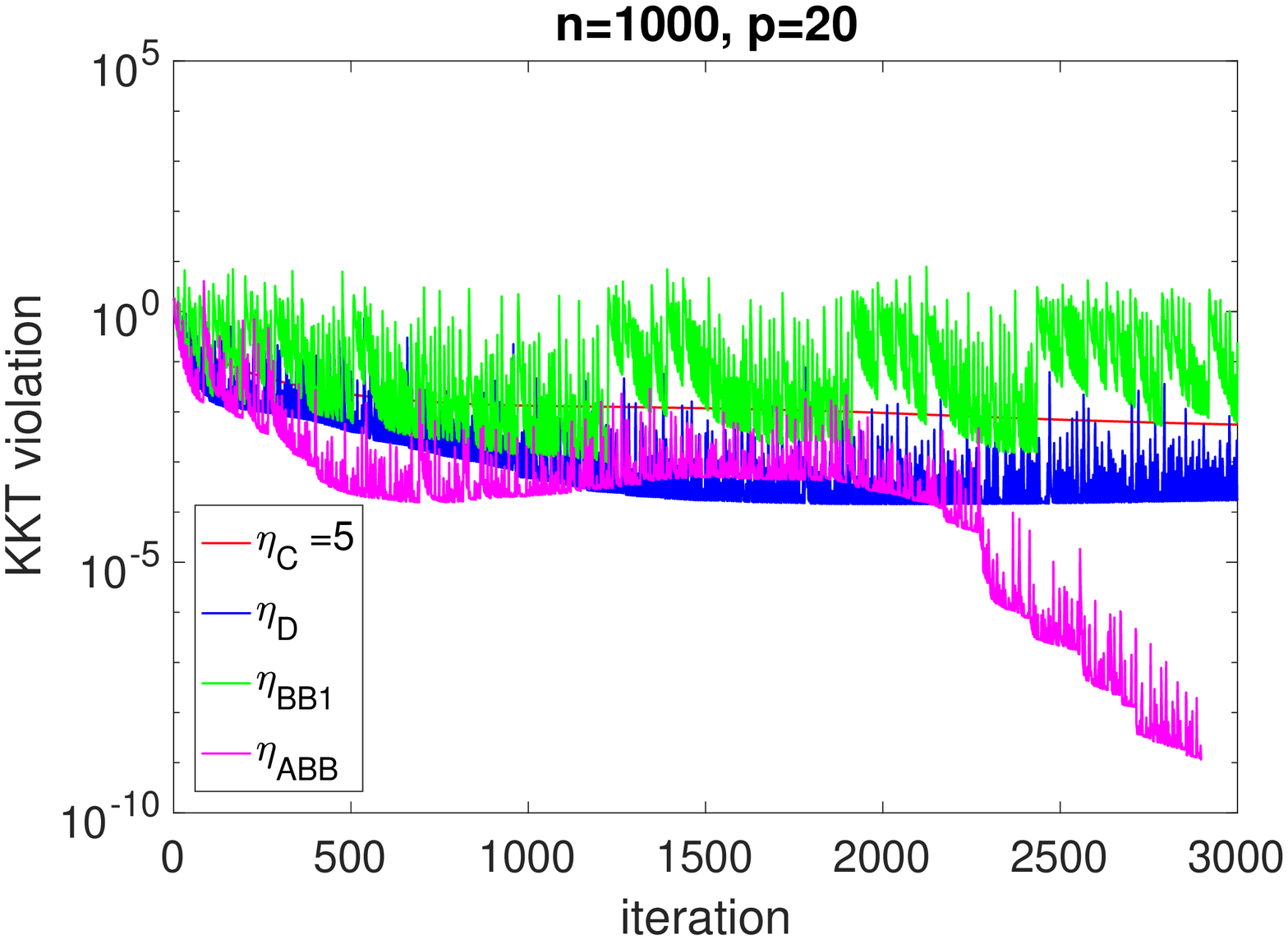}}	
	\caption{A comparison of KKT violation for PLAM (a)-(d) and PCAL (e)-(h) with different $\eta$ 
	($\beta=s+0.1$)}
	\label{fig:eta}
\end{figure}

We next compare the performance among PLAM and PCAL variations corresponding
to different $\beta$. 
In the comparison, we set $\beta$ varying among $0, 0.01s, 0.1s, s+0.1, 10s+1$. 
The proximal parameter is fixed as its default $\eta=\eta_{\mathrm{ABB}}$.
We present all the numerical results in Figure \ref{fig:beta}. 
We notice from subfigures (a)-(d) that PLAM with small $\beta$ might be divergent in some cases,
while large $\beta$ causes slow convergence. Therefore, a suitable chosen
$\beta$, often unreachable in practice, is necessary for good performance of PLAM.  
On the other hand, the dependence on $\beta$ of PCAL can be learnt from subfigures (e)-(h). 
The smaller $\beta$ for PCAL has
the better performance in some instances, and the behavior of PCAL 
is completely not sensitive to $\beta$ in other instances.  
To take more distinctive look at the difference between PLAM and PCAL,
we present a comparison in Figure \ref{fig:PCAL_beta_sensitive}.
Therefore, in practice, we suggest an approximation of $s$ to be the default $\beta$
of PLAM and $1$ for PCAL. Since it is easier to tune $\beta$ for PCAL than PLAM, we
choose PCAL to be the default algorithm of ours in Subsection \ref{section: parallel}.
\begin{figure}[htbp]
	\centering
	\subfigure[Problem 1]
	{\includegraphics[width=0.245\textwidth,height=0.2\textwidth]
		{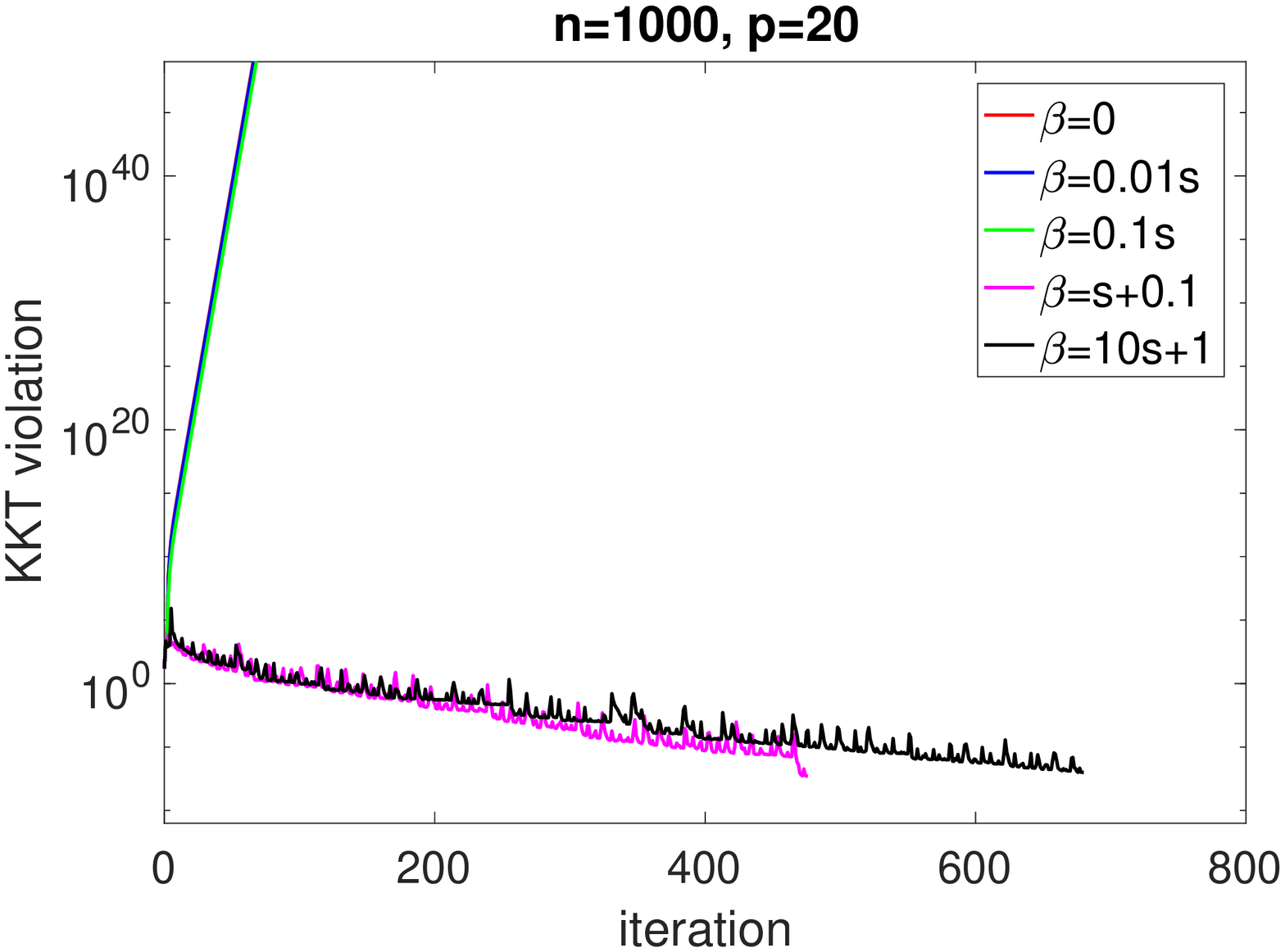}}
	\hfill
	\subfigure[Problem 2]
	{\includegraphics[width=0.245\textwidth,height=0.2\textwidth]
		{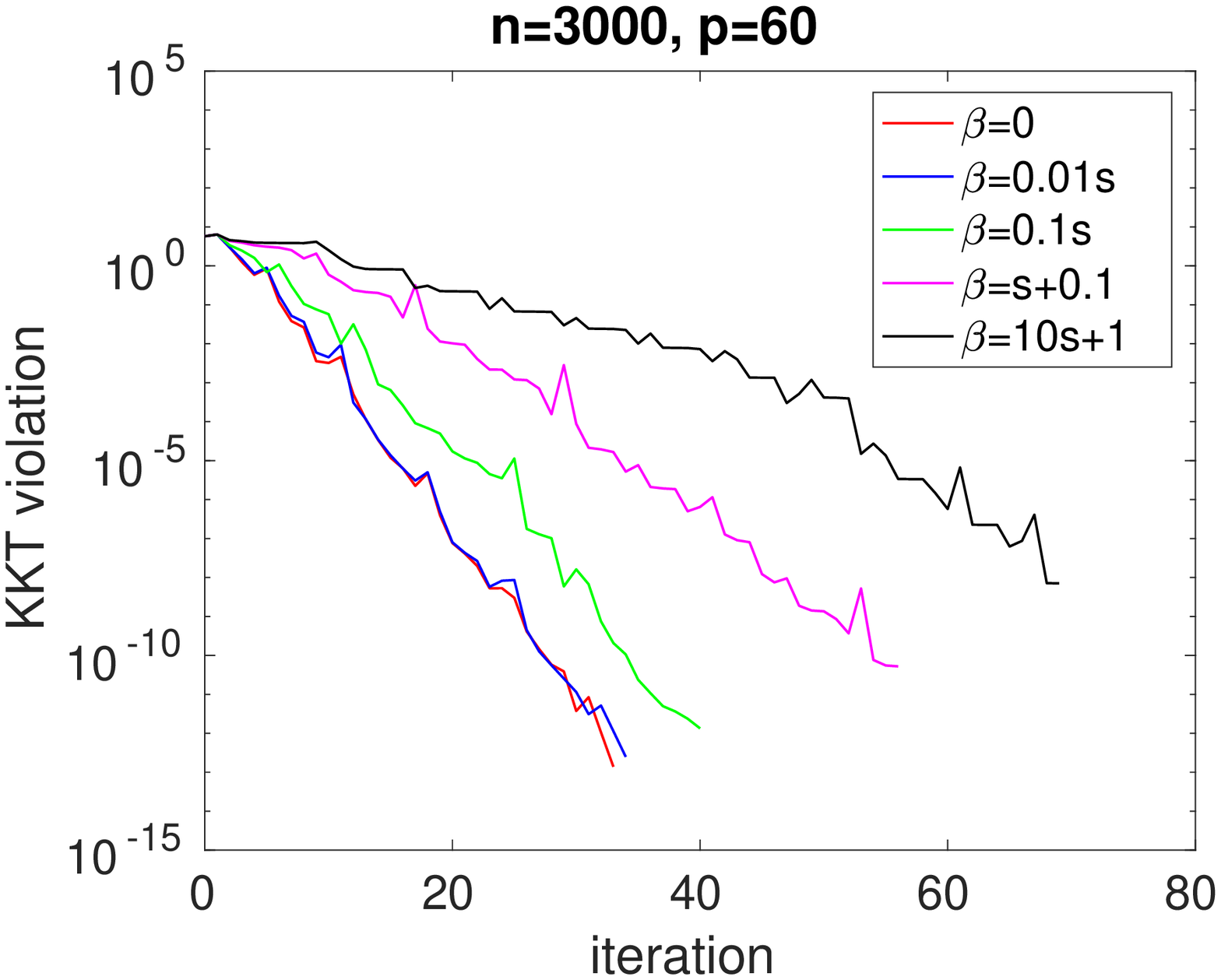}}
	\hfill
	\subfigure[Problem 3]
	{\includegraphics[width=0.245\textwidth,height=0.2\textwidth]
		{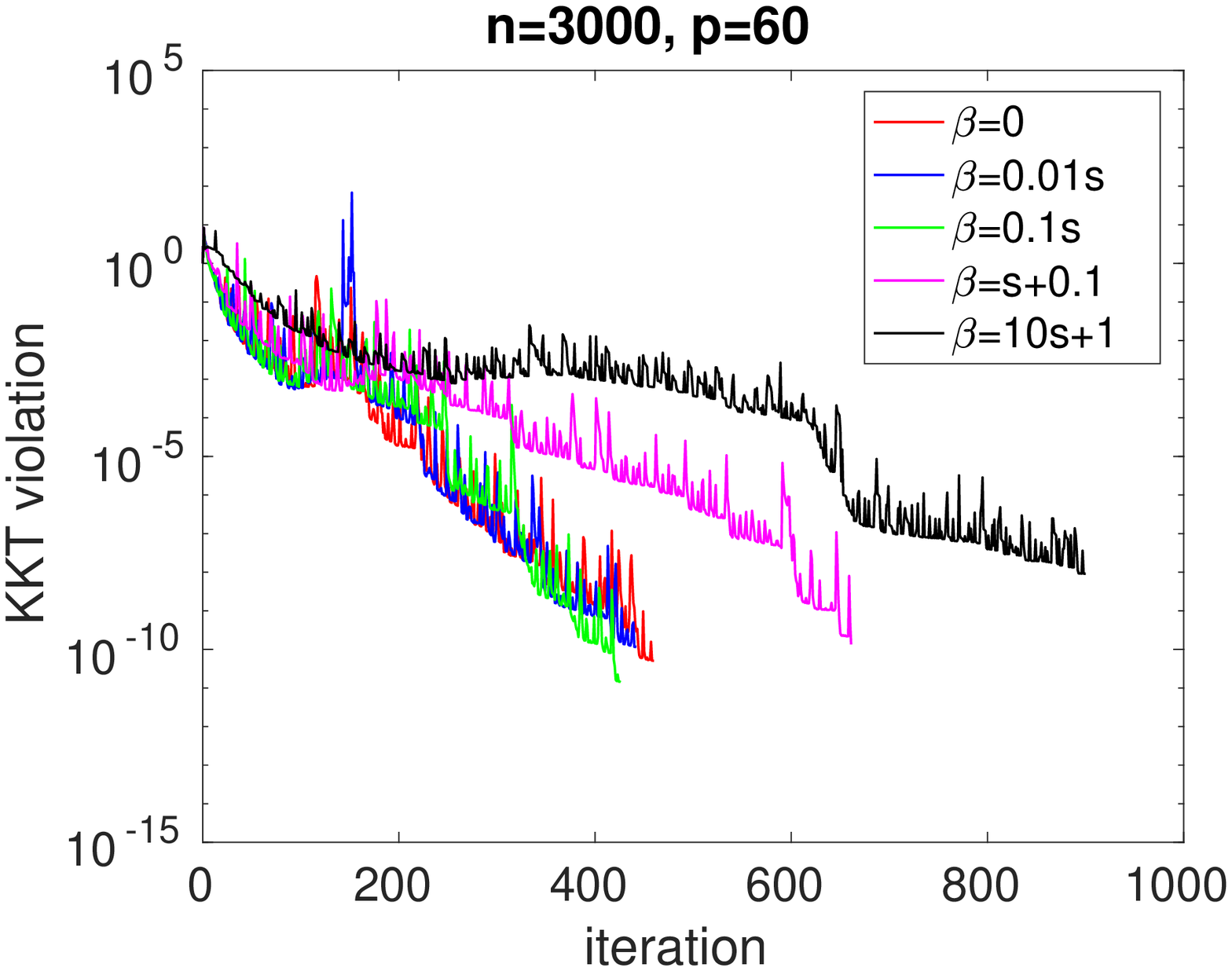}}
	\hfill
	\subfigure[Problem 4]
	{\includegraphics[width=0.245\textwidth,height=0.2\textwidth]
		{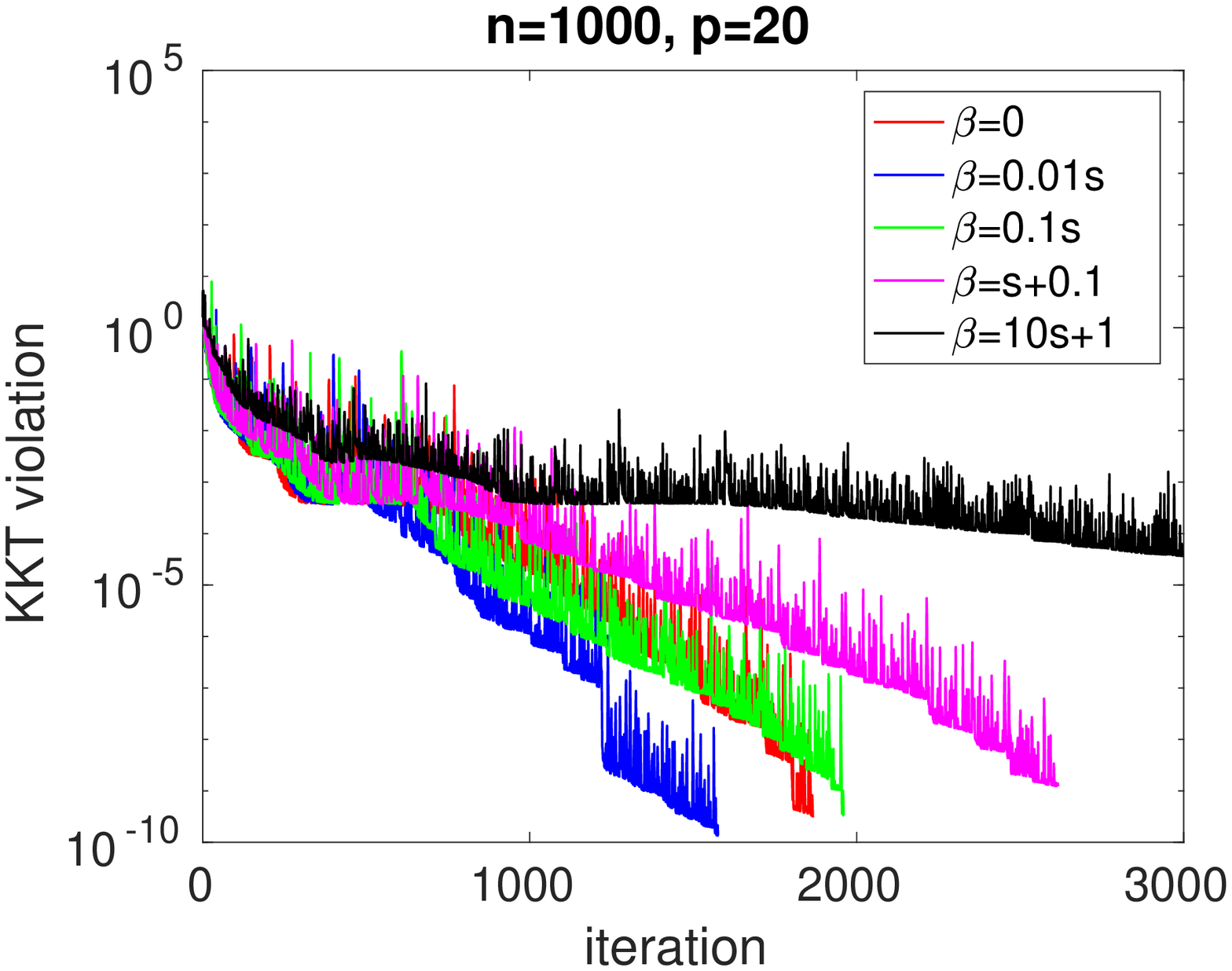}}
	
	\subfigure[Problem 1]
	{\includegraphics[width=0.245\textwidth,height=0.2\textwidth]
		{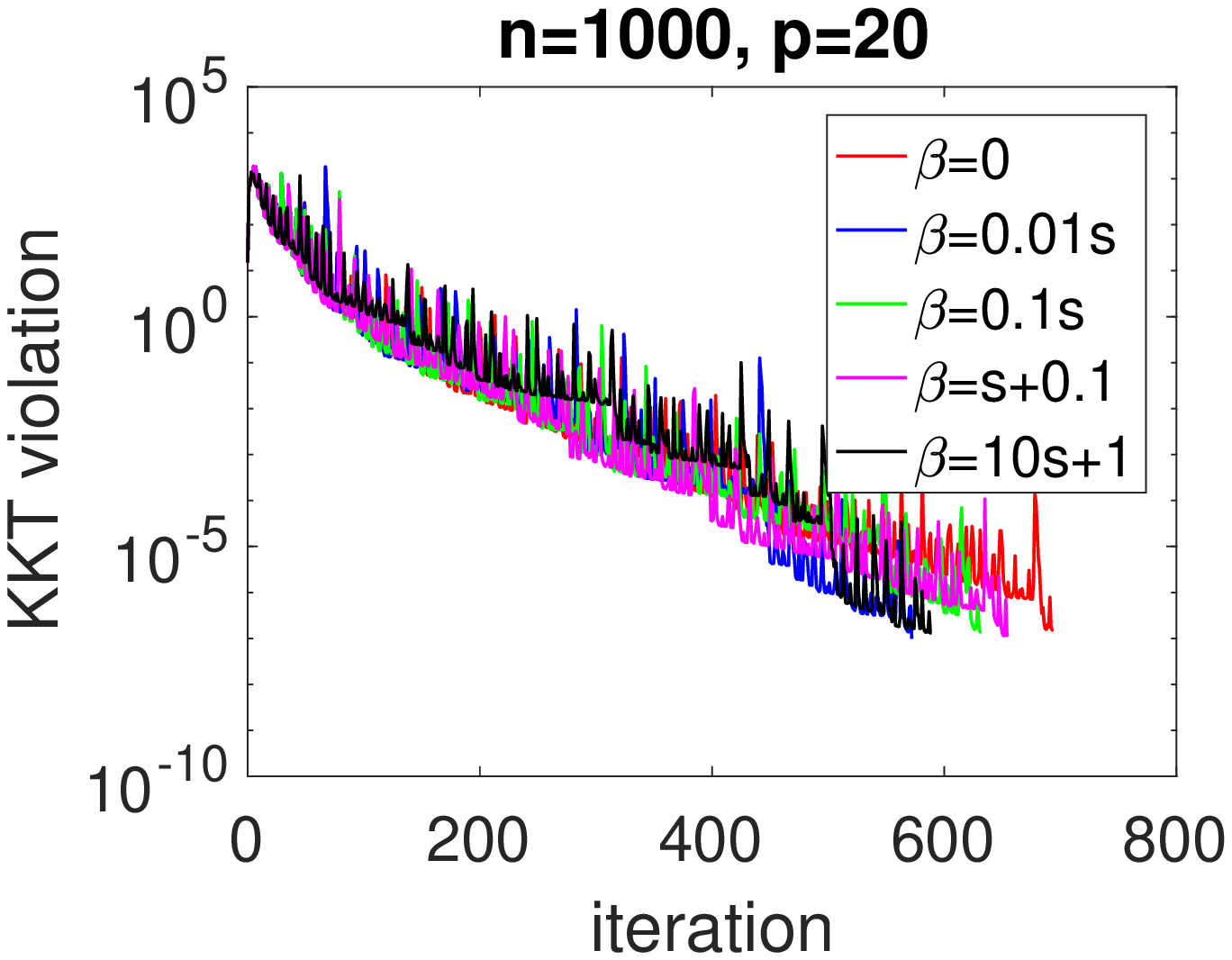}}
	\hfill
	\subfigure[Problem 2]
	{\includegraphics[width=0.245\textwidth,height=0.2\textwidth]
		{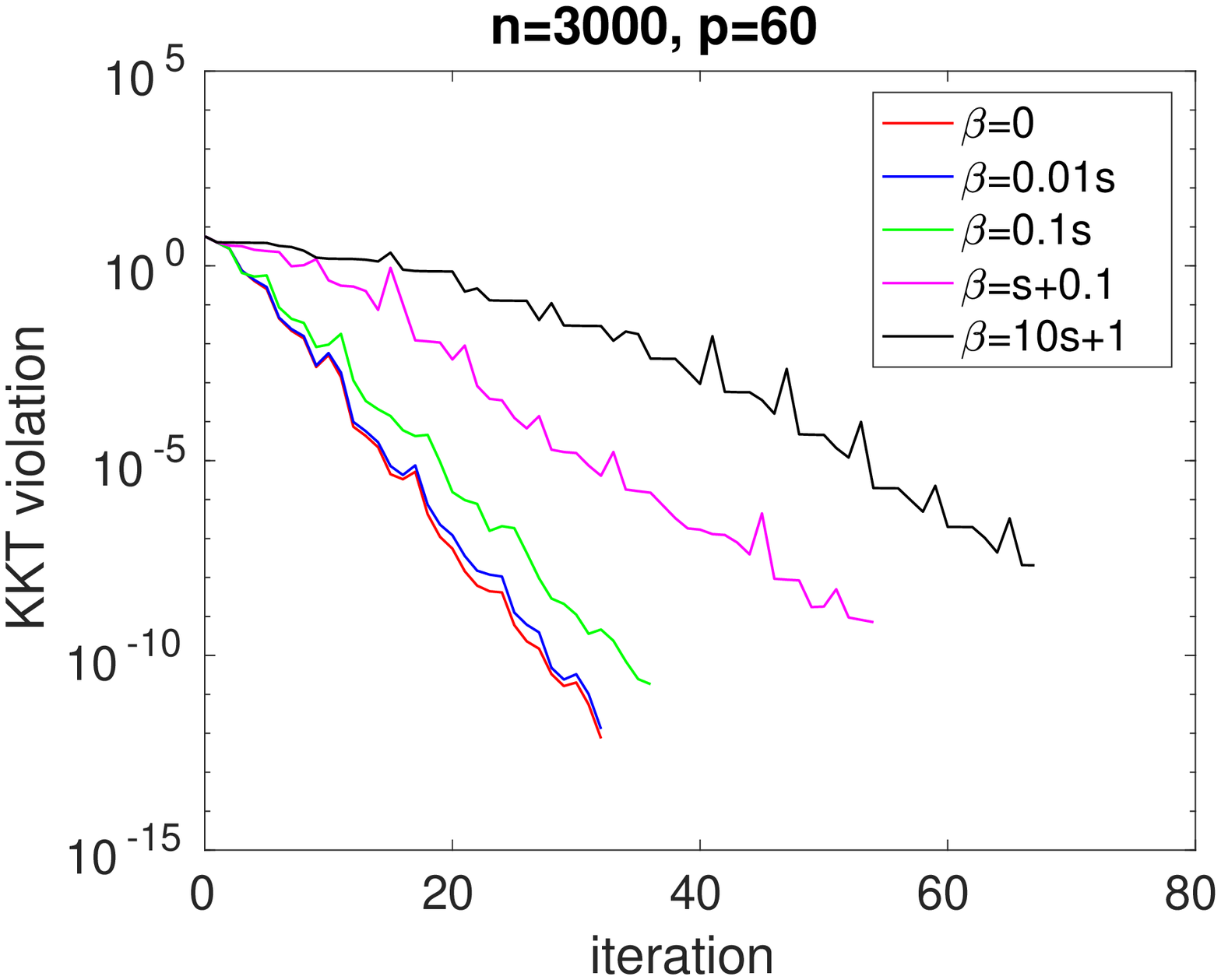}}
	\hfill
	\subfigure[Problem 3]
	{\includegraphics[width=0.245\textwidth,height=0.2\textwidth]
		{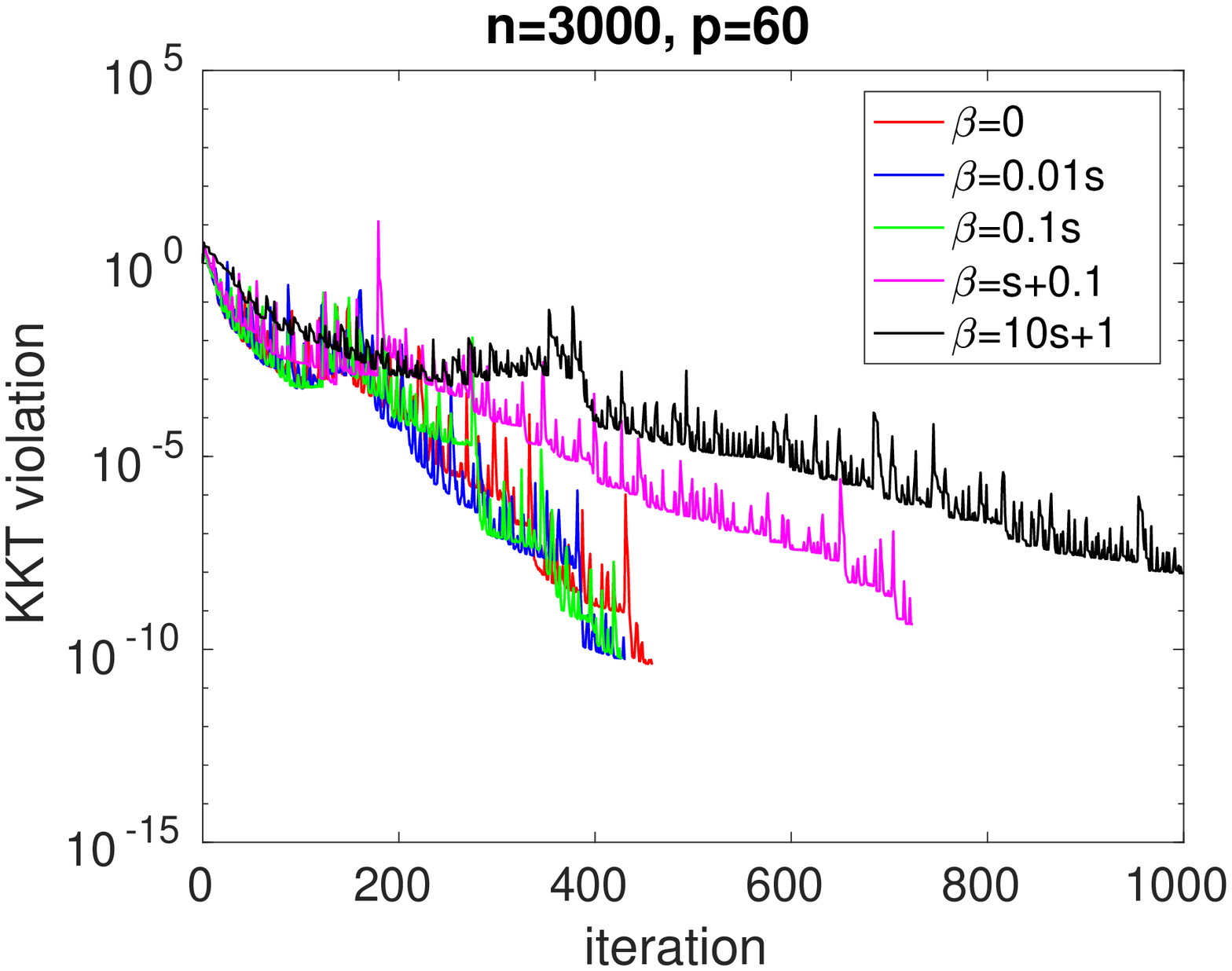}}
	\hfill
	\subfigure[Problem 4]
	{\includegraphics[width=0.245\textwidth,height=0.2\textwidth]
		{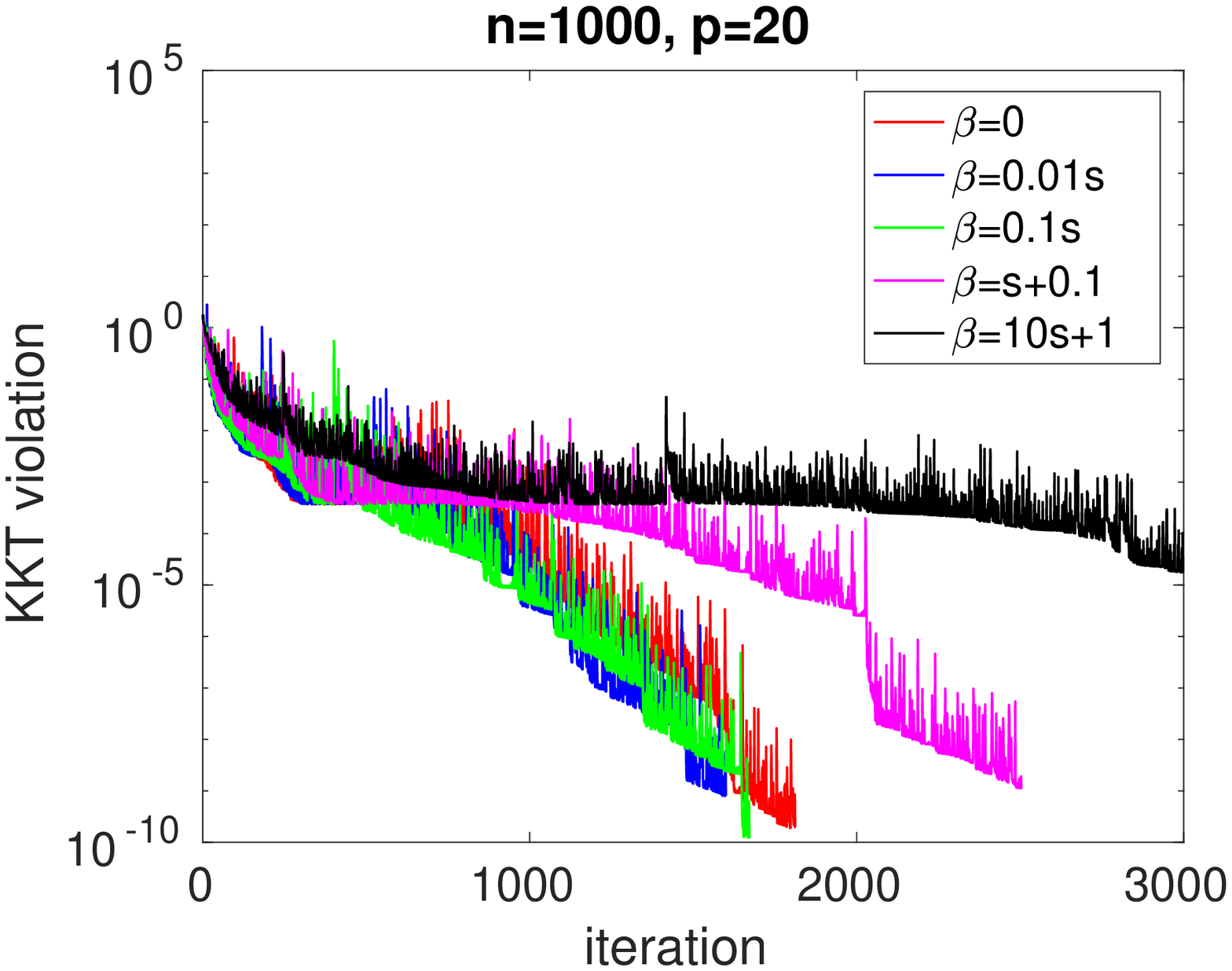}}
	\caption{A comparison of KKT violation for PLAM (a)-(d) and PCAL (e)-(h) with different $\beta$ ($\eta=\eta_{\mathrm{ABB}}$)}
	\label{fig:beta}
\end{figure}

\begin{figure}[htbp]
	\centering
	\subfigure[$\beta = 0.01s$]
	{\includegraphics[width=0.32\textwidth,height=0.25\textwidth]
		{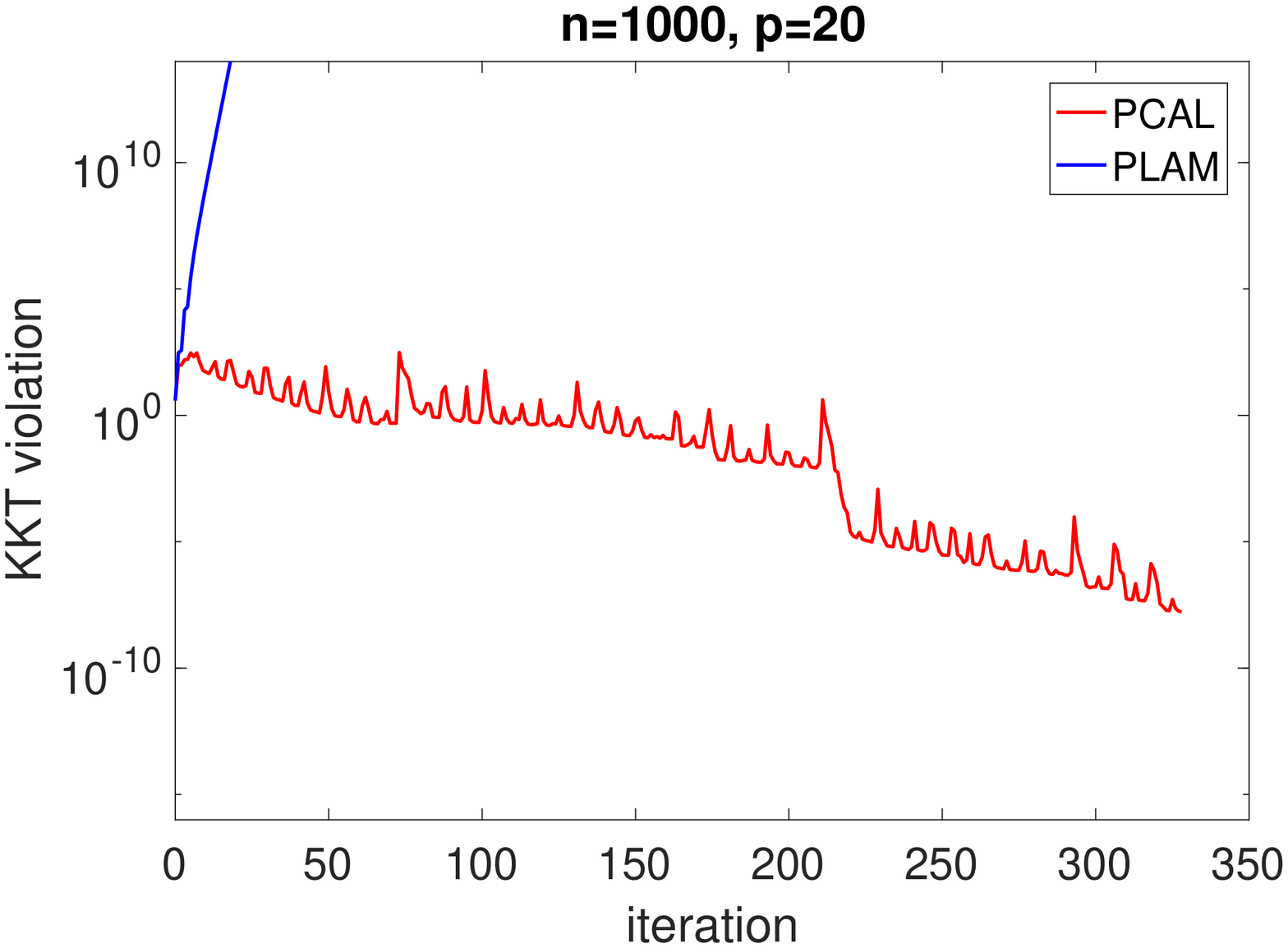}}
	\hfill
	\subfigure[$\beta = s+0.1$]
	{\includegraphics[width=0.32\textwidth,height=0.25\textwidth]
		{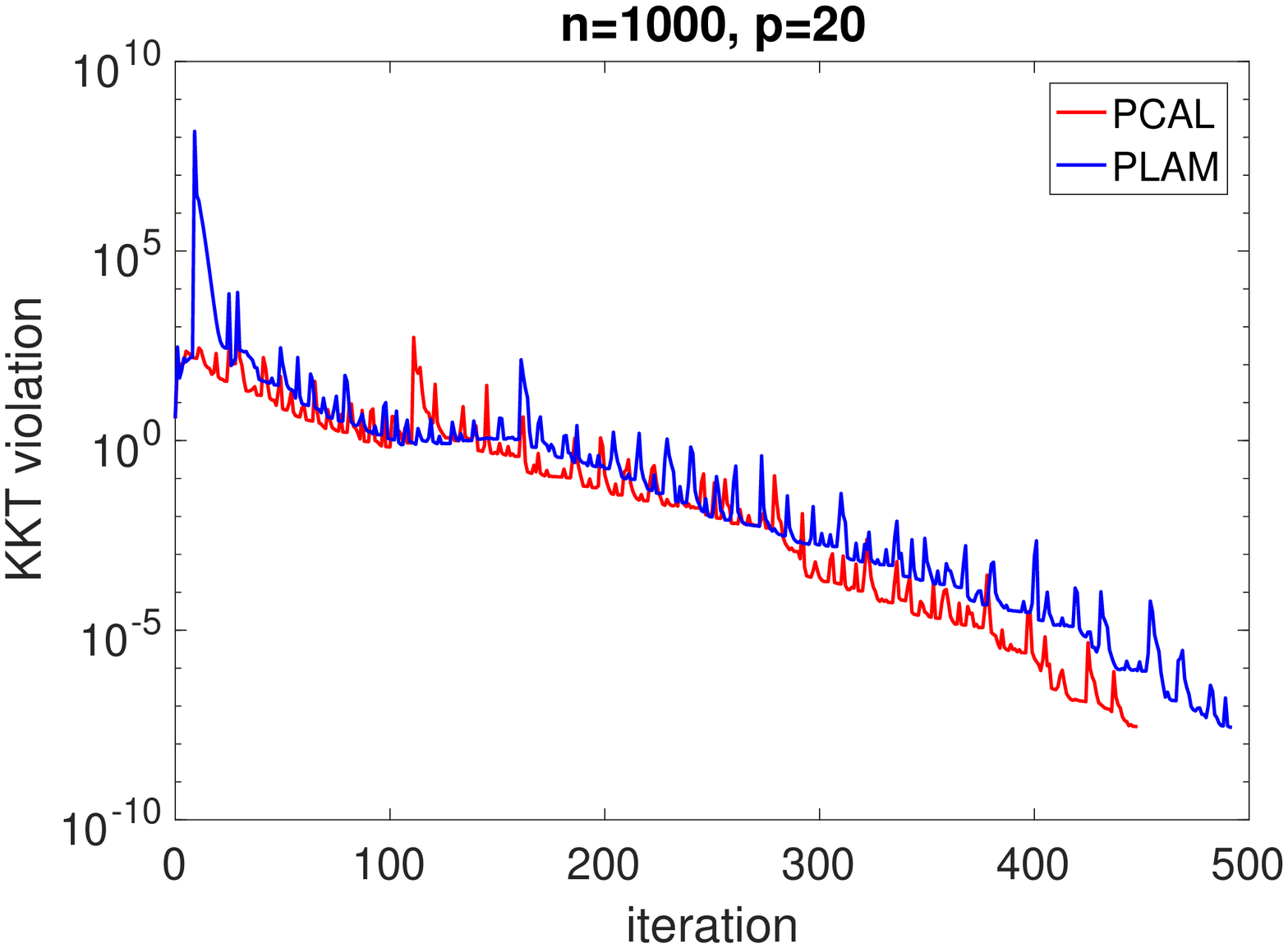}}
	\hfill
	\subfigure[$\beta = 10s+1$]
	{\includegraphics[width=0.32\textwidth,height=0.25\textwidth]
		{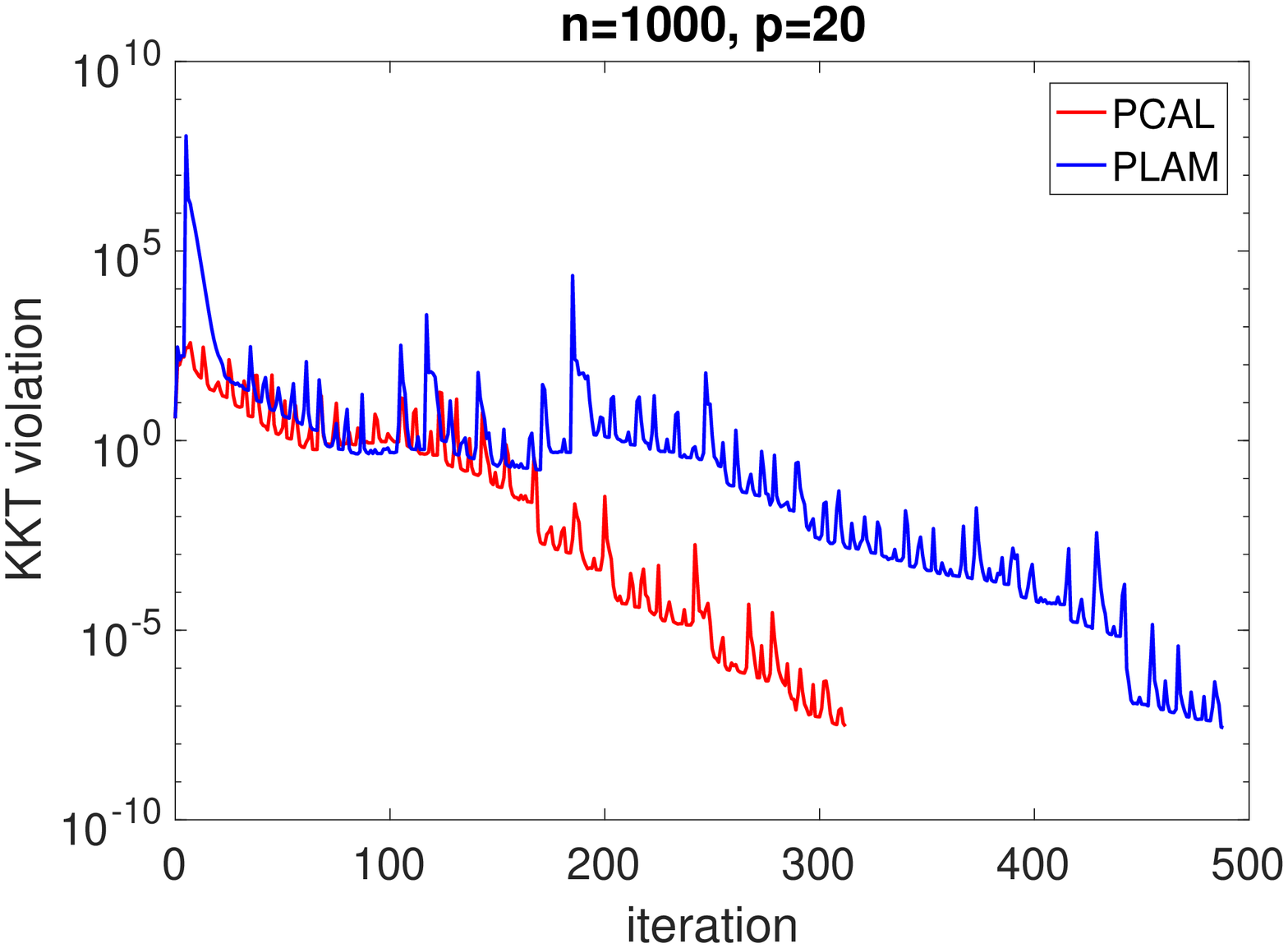}}
	\hfill
	
	\caption{A comparison between PLAM and PCAL with different $\beta$ on Problem 1}
	\label{fig:PCAL_beta_sensitive}
\end{figure}


There are two distinctions between PLAM 
and ALM. Firstly, a gradient step takes the place of 
solving the subproblem to some given precision in the update of the prime variables.
Secondly, a closed-form expression is used to update the Lagrangian multipliers in stead
of dual ascend. In order to show that the new update formula for multipliers 
is a crucial fact of the efficiency of PLAM and PCAL, we compare PLAM and PCAL
with PLAM-DA and PCAL-DA, respectively. Here PLAM-DA and PCAL-DA stand for Algorithm \ref{alg:PLAM} and 
\ref{alg:PCAL} with Step 3 using dual ascend to update the multipliers, respectively.
We report the numerical results in Figure \ref{fig:PCAL_multiplier}.
It can be observed that the closed-form expression for updating
Lagrangian multipliers is superior to dual ascend in solving optimization 
problems with orthogonality constraints.
\begin{figure}[htb]
	\centering
	\subfigure[Problem 1]
	{\includegraphics[scale=.30]
		{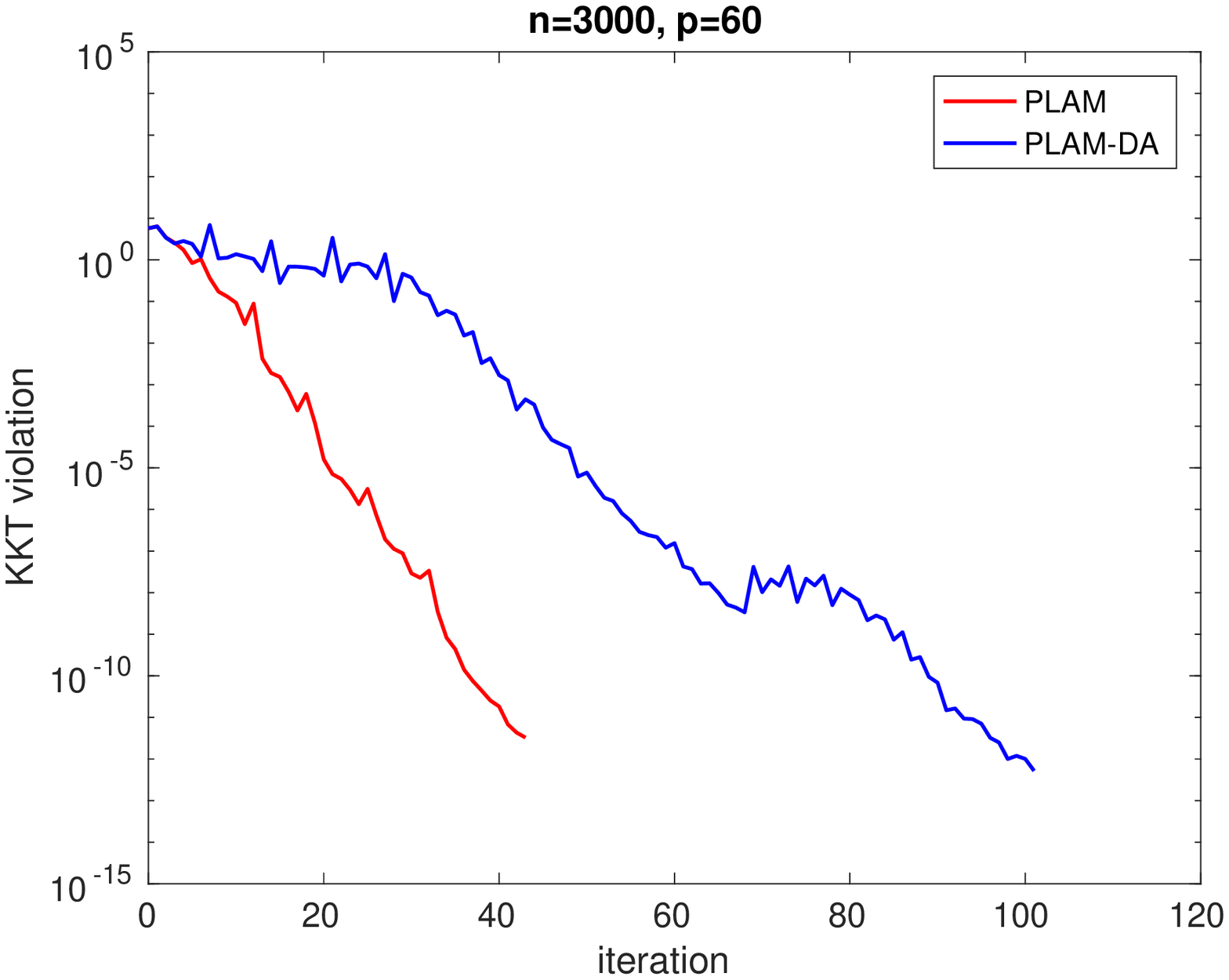}}
	\qquad
	\subfigure[Problem 2]
	{\includegraphics[scale=.30]
		{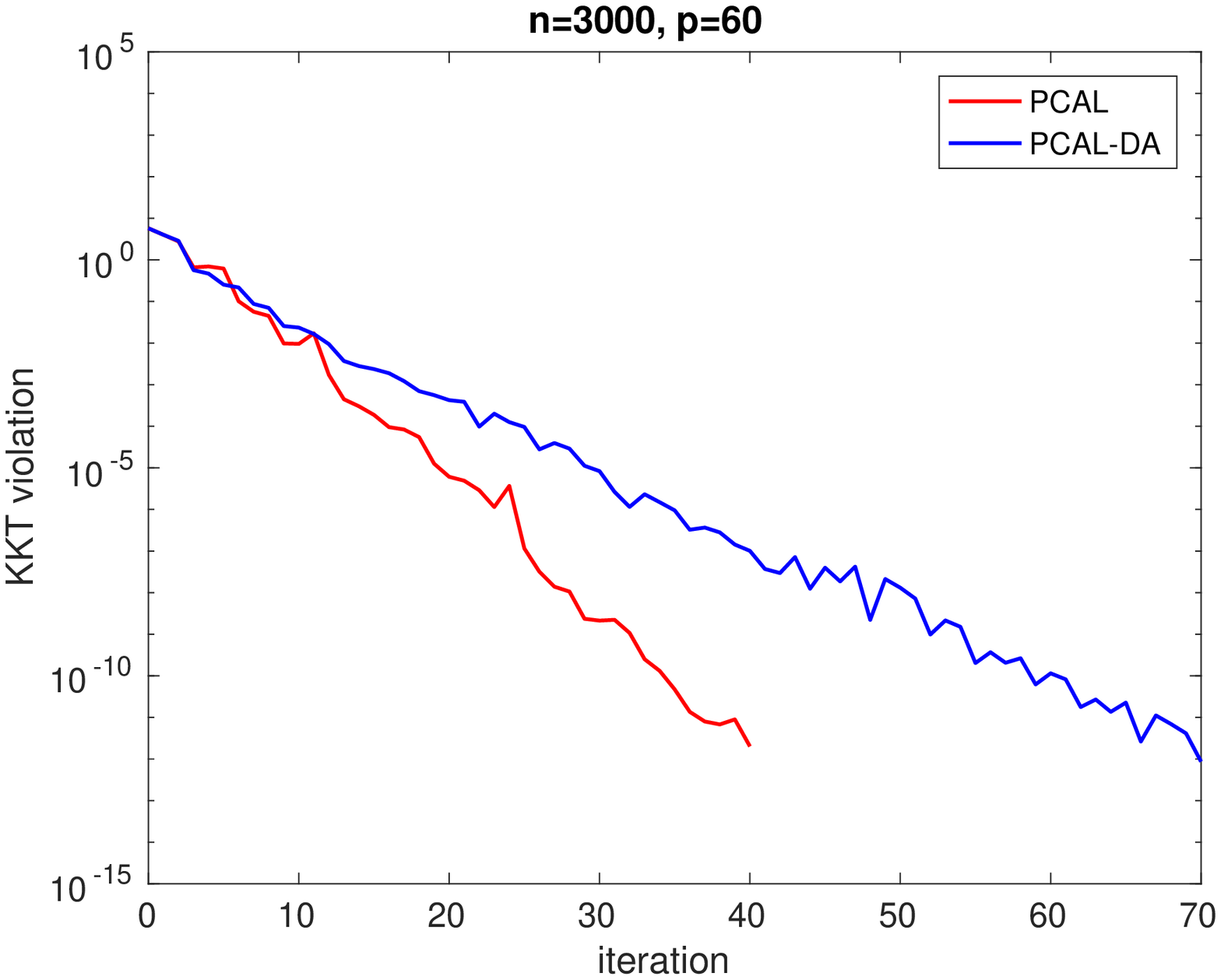}}
	\caption{A comparion bewteen PLAM and PCAL on multilplier}
	\label{fig:PCAL_multiplier}
\end{figure}

In the end of this subsection, we show how KKT and feasibility violations decay 
in the iterations, when PLAM and PCAL are used to 
solve Problem 1. The numerical results are presented in Figure \ref{fig:feasibility}.
We notice that the decay of feasibility violations is nonmonotone and
has a similar variation tendency as KKT violations, which 
coincides our theoretical analysis Lemma \ref{lm:2}. 
If we want a high accuracy for the feasibility
but a mild one for KKT conditions, we can set a mild tolerance for 
KKT violation and impose an orthonormalization step as a post process
when we obtain the last iterate by PLAM or PCAL.
Table \ref{tab:orth step} illustrates that such post process 
does not affect the KKT violation, but do improve the feasibility.
Here, "stop" and "orth" represent the relative values at the last iterate 
and the one after post process, respectively. 
Hereinafter, the orthonormalization post process, 
achieved by an internal function {\ttfamily qr$(\cdot)$} assembled in MATLAB,
 is the default last step of PLAM and PCAL.
\begin{figure}[htb]
	\centering
	\subfigure[KKT violation]
	{\includegraphics[scale=.30]
		{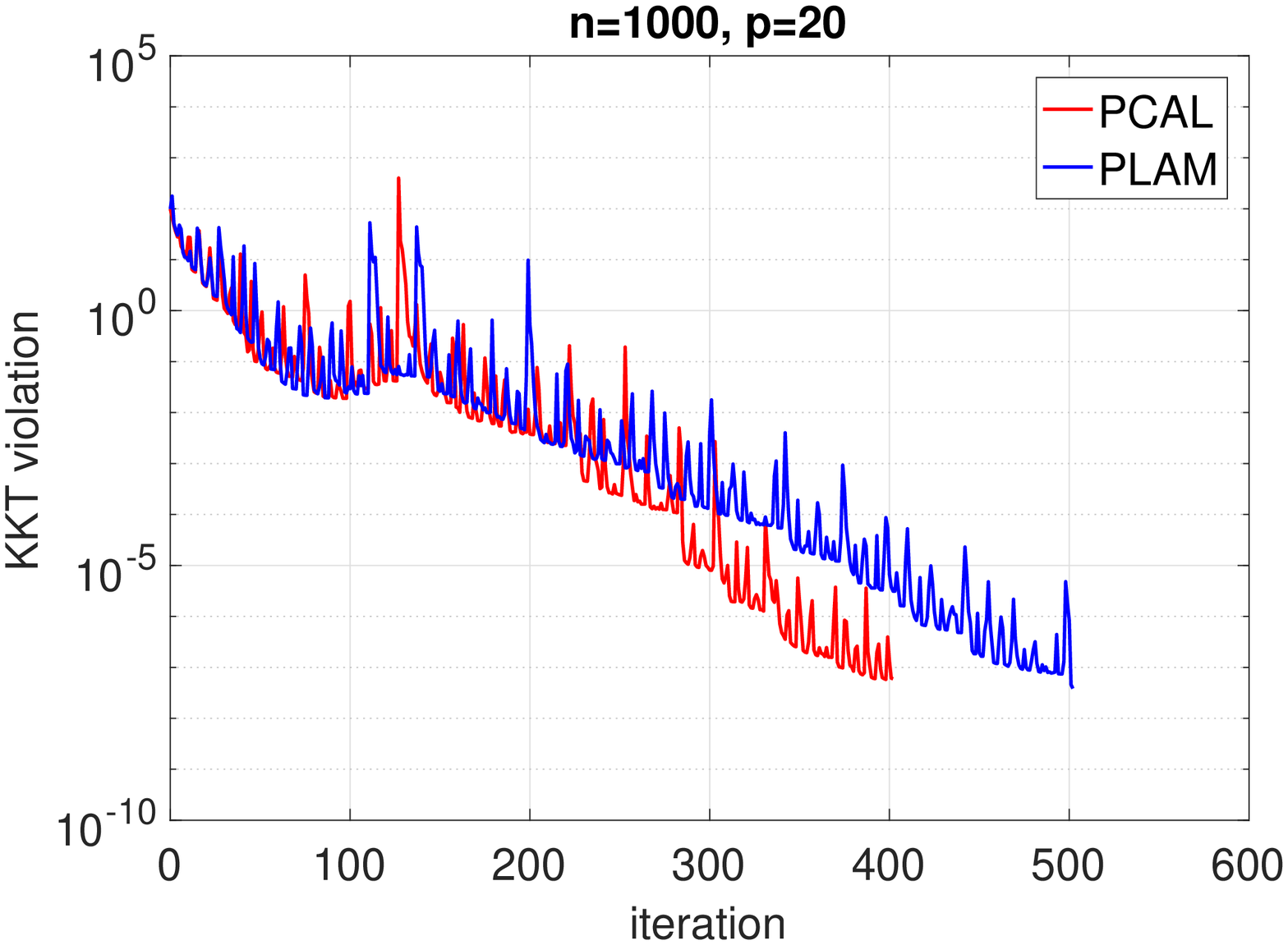}}
	\qquad
	\subfigure[Feasibility violation]
	{\includegraphics[scale=.30]
		{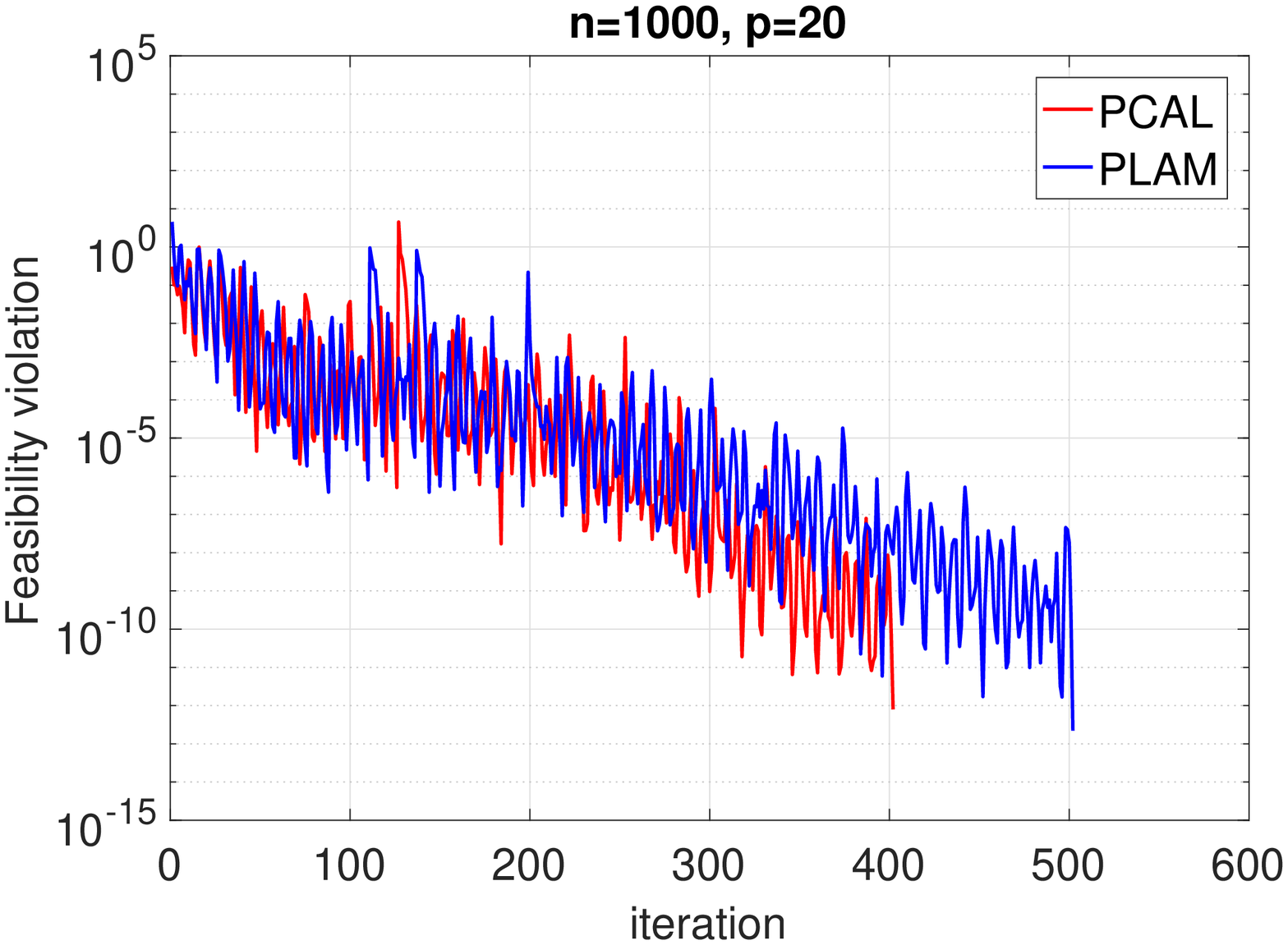}}
	\caption{The results of KKT and Feasibility violation for PLAM and PCAL on Problem 1} 
	\label{fig:feasibility}
\end{figure}

\begin{table}[htbp]
	\centering
	\begin{tabular}{c|c||ccc}
		\hline
		\toprule[.3mm]
		\multicolumn{2}{c}{Solver} &   {Function value}  &   {KKT violation}   & {Feasibility violation}    \\
		\midrule[.2mm]
		\multicolumn{1}{c}{} &\multicolumn{1}{c}{}& \multicolumn{3}{c}{$n=1000$, $p=20$, $\alpha=1$} \\\midrule

		\multirow{2}{*}{PLAM} & stop & -4.205530767124e+02 & 8.74e-06 & 2.56e-09  \\								
		& orth & -4.205530767662e+02 &  8.74e-06 & 5.61e-15  \\\midrule

		\multirow{2}{*}{PCAL} & stop & -4.205530767773e+02 & 6.01e-06 & 1.13e-08  \\				
		& orth & -4.205530767665e+02 &  6.00e-06 & 2.00e-14
		\\\bottomrule[.3mm]
	\end{tabular}
	\caption{The results of orthogonal step for PLAM and PCAL on Problem 1\label{tab:orth step}} 
\end{table}

\subsection{Kohn-Sham Total Energy Minimization\label{subsection:KSDFT}}
In this subsection, we compare PLAM and PCAL with the state-of-the-art solvers in solving Kohn-Sham total 
energy minimization \eqref{eq:KSDFT} in serial.
In other word, we aim to investigate the numerical performance of two proposed infeasible 
algorithms as general solvers for optimization problems with orthogonality constraints
without consideration of parallelization. 

Our test is based on KSSOLV\footnote{Available from 
\href{http://crd-legacy.lbl.gov/~chao/KSSOLV/}{http://crd-legacy.lbl.gov/$\sim$chao/KSSOLV/}} 
\cite{Yang09}, which is a MATLAB toolbox for electronic structure calculation. It allows researchers to 
investigate their own algorithms easily and friendly for different steps in electronic structure 
calculation. We choose two integrated solvers in KSSOLV.
One is the self-consistent field (SCF) iteration, which
minimize a quadratic surrogate of the objective of \eqref{eq:KSDFT}
with orthogonality constraints in each iteration \cite{Liu2014}.
SCF and its variations are the most widely used in real KSDFT calculation.
The other one is called trust-region direct constrained minimization (TRDCM)
\cite{Yang2007}, which combines the trust-region framework
and SCF to solve the subproblem.
Besides SCF and TRDCM,
which are particularly for KSDFT, 
we also pick up two state-of-the-art solvers in solving general optimization
problems with orthogonality constraints. 
One is OptM\footnote{Available from \href{http://optman.blogs.rice.edu}{http://optman.blogs.rice.edu}}, 
which is based on the algorithm proposed in \cite{WenYin2013}. OptM adopts
Cayley transform to preserve the feasibility on the Stiefel manifold in each iteration. 
Nonmonotone line search with BB stepsize is the default setting in OptM. 
Another existing solver for comparison, we intend to choose MOptQR,
which is based on a projection-like retraction method introduced in \cite{Abisil2008}. 
Its original version is MOptQR-LS (manifold QR method with line search\footnote{Available from 
	\href{http://www.manopt.org}{http://www.manopt.org}}). 
For fair comparison, we implement the same alternating BB stepsize strategy as PLAM and PCAL to MOptQR-LS,
and form the MOptQR used in this section.

We select 18 testing problems with respect to different molecules, which are assembled in KSSOLV. 
For all the methods, the stopping criterion is set as $\norm{(I_n-XX\zz)H(X)X}\ff<10^{-5}$. And we set the 
max iteration number $\text{MaxIter}=200$ for methods SCF and TRDCM, while MOptQR, OptM, PLAM and PCAL set 
their max iteration number with $\text{MaxIter}=1000$ to get a comparable solution with other methods. 
The penalty parameter $\beta_{\text{PLAM}}$ for PLAM is tuned case by case to achieve a good performance. 
Meanwhile, $\beta_{\text{PCAL}}$ for PCAL is always set as $1$.
Other parameters for all these methods take their default values. 
For all of the testing algorithms, we set the same initial guess $X^0$ by using the function 
``{\ttfamily getX0}", which is provided by KSSOLV. The numerical results are illustrated 
in Tables \ref{tab:KS}, \ref{tab:KS1} and \ref{tab:KS2}. 

Here, ``$E_{tot}$" represents the total energy function value, and ``KKT violation", ``Iteration", 
``Feasibility violation" and ``Time(s)" stand for $\norm{(I_n-XX\zz)H(X)X}\ff$, the number of iteration, 
$\norm{X\zz X -I_p}\ff$ and the total running wall-clock time in second, respectively. 
From the tables, we observe that PCAL has a better performance than other algorithms, and in most cases, it 
obtains a comparable total energy function value and a lower KKT violation. 
In particular, in the large size problem ``graphene30", PCAL achieves the same total energy function value 
and same magnitude KKT violation in much less CPU time than others. In the problem ``qdot", we observe that 
only PLAM and PCAL can output a point satisfying the KKT violation tolerance,
while all the other algorithms terminate abnormally.
Therefore, we can conclude that PCAL and PLAM perform comparable with the existent 
feasible algorithms in solving discretized Kohn-Sham total energy minimization.	

\begin{table}[htbp]
	\centering
	\begin{tabular}{c||ccccc}
		\hline
		\toprule[.4mm]
		Solver	&   {$E_{tot}$}  &   {KKT violation}   & {Iteration}  & {Feasibility violation} & {Time(s)}   \\
		\midrule[.3mm]
		
		\multicolumn{1}{c}{} & \multicolumn{5}{c}{al, $n=16879$, $p=12$ \qquad\qquad($\beta_{\text{PLAM}}=10, \beta_{\text{PCAL}}=1$)} \\\midrule 
		SCF 	 & 	 -1.5789379003e+01 & 	 4.88e-03 & 	 200  & 	 6.53e-15 & 	 539.51 \\ 
		TRDCM 	 & 	 -1.5803791151e+01 & 	 6.36e-06 & 	 154  & 	 4.94e-15 & 	 336.79 \\ 
		MOptQR 	 & 	 -1.5803814080e+01 & 	 1.88e-04 & 	 1000  & 	 1.33e-14 & 	 393.54 \\ 
		OptM 	 & 	 -1.5803791098e+01 & 	 2.38e-05 & 	 1000  & 	 3.19e-14 & 	 378.80 \\
		PLAM 	 & 	 -1.5803790675e+01 & 	 1.29e-05 & 	 1000  & 	 3.34e-07 & 	 399.80 \\  
		PCAL 	 & 	 -1.5803791055e+01 & 	 8.96e-06 & 	 596  & 	 5.95e-15 & 	 228.06 \\  
		\midrule   
		
		\multicolumn{1}{c}{} & \multicolumn{5}{c}{alanine, $n=12671$, $p=18$\qquad\qquad($\beta_{\text{PLAM}}=13, \beta_{\text{PCAL}}=1$)} \\\midrule 
		SCF 		 & 	 -6.1161921212e+01 & 	 3.80e-07 & 	 13  & 	 7.20e-15 & 	 21.46 \\ 
		TRDCM 	 & 	 -6.1161921213e+01 & 	 6.02e-06 & 	 15  & 	 5.20e-15 & 	 16.97 \\ 
		MOptQR 	 & 	 -6.1161921213e+01 & 	 7.52e-06 & 	 64  & 	 6.77e-15 & 	 14.89 \\ 
		OptM 	 & 	 -6.1161921213e+01 & 	 2.27e-06 & 	 69  & 	 4.03e-14 & 	 16.44 \\ 
		PLAM 	 & 	 -6.1161921212e+01 & 	 9.50e-06 & 	 76  & 	 7.90e-15 & 	 17.14 \\  
		PCAL 	 & 	 -6.1161921213e+01 & 	 4.14e-06 & 	 61  & 	 7.19e-15 & 	 15.89 \\\midrule 
		
		\multicolumn{1}{c}{} & \multicolumn{5}{c}{benzene, $n=8407$, $p=15$\qquad\qquad($\beta_{\text{PLAM}}=10, \beta_{\text{PCAL}}=1$)} \\\midrule 
		SCF 		 & 	 -3.7225751349e+01 & 	 2.10e-07 & 	 10  & 	 7.82e-15 & 	 10.07 \\ 
		TRDCM 	 & 	 -3.7225751363e+01 & 	 9.23e-06 & 	 15  & 	 7.12e-15 & 	 9.83 \\ 
		MOptQR 	 & 	 -3.7225751362e+01 & 	 8.12e-06 & 	 146  & 	 7.24e-15 & 	 19.91 \\ 
		OptM 	 & 	 -3.7225751363e+01 & 	 2.50e-06 & 	 70  & 	 1.54e-14 & 	 9.61 \\ 
		PLAM 	 & 	 -3.7225751362e+01 & 	 9.37e-06 & 	 71  & 	 4.62e-15 & 	 9.55 \\  
		PCAL 	 & 	 -3.7225751362e+01 & 	 9.22e-06 & 	 50  & 	 5.15e-15 & 	 7.74 \\\midrule   
		
		\multicolumn{1}{c}{} & \multicolumn{5}{c}{c2h6, $n=2103$, $p=7$\qquad\qquad($\beta_{\text{PLAM}}=10, \beta_{\text{PCAL}}=1$)} \\\midrule 
		SCF 	 & 	 -1.4420491315e+01 & 	 3.70e-09 & 	 10  & 	 3.66e-15 & 	 3.40 \\ 
		TRDCM 	 & 	 -1.4420491322e+01 & 	 8.75e-06 & 	 13  & 	 2.76e-15 & 	 4.01 \\ 
		MOptQR 	 & 	 -1.4420491321e+01 & 	 8.59e-06 & 	 47  & 	 2.58e-15 & 	 2.57 \\ 
		OptM 	 & 	 -1.4420491322e+01 & 	 2.62e-06 & 	 55  & 	 1.18e-14 & 	 2.87 \\ 
		PLAM 	 & 	 -1.4420491322e+01 & 	 7.91e-06 & 	 69  & 	 2.92e-15 & 	 3.41 \\
		PCAL 	 & 	 -1.4420491322e+01 & 	 4.91e-06 & 	 45  & 	 2.33e-15 & 	 2.58 \\  
		\midrule   
		
		\multicolumn{1}{c}{} & \multicolumn{5}{c}{c12h26, $n=5709$, $p=37$\qquad\qquad($\beta_{\text{PLAM}}=10, \beta_{\text{PCAL}}=1$)} \\\midrule 
		SCF 	 & 	 -8.1536091894e+01 & 	 4.95e-08 & 	 14  & 	 1.40e-14 & 	 30.08 \\ 
		TRDCM 	 & 	 -8.1536091937e+01 & 	 4.84e-06 & 	 16  & 	 1.17e-14 & 	 21.77 \\ 
		MOptQR 	 & 	 -8.1536091936e+01 & 	 6.68e-06 & 	 147  & 	 1.43e-14 & 	 39.57 \\ 
		OptM 	 & 	 -8.1536091937e+01 & 	 1.07e-06 & 	 83  & 	 7.10e-14 & 	 22.65 \\
		PLAM 	 & 	 -8.1536091936e+01 & 	 5.88e-06 & 	 96  & 	 1.55e-14 & 	 25.11 \\
		PCAL 	 & 	 -8.1536091936e+01 & 	 8.75e-06 & 	 70  & 	 1.45e-14 & 	 22.88 \\  
		\midrule  
		
		\multicolumn{1}{c}{} & \multicolumn{5}{c}{co2, $n=2103$, $p=8$\qquad\qquad($\beta_{\text{PLAM}}=10, \beta_{\text{PCAL}}=1$)} \\\midrule 
		SCF 		 & 	 -3.5124395789e+01 & 	 6.17e-08 & 	 10  & 	 2.53e-15 & 	 2.61 \\ 
		TRDCM 	 & 	 -3.5124395801e+01 & 	 4.14e-06 & 	 14  & 	 4.11e-15 & 	 2.09 \\ 
		MOptQR 	 & 	 -3.5124395800e+01 & 	 9.30e-06 & 	 88  & 	 2.35e-15 & 	 2.90 \\ 
		OptM 	 & 	 -3.5124395801e+01 & 	 1.70e-06 & 	 48  & 	 3.55e-14 & 	 1.68 \\ 
		PLAM 	 & 	 -3.5124395801e+01 & 	 7.92e-06 & 	 57  & 	 2.30e-15 & 	 1.84 \\  
		PCAL 	 & 	 -3.5124395801e+01 & 	 9.15e-06 & 	 43  & 	 2.11e-15 & 	 1.74 \\
		\bottomrule[.4mm]
	\end{tabular}
	\caption{The results in Kohn-Sham total energy minimization\label{tab:KS}} 
\end{table} 

\begin{table}[htbp]
	\centering
	\begin{tabular}{c||ccccc}
		\hline
		\toprule[.4mm]
		Solver	&   {$E_{tot}$}  &   {KKT violation}   & {Iteration}  &{Feasibility violation} & {Time(s)}   \\
		\midrule[.3mm]
		
		\multicolumn{1}{c}{} & \multicolumn{5}{c}{ctube661, $n=12599$, $p=48$\qquad\qquad($\beta_{\text{PLAM}}=13, \beta_{\text{PCAL}}=1$)} \\\midrule 
		SCF 	 & 	 -1.3463843175e+02 & 	 3.88e-07 & 	 11  & 	 1.43e-14 & 	 56.43 \\ 
		TRDCM 	 & 	 -1.3463843176e+02 & 	 6.85e-06 & 	 23  & 	 1.09e-14 & 	 87.41 \\ 
		MOptQR 	 & 	 -1.3463843176e+02 & 	 7.21e-06 & 	 152  & 	 1.78e-14 & 	 107.62 \\ 
		OptM 	 & 	 -1.3463843176e+02 & 	 2.35e-06 & 	 82  & 	 2.15e-14 & 	 59.23 \\ 
		PLAM 	 & 	 -1.3463843176e+02 & 	 4.34e-06 & 	 107  & 	 2.37e-14 & 	 72.18 \\
		PCAL 	 & 	 -1.3463843176e+02 & 	 9.68e-06 & 	 65  & 	 1.95e-14 & 	 54.07 \\  
		\midrule 
		
		\multicolumn{1}{c}{} & \multicolumn{5}{c}{glutamine, $n=16517$, $p=29$\qquad\qquad($\beta_{\text{PLAM}}=13, \beta_{\text{PCAL}}=1$)} \\\midrule 
		SCF 		 & 	 -9.1839425202e+01 & 	 1.12e-07 & 	 15  & 	 1.07e-14 & 	 67.40 \\ 
		TRDCM 	 & 	 -9.1839425244e+01 & 	 3.23e-06 & 	 16  & 	 7.00e-15 & 	 54.65 \\ 
		MOptQR 	 & 	 -9.1839425243e+01 & 	 9.83e-06 & 	 78  & 	 9.07e-15 & 	 51.46 \\ 
		OptM 	 & 	 -9.1839425244e+01 & 	 2.47e-06 & 	 87  & 	 9.73e-15 & 	 57.65 \\ 
		PLAM 	 & 	 -9.1839425243e+01 & 	 8.72e-06 & 	 104  & 	 9.26e-15 & 	 66.31 \\  
		PCAL 	 & 	 -9.1839425243e+01 & 	 6.28e-06 & 	 74  & 	 9.33e-15 & 	 53.53 \\\midrule  
		
		\multicolumn{1}{c}{} & \multicolumn{5}{c}{graphene16, $n=3071$, $p=37$\qquad\qquad($\beta_{\text{PLAM}}=10, \beta_{\text{PCAL}}=1$)} \\\midrule 
		SCF 	& 	 -9.4023322108e+01 & 	 2.07e-03 & 	 200  & 	 1.32e-14 & 	 309.33 \\ 
		TRDCM 	 & 	 -9.4046217545e+01 & 	 8.85e-06 & 	 45  & 	 1.08e-14 & 	 47.87 \\ 
		MOptQR 	 & 	 -9.4046217225e+01 & 	 9.90e-06 & 	 422  & 	 1.15e-14 & 	 80.67 \\ 
		OptM 	 & 	 -9.4046217545e+01 & 	 2.27e-06 & 	 245  & 	 1.03e-14 & 	 48.66 \\ 
		PLAM 	 & 	 -9.4046217854e+01 & 	 9.52e-06 & 	 278  & 	 1.34e-14 & 	 51.57 \\  
		PCAL 	 & 	 -9.4046217542e+01 & 	 8.68e-06 & 	 176  & 	 1.17e-14 & 	 41.11 \\\midrule   
		
		\multicolumn{1}{c}{} & \multicolumn{5}{c}{graphene30, $n=12279$, $p=67$\qquad\qquad($\beta_{\text{PLAM}}=13, \beta_{\text{PCAL}}=1$)} \\\midrule 
		SCF 	 & 	 -1.7358453985e+02 & 	 5.19e-03 & 	 200  & 	 1.93e-14 & 	 2815.79 \\ 
		TRDCM 	 & 	 -1.7359510506e+02 & 	 4.80e-06 & 	 71  & 	 1.42e-14 & 	 765.92 \\ 
		MOptQR 	 & 	 -1.7359510505e+02 & 	 9.92e-06 & 	 456  & 	 2.59e-14 & 	 800.08 \\  
		OptM 	 & 	 -1.7359510506e+02 & 	 2.47e-06 & 	 472  & 	 2.49e-14 & 	 904.44 \\ 
		PLAM 	 & 	 -1.7359510505e+02 & 	 8.88e-06 & 	 330  & 	 2.75e-14 & 	 601.41 \\
		PCAL 	 & 	 -1.7359510505e+02 & 	 8.52e-06 & 	 253  & 	 2.62e-14 & 	 548.70 \\  
		\midrule    
		
		\multicolumn{1}{c}{} & \multicolumn{5}{c}{h2o, $n=2103$, $p=4$\qquad\qquad($\beta_{\text{PLAM}}=10, \beta_{\text{PCAL}}=1$)} \\\midrule 
		SCF 		 & 	 -1.6440507245e+01 & 	 1.16e-08 & 	 8  & 	 1.15e-15 & 	 1.29 \\ 
		TRDCM 	 & 	 -1.6440507246e+01 & 	 6.48e-06 & 	 11  & 	 1.11e-15 & 	 1.02 \\ 
		MOptQR 	 & 	 -1.6440507246e+01 & 	 3.84e-06 & 	 49  & 	 9.30e-16 & 	 1.14 \\ 
		OptM 	 & 	 -1.6440507246e+01 & 	 2.01e-06 & 	 61  & 	 6.40e-15 & 	 1.50 \\ 
		PLAM 	 & 	 -1.6440507245e+01 & 	 6.43e-06 & 	 56  & 	 2.37e-15 & 	 1.29 \\  
		PCAL 	 & 	 -1.6440507246e+01 & 	 7.42e-06 & 	 42  & 	 1.86e-15 & 	 1.06 \\\midrule
		
		\multicolumn{1}{c}{} & \multicolumn{5}{c}{hnco, $n=2103$, $p=8$\qquad\qquad($\beta_{\text{PLAM}}=10, \beta_{\text{PCAL}}=1$)} \\\midrule 
		SCF 	 & 	 -2.8634664360e+01 & 	 9.44e-08 & 	 12  & 	 3.82e-15 & 	 4.32 \\ 
		TRDCM 	 & 	 -2.8634664365e+01 & 	 9.54e-06 & 	 13  & 	 3.47e-15 & 	 4.47 \\ 
		MOptQR 	 & 	 -2.8634664363e+01 & 	 9.74e-06 & 	 163  & 	 3.17e-15 & 	 12.26 \\ 
		OptM 	 & 	 -2.8634664365e+01 & 	 5.30e-06 & 	 117  & 	 2.26e-15 & 	 8.30 \\ 
		PLAM 	 & 	 -2.8634664364e+01 & 	 9.95e-06 & 	 105  & 	 3.18e-15 & 	 7.39 \\
		PCAL 	 & 	 -2.8634664364e+01 & 	 9.03e-06 & 	 70  & 	 2.60e-15 & 	 5.36 \\  
		\bottomrule[.4mm]
	\end{tabular}
	\caption{The results in Kohn-Sham total energy minimization\label{tab:KS1}} 
\end{table} 

\begin{table}[htbp]
	\centering
	\begin{tabular}{c||ccccc}
		\hline
		\toprule[.4mm]
		Solver	&   {$E_{tot}$}  &   {KKT violation}   & {Iteration}  & {Feasibility violation} & {Time(s)}   \\
		\midrule[.3mm]
		
		\multicolumn{1}{c}{} & \multicolumn{5}{c}{nic, $n=251$, $p=7$\qquad\qquad($\beta_{\text{PLAM}}=10, \beta_{\text{PCAL}}=1$)} \\\midrule 
		SCF 	 & 	 -2.3543529950e+01 & 	 2.13e-10 & 	 11  & 	 2.99e-15 & 	 1.47 \\ 
		TRDCM 	 & 	 -2.3543529955e+01 & 	 7.94e-06 & 	 15  & 	 4.49e-15 & 	 0.99 \\ 
		MOptQR 	 & 	 -2.3543529955e+01 & 	 3.04e-06 & 	 111  & 	 2.73e-15 & 	 1.53 \\ 
		OptM 	 & 	 -2.3543529955e+01 & 	 3.86e-07 & 	 63  & 	 8.80e-15 & 	 0.90 \\ 
		PLAM 	 & 	 -2.3543529955e+01 & 	 4.02e-06 & 	 67  & 	 1.39e-15 & 	 0.89 \\
		PCAL 	 & 	 -2.3543529955e+01 & 	 8.42e-06 & 	 52  & 	 1.88e-15 & 	 0.99 \\  
		\midrule 
		
		\multicolumn{1}{c}{} & \multicolumn{5}{c}{pentacene, $n=44791$, $p=51$\qquad\qquad($\beta_{\text{PLAM}}=13, \beta_{\text{PCAL}}=1$)} \\\midrule 
		SCF 		 & 	 -1.3189029494e+02 & 	 5.76e-07 & 	 13  & 	 1.58e-14 & 	 293.68 \\ 
		TRDCM 	 & 	 -1.3189029495e+02 & 	 7.60e-06 & 	 22  & 	 1.08e-14 & 	 276.25 \\ 
		MOptQR 	 & 	 -1.3189029495e+02 & 	 7.78e-06 & 	 112  & 	 3.21e-14 & 	 306.97 \\ 
		OptM 	 & 	 -1.3189029495e+02 & 	 1.39e-06 & 	 97  & 	 3.39e-14 & 	 283.02 \\ 
		PLAM 	 & 	 -1.3189029495e+02 & 	 8.66e-06 & 	 123  & 	 3.52e-14 & 	 321.04 \\  
		PCAL 	 & 	 -1.3189029495e+02 & 	 7.67e-06 & 	 89  & 	 3.08e-14 & 	 271.32 \\\midrule  
		
		\multicolumn{1}{c}{} & \multicolumn{5}{c}{ptnio, $n=4069$, $p=43$\qquad\qquad($\beta_{\text{PLAM}}=13, \beta_{\text{PCAL}}=1$)} \\\midrule 
		SCF 	 & 	 -2.2678884268e+02 & 	 1.09e-05 & 	 53  & 	 1.46e-14 & 	 168.25 \\ 
		TRDCM 	 & 	 -2.2678882693e+02 & 	 2.81e-04 & 	 200  & 	 1.07e-14 & 	 471.34 \\ 
		MOptQR 	 & 	 -2.2678884271e+02 & 	 9.57e-06 & 	 786  & 	 1.06e-14 & 	 347.38 \\ 
		OptM 	 & 	 -2.2678884273e+02 & 	 9.52e-06 & 	 508  & 	 1.14e-14 & 	 203.63 \\ 
		PLAM 	 & 	 -2.2678884271e+02 & 	 9.00e-06 & 	 579  & 	 1.01e-14 & 	 213.60 \\
		PCAL 	 & 	 -2.2678884271e+02 & 	 8.55e-06 & 	 386  & 	 1.19e-14 & 	 189.70 \\  
		\midrule  
		
		\multicolumn{1}{c}{} & \multicolumn{5}{c}{qdot, $n=2103$, $p=8$\qquad\qquad($\beta_{\text{PLAM}}=10, \beta_{\text{PCAL}}=1$)} \\\midrule 
		SCF 	 & 	 2.7700280133e+01 & 	 6.70e-03 & 	 5  & 	 2.92e-15 & 	 1.09 \\ 
		TRDCM 	 & 	 2.7699537080e+01 & 	 1.43e-02 & 	 200  & 	 2.73e-15 & 	 27.01 \\ 
		MOptQR 	 & 	 1.0483319768e+02 & 	 3.45e+01 & 	 1000  & 	 1.77e-15 & 	 28.72 \\ 
		OptM 	 & 	 2.7699807230e+01 & 	 1.45e-04 & 	 1000  & 	 2.39e-15 & 	 29.89 \\ 
		PLAM 	 & 	 2.7699800860e+01 & 	 9.68e-06 & 	 678  & 	 1.98e-15 & 	 19.30 \\
		PCAL 	 & 	 2.7699800851e+01 & 	 5.41e-06 & 	 962  & 	 2.88e-15 & 	 35.01 \\  
		\midrule
		
		\multicolumn{1}{c}{} & \multicolumn{5}{c}{si2h4, $n=2103$, $p=6$\qquad\qquad($\beta_{\text{PLAM}}=10, \beta_{\text{PCAL}}=1$)} \\\midrule 
		SCF 	 & 	 -6.3009750375e+00 & 	 5.25e-07 & 	 11  & 	 3.62e-15 & 	 2.97 \\ 
		TRDCM 	 & 	 -6.3009750459e+00 & 	 8.24e-06 & 	 16  & 	 3.12e-15 & 	 4.30 \\ 
		MOptQR 	 & 	 -6.3009750460e+00 & 	 3.70e-06 & 	 116  & 	 2.00e-15 & 	 5.96 \\ 
		OptM 	 & 	 -6.3009750459e+00 & 	 9.60e-06 & 	 68  & 	 1.41e-14 & 	 4.15 \\  
		PLAM 	 & 	 -6.3009750455e+00 & 	 7.27e-06 & 	 89  & 	 1.58e-15 & 	 5.33 \\
		PCAL 	 & 	 -6.3009750459e+00 & 	 4.33e-06 & 	 62  & 	 2.42e-15 & 	 3.90 \\  
		\midrule   
		
		\multicolumn{1}{c}{} & \multicolumn{5}{c}{sih4, $n=2103$, $p=4$\qquad\qquad($\beta_{\text{PLAM}}=10, \beta_{\text{PCAL}}=1$)} \\\midrule 
		SCF 	 & 	 -6.1769279820e+00 & 	 2.07e-08 & 	 8  & 	 1.75e-15 & 	 1.91 \\ 
		TRDCM 	 & 	 -6.1769279850e+00 & 	 9.53e-06 & 	 10  & 	 1.14e-15 & 	 1.60 \\ 
		MOptQR 	 & 	 -6.1769279851e+00 & 	 4.32e-06 & 	 34  & 	 1.58e-15 & 	 1.07 \\ 
		OptM 	 & 	 -6.1769279851e+00 & 	 8.18e-06 & 	 46  & 	 8.52e-16 & 	 1.62 \\ 
		PLAM 	 & 	 -6.1769279849e+00 & 	 7.37e-06 & 	 56  & 	 1.99e-15 & 	 1.79 \\
		PCAL 	 & 	 -6.1769279847e+00 & 	 9.16e-06 & 	 47  & 	 1.55e-15 & 	 1.69 \\  
		\bottomrule[.4mm]
	\end{tabular}
	\caption{The results in Kohn-Sham total energy minimization\label{tab:KS2}} 
\end{table}

\subsection{Parallel Efficiency\label{section: parallel}}
In this subsection, we examine the parallel efficiency of our algorithms PLAM and PCAL. 
To investigate the parallel scalability, we need to test large scale problems
in a single core, which consumes lots of CPU time.
To avoid meaningless tests, we only compare the parallel 
performances of PCAL with MOptQR in this subsection.

Both algorithms are implemented in the C++ language and parallelized by OpenMP. 
The linear algebra library we used in comparison is Eigen\footnote{Available from 
\href{http://eigen.tuxfamily.org/index.php?title=Main_Page}
{http://eigen.tuxfamily.org/index.php?title=Main\_Page}} (version 3.3.4), 
which is an open and popular 
C++ template library for matrix computation. 
We define the speedup factor for running a code on $m$ cores as
$$\text{speedup-factor}(m)=\frac{\text{wall-clock time for a single core run}}
{\text{wall-clock time for a~} m\text{-core run}}.$$

BLAS3 type arithmetic operations contribute a high proportion in
computational cost in both PCAL and MOptQR. 
Therefore, a good parallel strategy for BLAS3 calculation is unnegligible in saving
CPU time. Given this, we first determine the parallel strategy for 
matrix-matrix multiplication by a set of tests.
We have two choices. The library Eigen provides its own 
multi-threading computing\footnote{More information at	
\href{http://eigen.tuxfamily.org/dox/TopicMultiThreading.html}
{http://eigen.tuxfamily.org/dox/TopicMultiThreading.html}} 
that is the default parallel strategy for dense matrix-matrix products and 
row-major-sparse$*$dense vector/matrix products 
in OpenMP.
Another strategy is to parallelize BLAS3 computation 
in the manner of column-wise product. Namely, when we calculate $AB$,
we multiply matrix $A$ by each column of $B$ in parallel.
To figure out which strategy is better, we test the 
parallel scalability of BLAS3 computation under these two schemes. We generate 
{\ttfamily A=Random(1000,10000)} and {\ttfamily B=Random(10000,1000)}, where 
``{\ttfamily Random($\cdot$,$\cdot$)}" is an internal generation function provided by Eigen. We run 
the code in parallel with $1$, $2$, $4$, $8$, $16$, $32$, $64$ and $96$ cores, respectively. The result of matrix-matrix multiplication 
$AB$ is illustrated in Figure \ref{fig:DD-BLAS3}. 
``Eigen" and ``Column-wise" represent the default parallel strategy and column-wise product strategy, 
respectively. We can observe that column-wise parallelization obviously 
outperforms the default setting of Eigen in multi-threading computing. 
Hence, in the following implementation, we choose column-wise parallelization strategy for BLAS3 in 
our experiments.
\begin{figure}[htbp]		
	\setcounter{subfigure}{0}
	\centering
	\subfigure[Speedup factor]
	{\includegraphics[scale=.4]
		{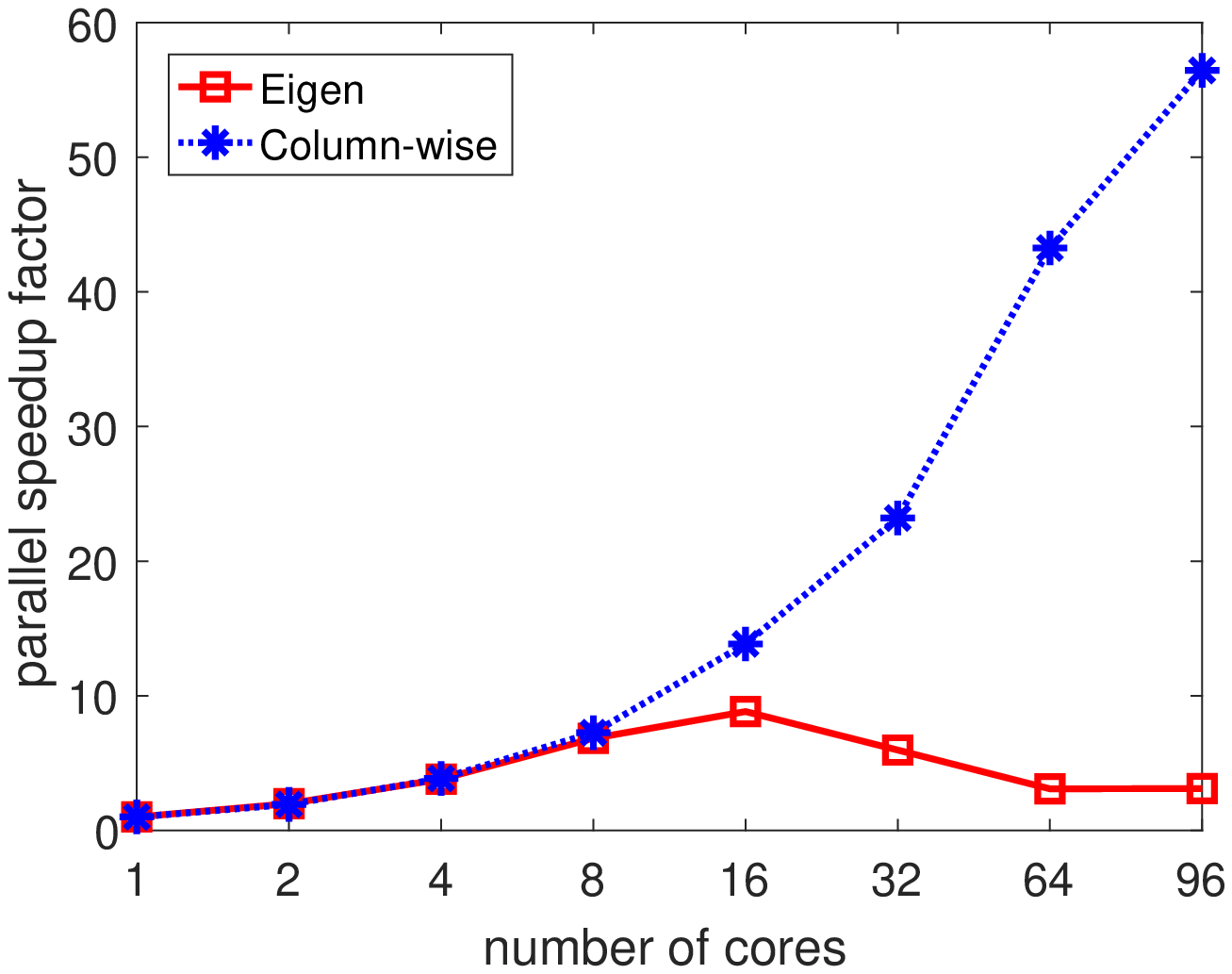}}
	\qquad
	\subfigure[Wall-clock time(s)]
	{\includegraphics[scale=.4]
		{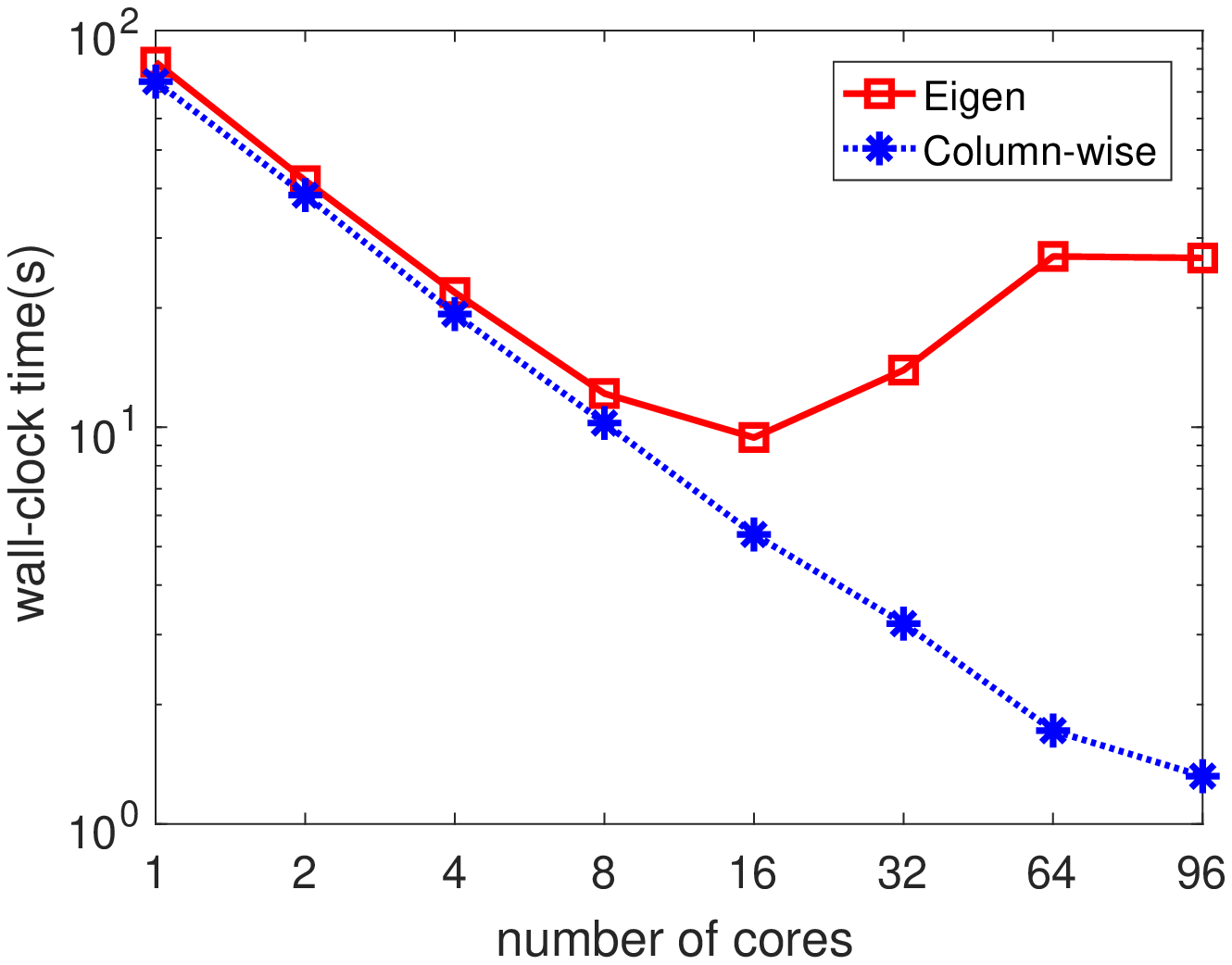}}	
	\caption{The results of dense-dense BLAS3: $A^{1000\times 10000}B^{10000\times 1000}$\label{fig:DD-BLAS3}}
\end{figure}

Next, we investigate the parallel scalability of the new proposed PCAL and MOptQR. 
According to the existent numerical report of Eigen\footnote{More information at 
\href{http://eigen.tuxfamily.org/dox/group__TutorialLinearAlgebra.html}
{http://eigen.tuxfamily.org/dox/group\_\_TutorialLinearAlgebra.html}}, 
we select the class ``{\ttfamily LLT}" in Eigen to compute QR factorization. 
The calculation of orthonormalization consists of a small size ($p$-by-$p$) Cholesky decomposition
and solving a $p$-by-$p$ linear system. The maximum number of iterations for  MOptQR and PCAL is set to
$1000$. All the parameters for MOptQR and PCAL take their default values. The initial guess $X^0$ is 
generated by $X^0=${\ttfamily random(n,p)}" and $X^0=${\ttfamily qr}$(X^0)$.

We first focus on the test Problems 1 and 2. For Problem 1, we set $L$ as a block diagonal matrix, i.e., 
$L=\Diag(L_1,\dots,L_s)$, where $L_i\in\R^{5\times5}$ is a tridiagonal matrix with $2$ on its main diagonal 
and $-1$ on subdigonal, for $i=1,\dots,s$. The coefficient $\alpha$ is set to $1$. For the generation of 
Problem 2, we set $A$ as a tridiagonal matrix with $2$ on its main diagonal and $-1$ on subdigonal, 
and $G${\ttfamily=Random(n,p)}. The advantage of such generation is to make function value and gradient
calculations parallelizable.
In the first group of tests, we aim to figure out how MOptQR and PCAL perform with the increasing
width of variables. We set $n=10000$ and $p$ varying from a set of 
increasing values $\{500,1000,1500,2000,2500\}$. Both algorithms are run in parallel with $96$ cores.
The wall-clock time results are shown in Figure \ref{fig:vary-p}.
Here, ``\#cores" stands for the number of cores. 
From Figure \ref{fig:vary-p}, we notice that PCAL always takes less amount of 
wall-clock time than MOptQR. As the width of the matrix variable increases, 
the running time of MOptQR increases much more rapidly than PCAL. 
\begin{figure}[htbp]
	\setcounter{subfigure}{0}
	\centering
	\subfigure[Problem 1]
	{\includegraphics[scale=.4]
		{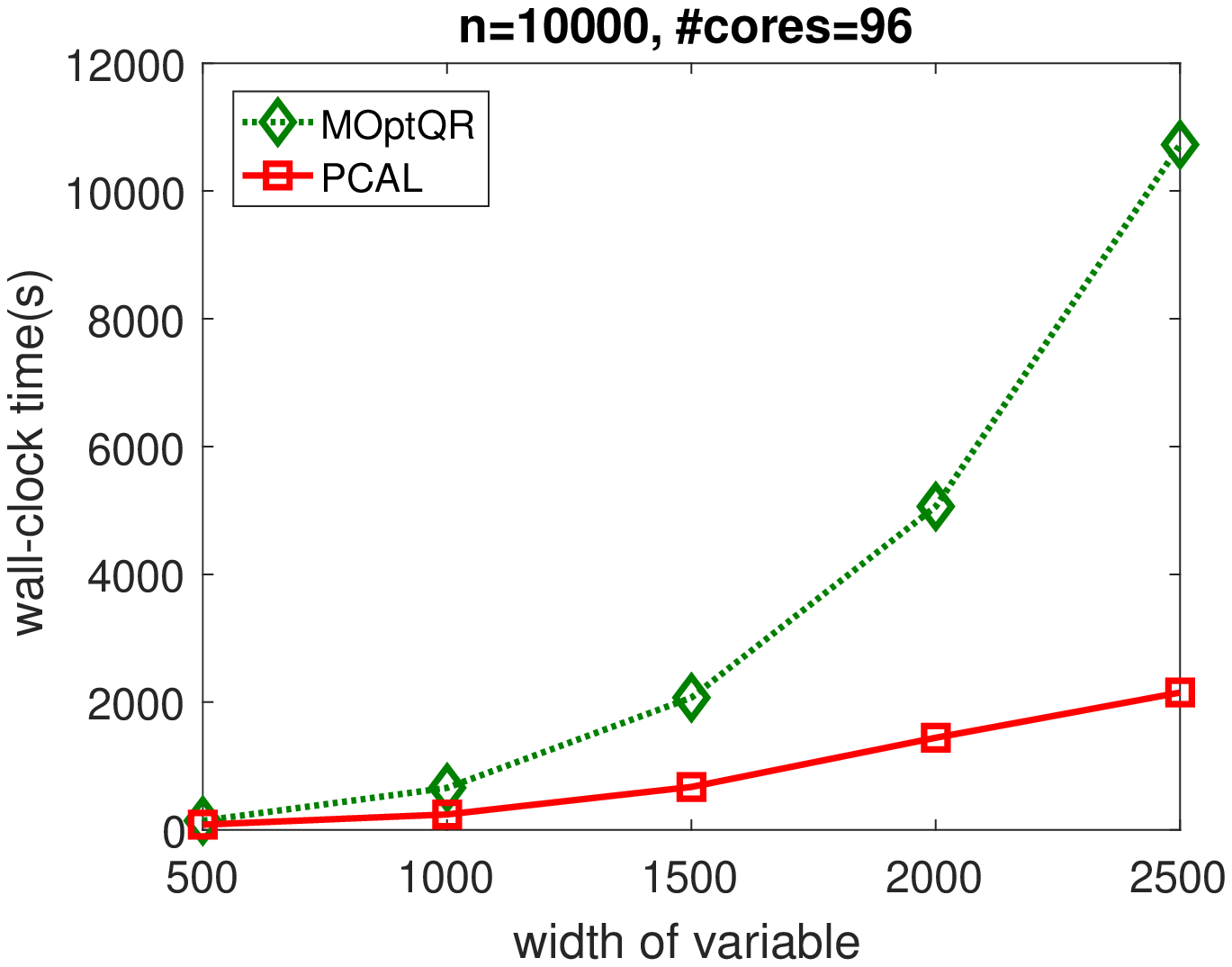}}
	\qquad
	\subfigure[Problem 2]
	{\includegraphics[scale=.4]
		{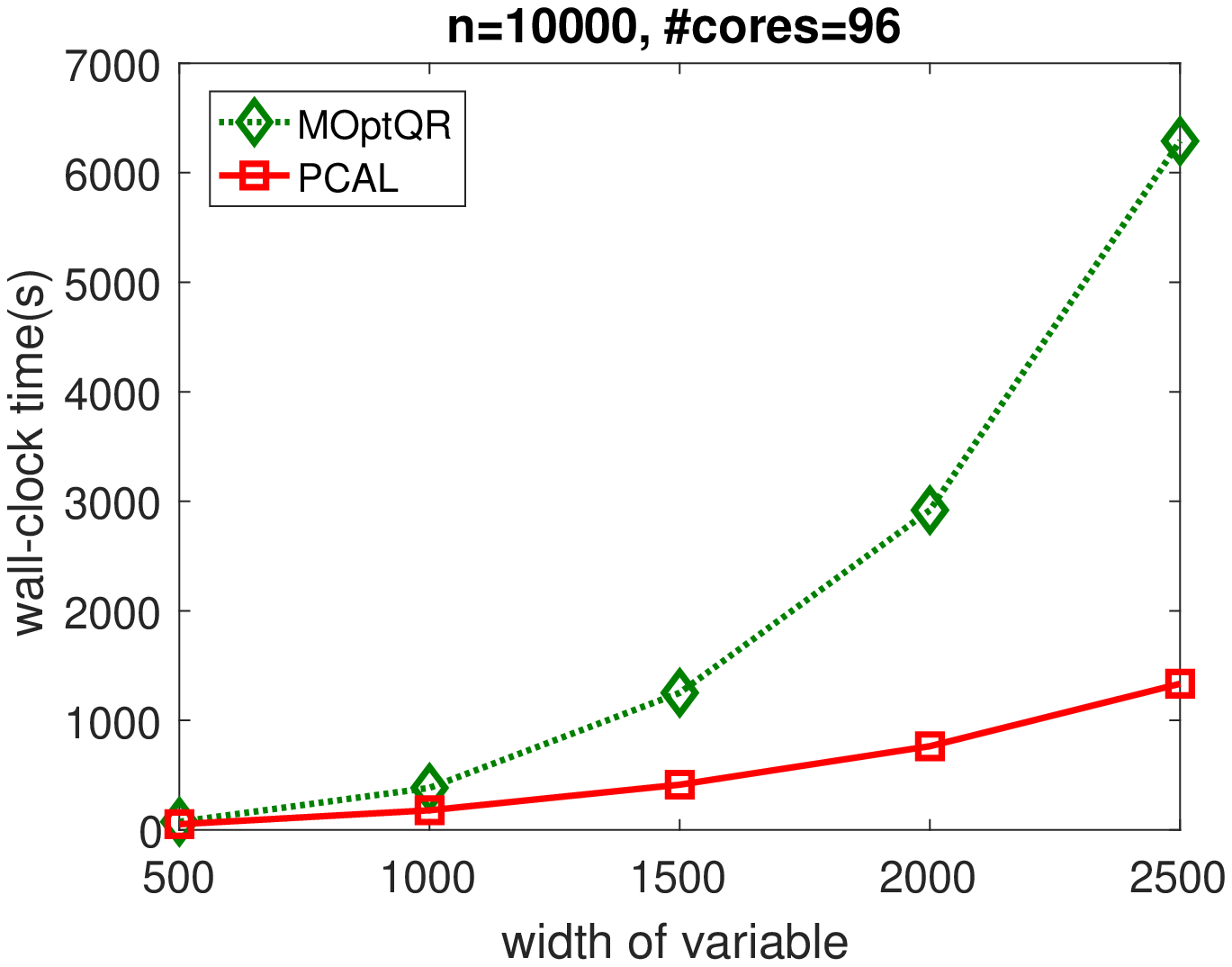}}	
	\caption{The wall-clock time results on varying width of the matrix variable\label{fig:vary-p}}
\end{figure}

In Figure \ref{fig:OCQP-percentage}, we show wall-clock time of four categories: 
three categories: ``BLAS3" (dense-dense matrix multiplication), ``Func" (function value and gradient
evaluation) and ``Orth" (orthonormalization including QR factorization for MOptQR and the final correction step in PCAL). These are the major
computational components of both PCAL and MOptQR, albeit in different proportions.
We have to clarify two issues: firstly, we categorize these categories of
calculation only at the
highest solver level. As such, any matrix-matrix multiplication involved in
function value and gradient evaluation is not counted as in the ``BLAS3" category.
Secondly, although the “correctness” of such a
classification scheme may be debatable, it does not alter the overall fact, as is clearly shown
by our computational results, that the category ``BLAS3" is much more scalable than the
category ``Orth" on our test platform.
The running time of each category is measured in terms of the percentage of wall-clock time
spent in that category over the total wall-clock time. We can clearly see that for PCAL
the run time of ``BLAS3" dominates the entire computation in almost all cases. The ``BLAS3"
time increases steadily as $p$ increases from $500$ to $2500$, while the ``Func" time decreases
steadily. The run time of ``Orth" is negligible.
However, for MOptQR, the ``BLAS3" time takes around $60\%$ of total run time and decreases
steadily with the increasing of $p$. Meanwhile, the ``Orth" time takes around ``40\%" of total
run time and increasing steadily.
\begin{figure}[htbp]		
	\setcounter{subfigure}{0}
	\centering
	\subfigure[MOptQR]
	{\includegraphics[scale=.4]
		{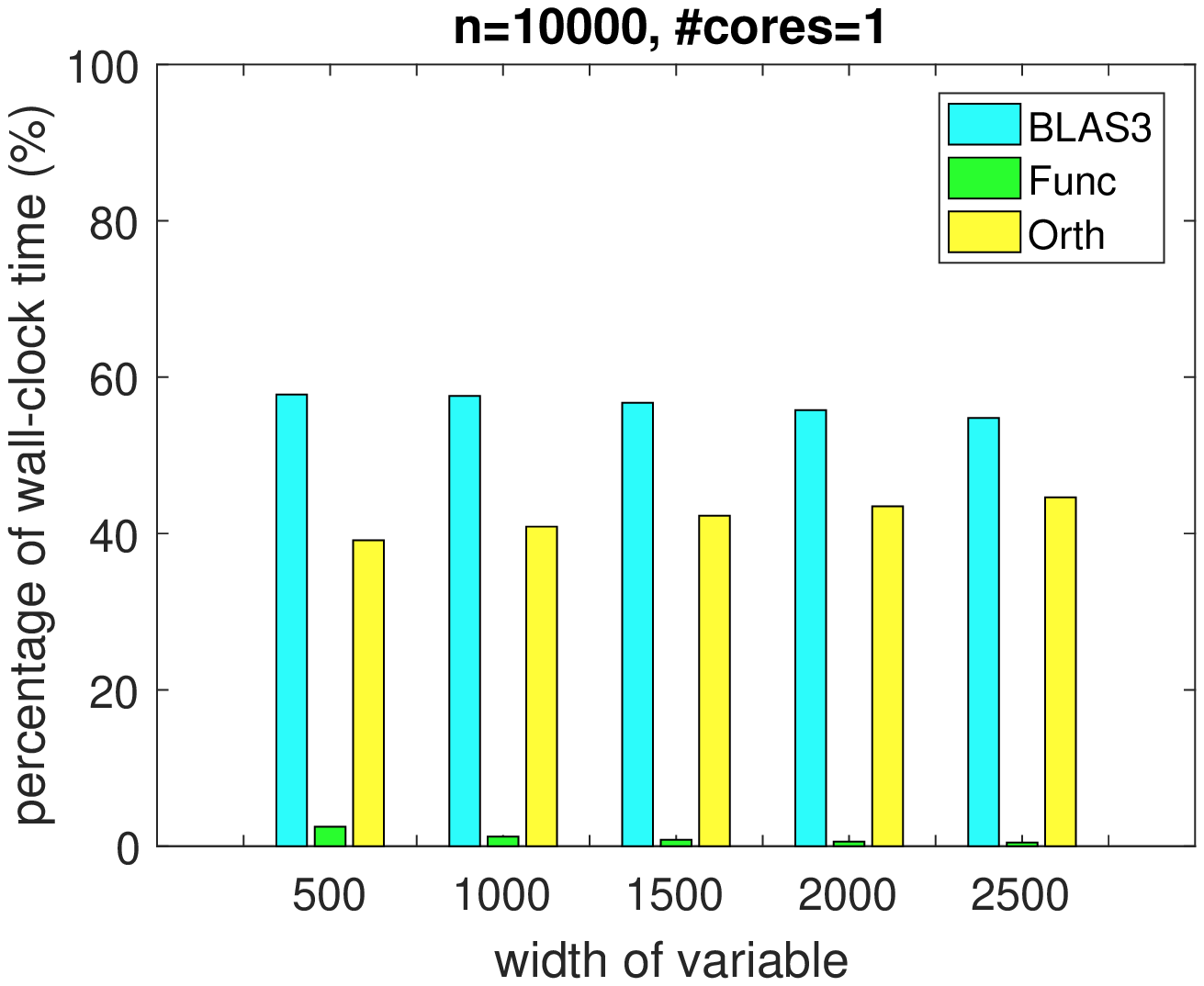}}
	\qquad
	\subfigure[PCAL]
	{\includegraphics[scale=.4]
		{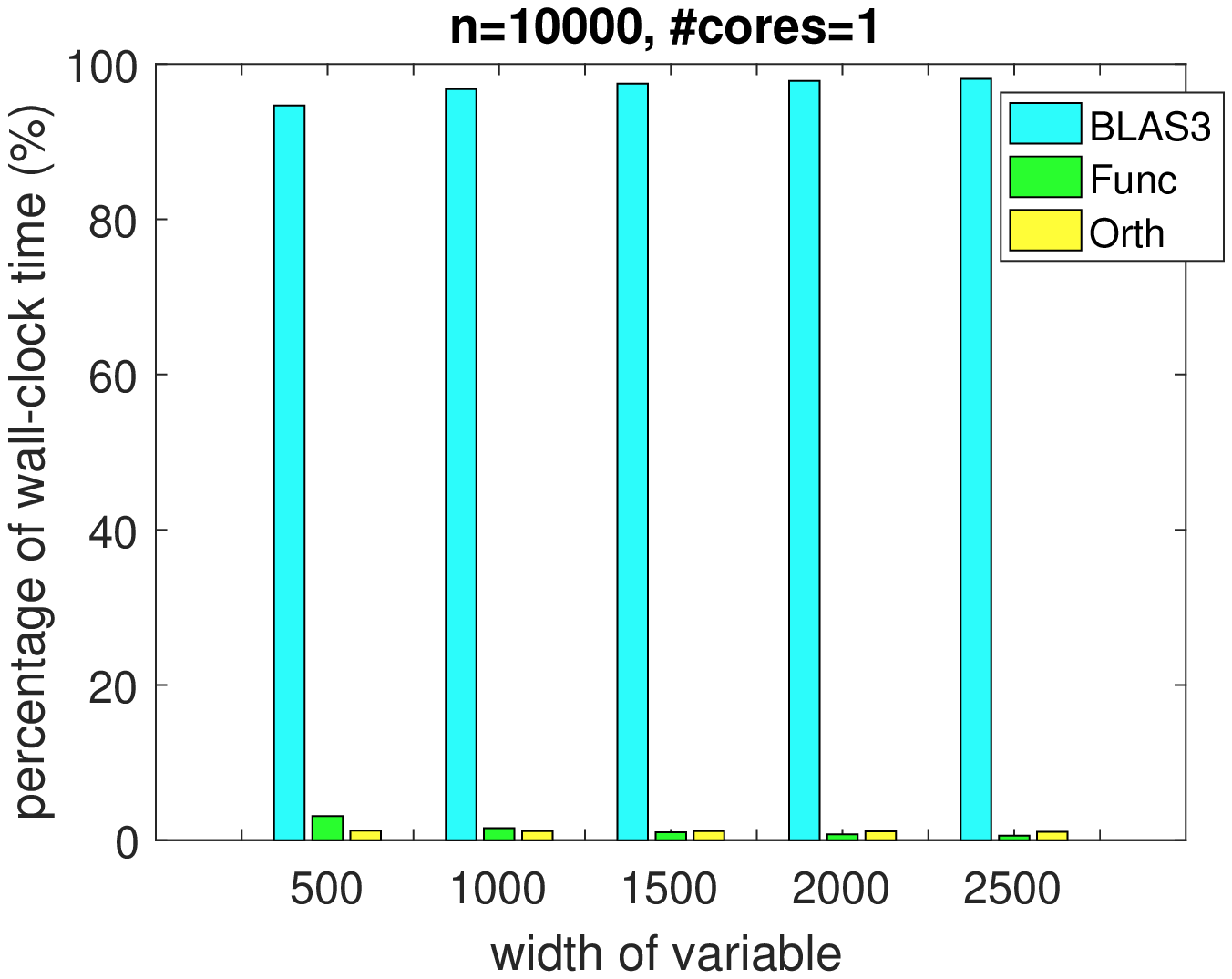}}	
	\caption{A comparison of timing profile on a single core for Problem 2\label{fig:OCQP-percentage}}
\end{figure}

Now, we set $n=10000$ and $p=1000, 2000$, and run PCAL and MOptQR 
in parallel with $1$, $2$, $4$, $8$, $16$, $32$, $64$ and $96$ cores, respectively.
Figure \ref{fig:speedup} and \ref{fig:speedup1} illustrate the speedup factors associated with total 
running wall-clock time, ``BLAS3", ``Func" and ``Orth", respectively. 
From these two figures, we can observe that BLAS3 operation has high parallel scalability, while the 
speedup factor of ``Orth" increases slowly as the number of cores increases,
which directly leads to the higher overall scalability of PCAL than MOptQR.
Moreover, as the width of the matrix variable increasing, the advantage of PCAL 
in parallel scalability becomes more obvious. 
\begin{figure}[htbp]	
	\setcounter{subfigure}{0}
	\centering
	\subfigure[Problem 1: MOptQR]
	{\includegraphics[scale=.4]
		{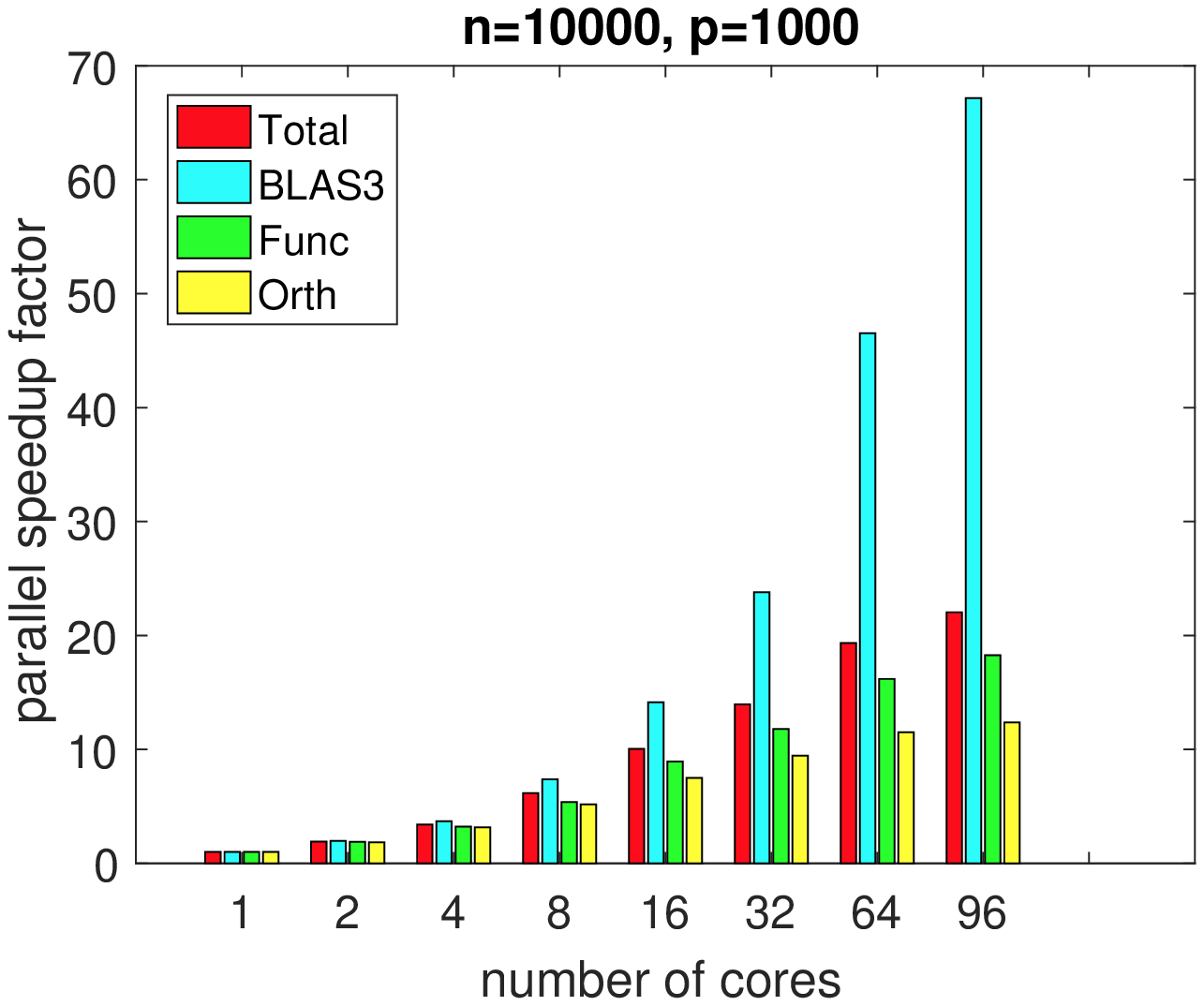}}
	\qquad
	\subfigure[Problem 1: PCAL]
	{\includegraphics[scale=.4]
		{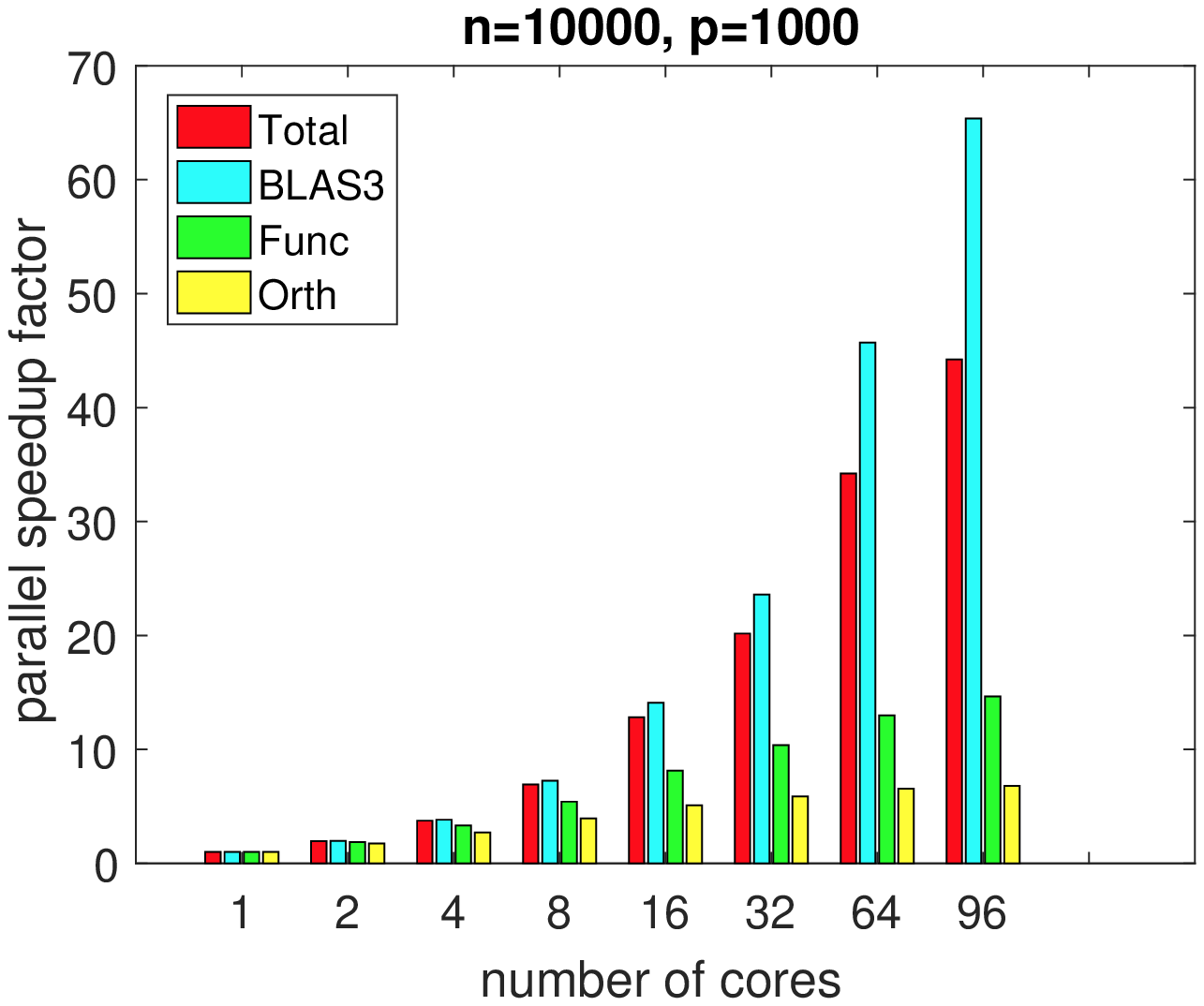}}	
	
	\subfigure[Problem 2: MOptQR]
	{\includegraphics[scale=.4]
		{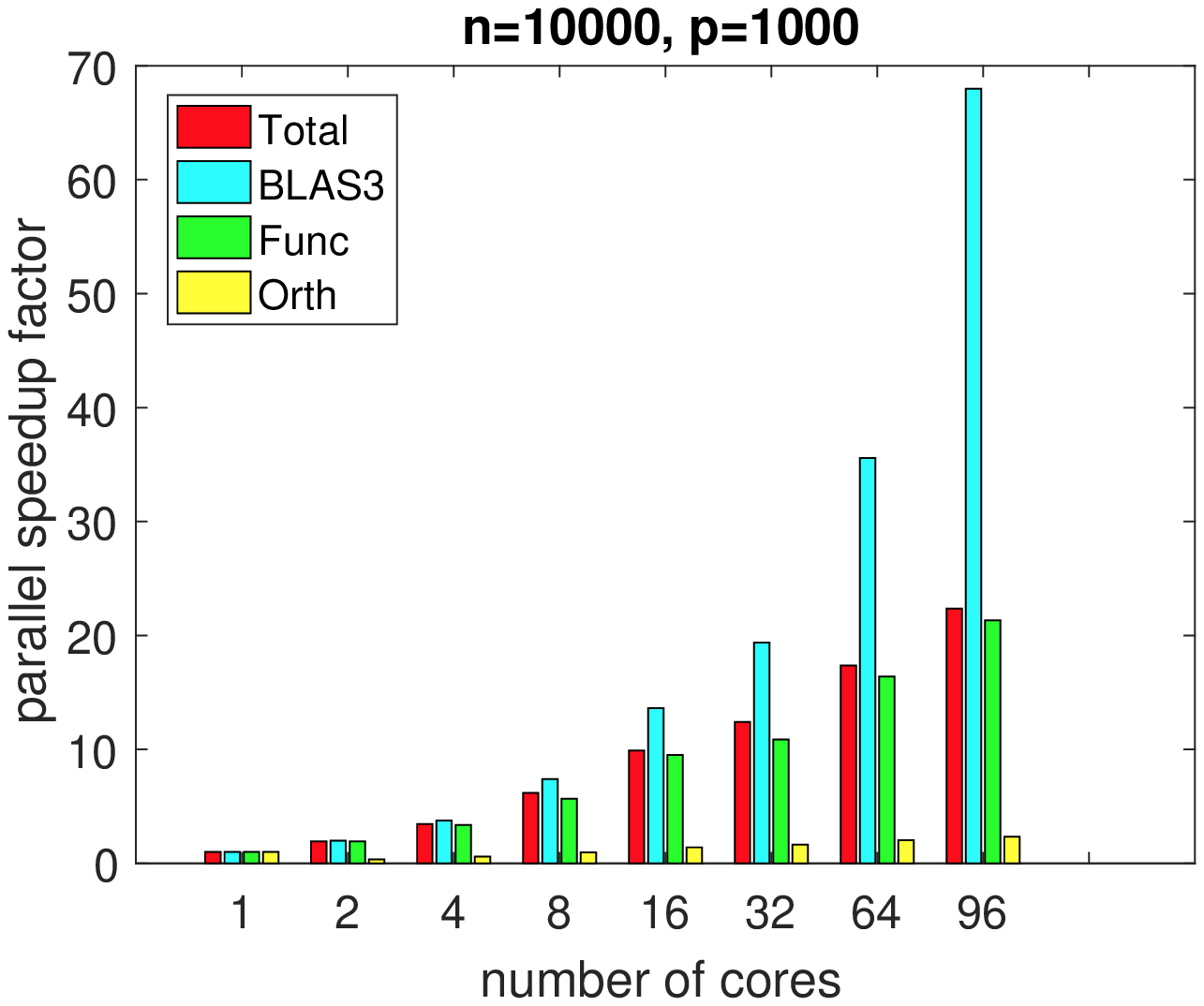}}
	\qquad
	\subfigure[Problem 2: PCAL]
	{\includegraphics[scale=.4]
		{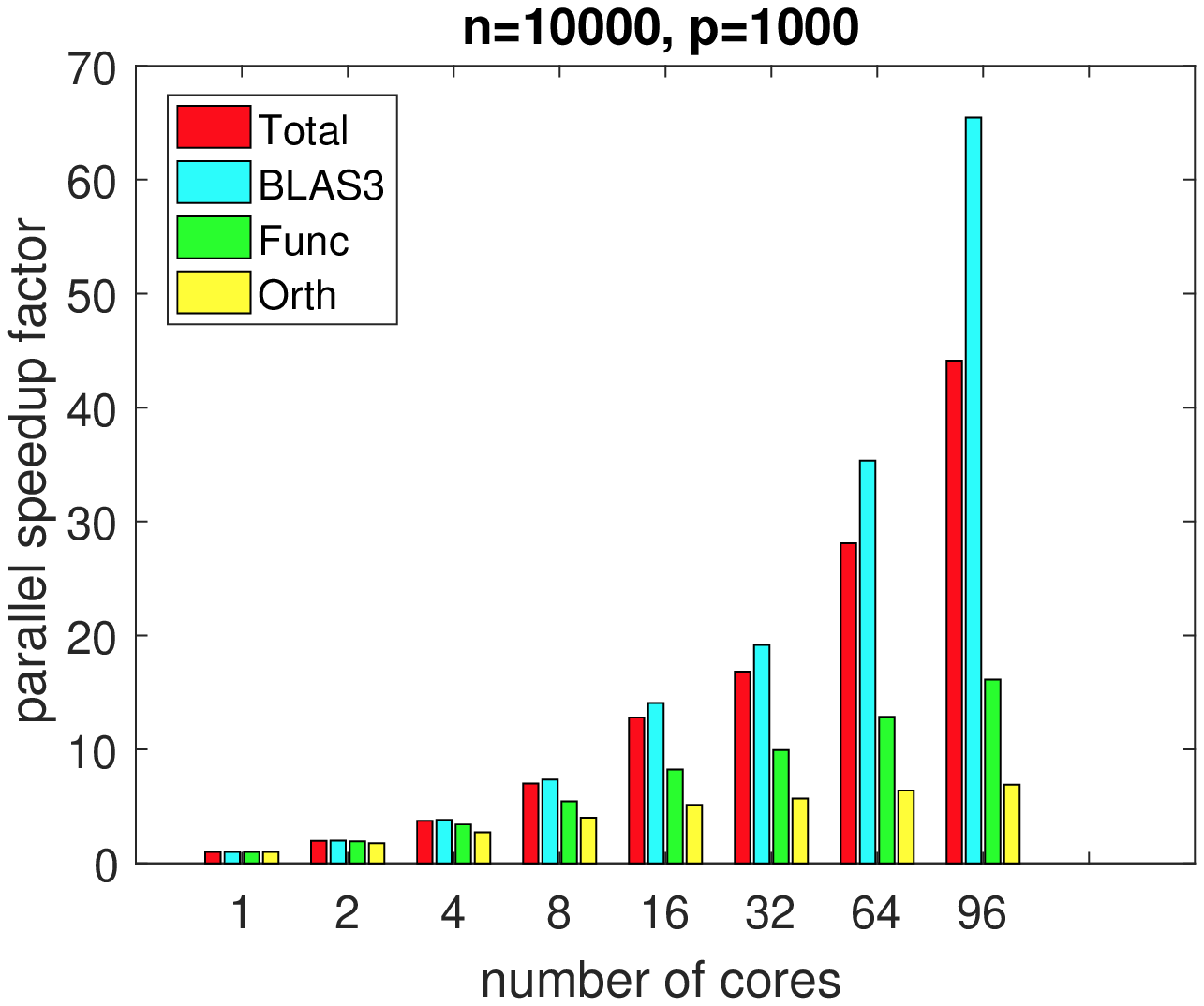}}		
	\caption{A comparison of speedup factor among MOptQR and PCAL ($p=1000$)\label{fig:speedup}}
\end{figure}
\begin{figure}[htbp]	
	\setcounter{subfigure}{0}
	\centering
	\subfigure[Problem 1: MOptQR]
	{\includegraphics[scale=.4]
		{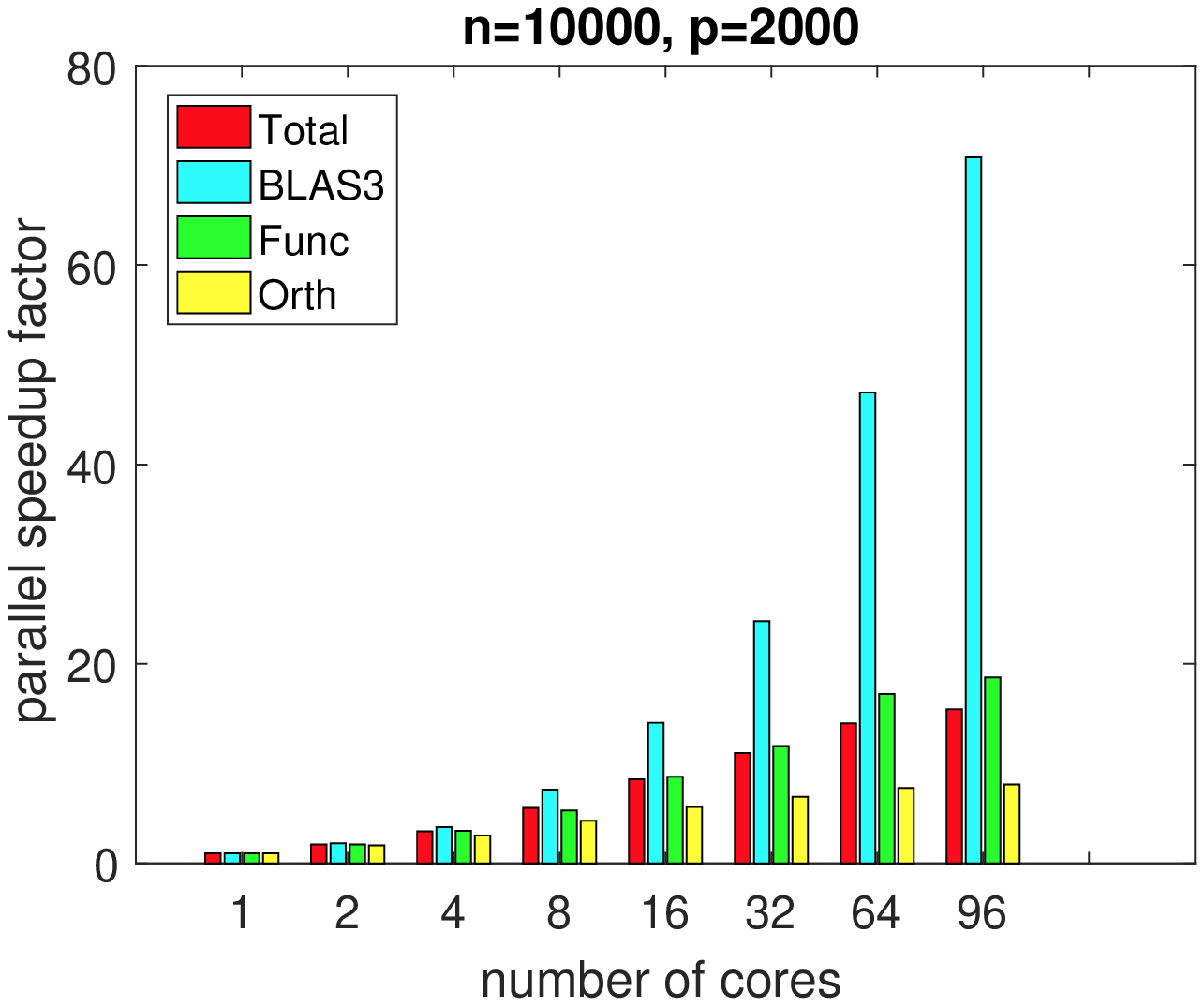}}
	\qquad
	\subfigure[Problem 1: PCAL]
	{\includegraphics[scale=.4]
		{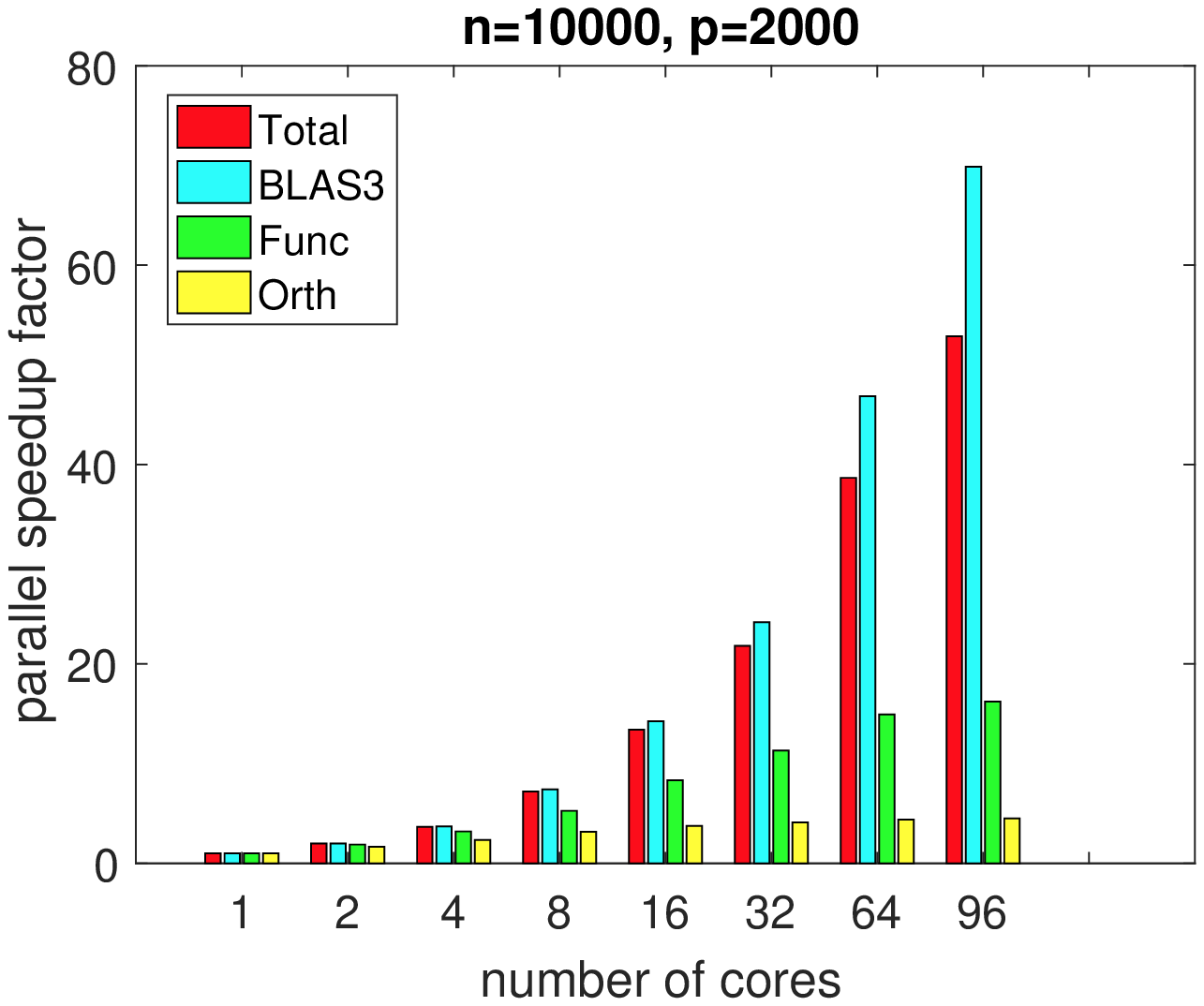}}	
	
	\subfigure[Problem 2: MOptQR]
	{\includegraphics[scale=.4]
		{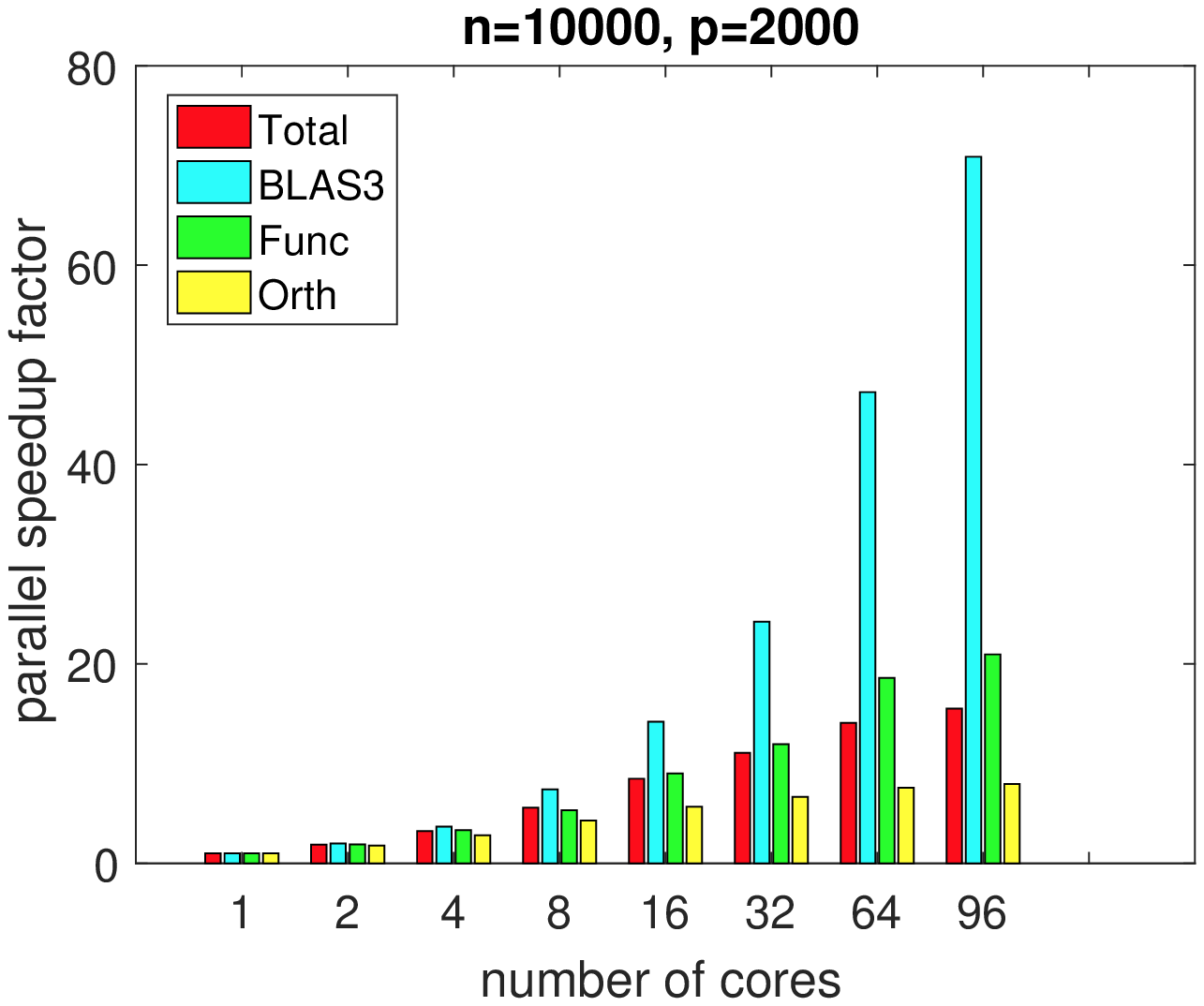}}
	\qquad
	\subfigure[Problem 2: PCAL]
	{\includegraphics[scale=.4]
		{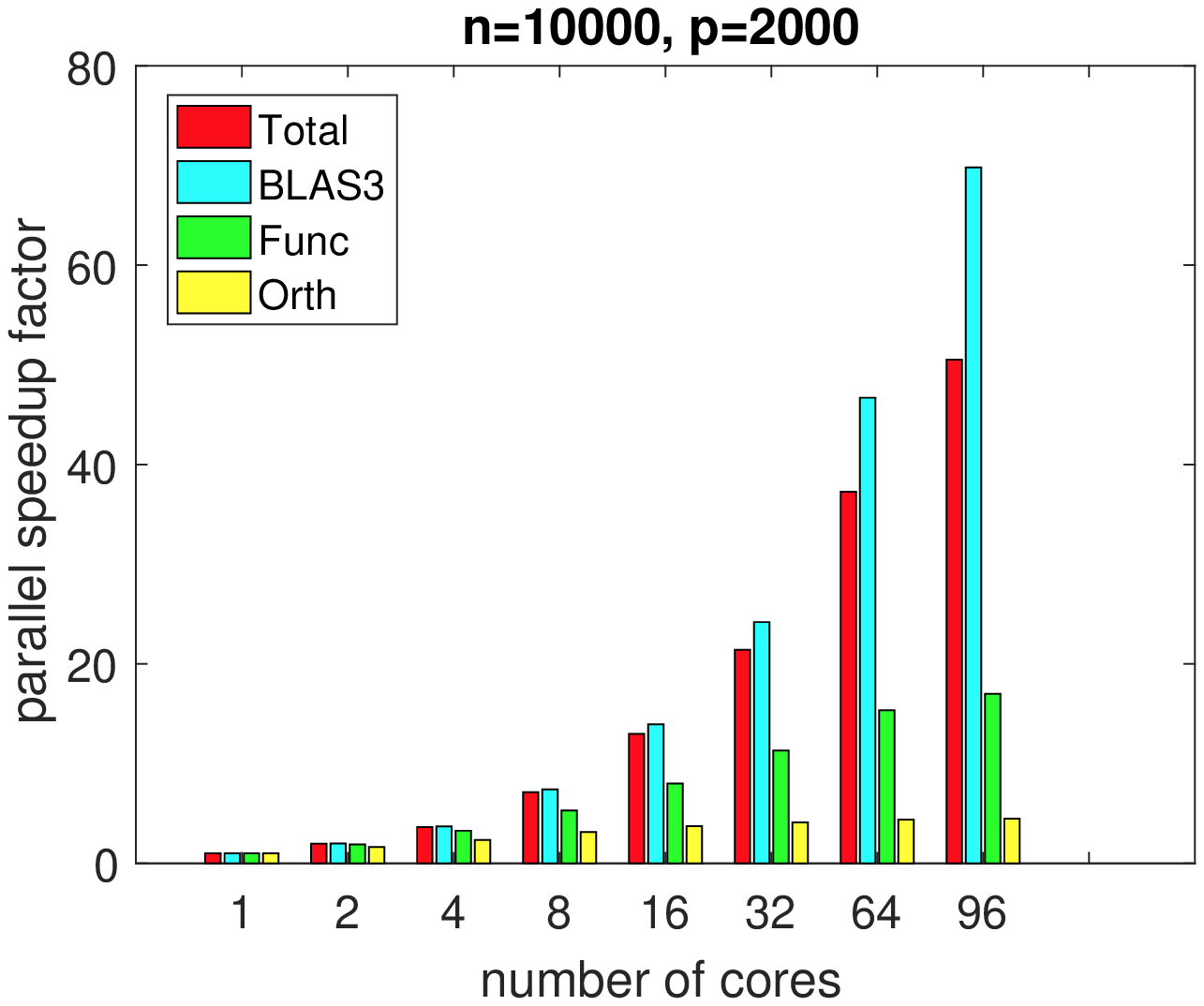}}		
	\caption{A comparison of speedup factor among MOptQR and PCAL ($p=2000$)\label{fig:speedup1}}
\end{figure}

In the end, we test Problem 6 under $n=10000$, $p=1000$.
Figure \ref{fig:speedupks} illustrate the results of speedup factors
associated with total 
running wall-clock time, ``BLAS3", ``Func" and ``Orth" of PCAL and MOptQR, respectively. 
We can learn from this figure
that the overall scalability of PCAL is again superior to that of MOptQR. 
\begin{figure}[htbp]	
	\setcounter{subfigure}{0}
	\centering
	\subfigure[MOptQR]
	{\includegraphics[scale=.4]
		{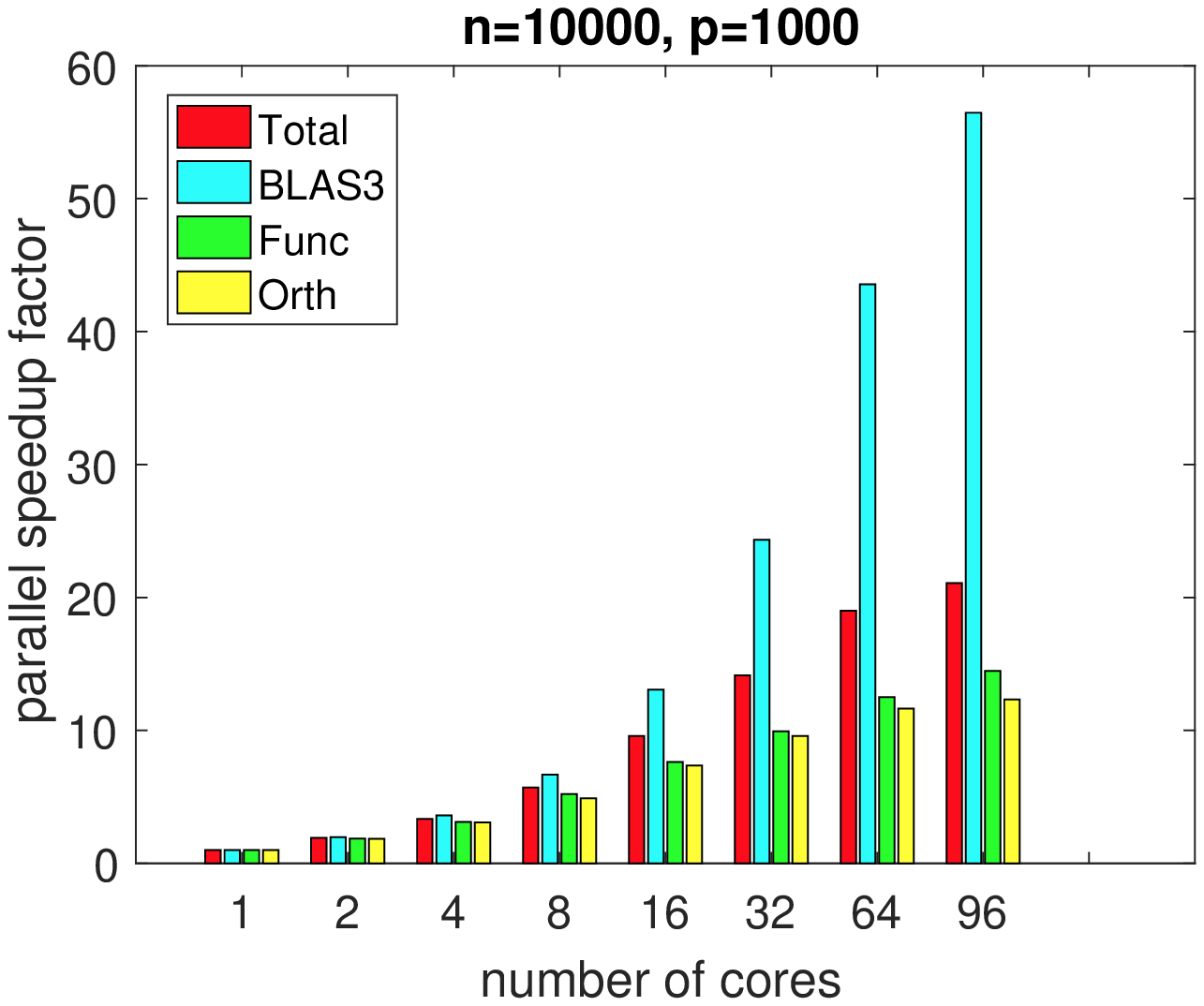}}
	\qquad
	\subfigure[PCAL]
	{\includegraphics[scale=.4]
		{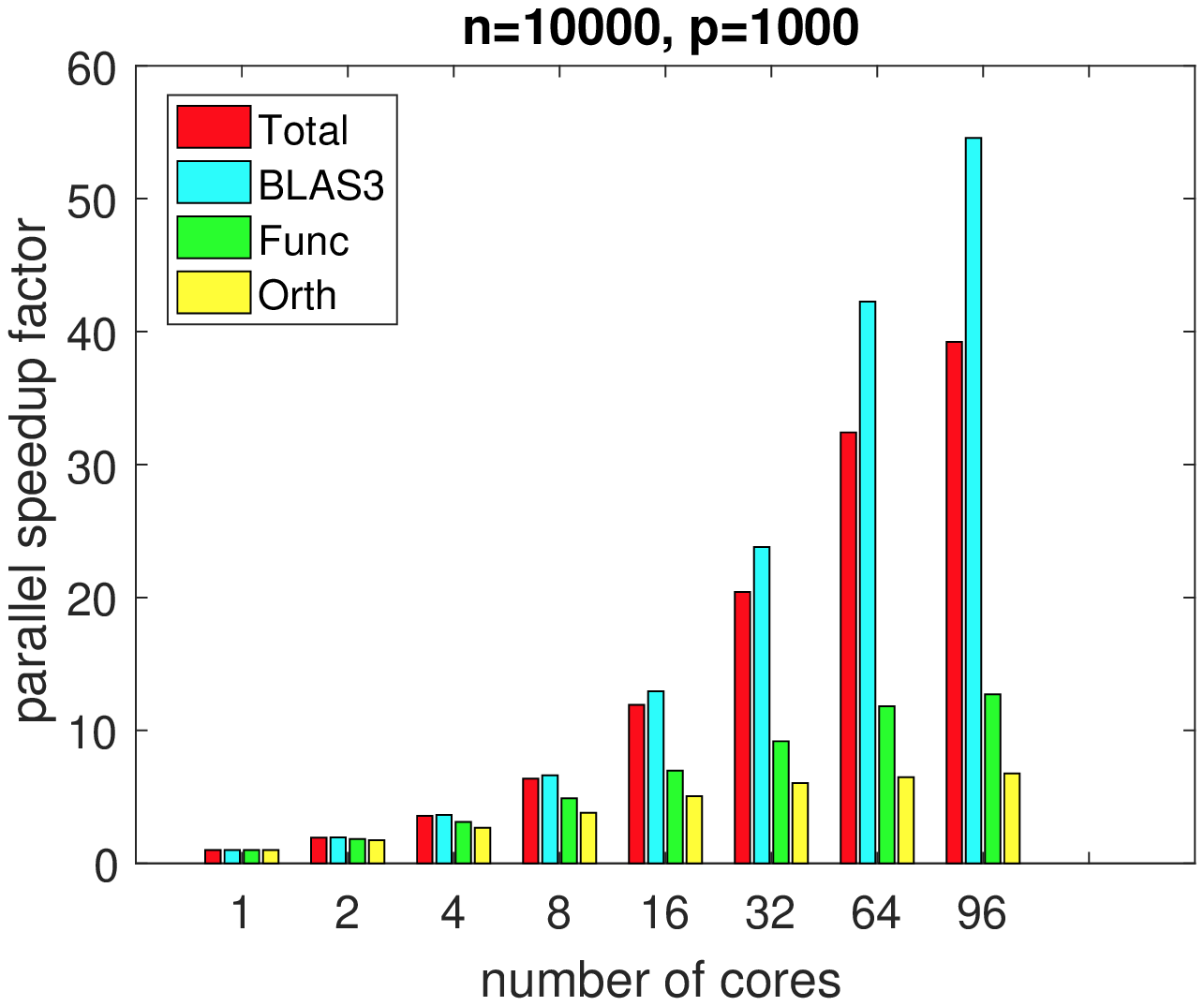}}	
	
	\caption{A comparison of speedup factor among MOptQR and PCAL on the simplified Kohn-Sham total energy minimization \label{fig:speedupks}}
\end{figure}

\section{Conclusion}

Optimization problems with orthogonality constraints have wide application in 
materials science, machine learning, image processing and so on. 
Particularly, when we apply Kohn-Sham density functional theory (KSDFT) to electronic structure 
calculation, the last step is to solve a Kohn-Sham total energy minimization with orthogonality 
constraints. There are plenty of \revise{existing} algorithms based on manifold optimization, which work
quite well when the number of columns of the matrix variable $p$ is relatively few.
With the increasing of $p$, a bottleneck of existent algorithms emerges, that is, lack of 
concurrency. The main reason leads to this bottleneck is that the orthonormalization process
has low parallel scalability. 

To solve this issue, we need to employ infeasible approaches. However, \revise{previous}
infeasible approaches including augmented Lagrangian method (ALM) is far less efficient 
than 
retraction based feasible methods. Even though the parallelization
reduces the running time of ALM more significantly than that of manifold methods, ALM
is still less efficient than manifold methods in parallel computing.
The main purpose of this paper 
is to provide practical 
efficient infeasible algorithms for optimization
problems with orthogonality constraints. Our main motivation is that the Lagrangian
multipliers have closed-form expression at any stationary points.
Hence, we use such expression to update multipliers instead \revise{of} dual ascent step, at the same time, the subproblem for the prime variables 
only takes one gradient step instead of being solved to a given tolerance.
The resultant algorithm, called PLAM, does not involve any orthonormalization.
PLAM is comparable with the existent feasible algorithms under well chosen penalty parameter $\beta$.
To avoid such restriction, we propose a modified version, PCAL, of PLAM.
The motivation of PCAL is to use normalized gradient step instead of gradient step
in updating prime variables. The numerical experiments show that PCAL works efficient,
robust and insensitive 
with penalty parameter $\beta$. Remarkably, it outperforms the existent feasible algorithms in
solving the KSDFT problems in
MATLAB platform  KSSOLV. We also run PCAL and MOptQR, an excellent representative of
retraction based optimization approach, in parallel with up to $96$ cores. 
Numerical experiments illustrate PCAL has higher scalability than MOptQR, and its
superiority becomes more and more noticeable with the increasing of $p$.

The potential of PCAL has already emerged. In the future work, we will apply our PCAL
to real KSDFT calculation.

\section*{Acknowledgements}  The authors would like to thank Michael Overton,
Tao Cui and Xingyu Gao for the insightful discussions.


\end{document}